\documentclass{amsart}
\usepackage{amsthm,amsfonts,amsmath,amscd,amssymb,latexsym,epsfig,mathrsfs}
\usepackage{appendix}
\usepackage{tikz-cd}
\usepackage{xcolor}
\usepackage[pdftex,pdfpagelabels,bookmarks,hyperindex,hyperfigures]{hyperref}
\hypersetup{
    colorlinks = true,
    linkcolor = {blue},
    citecolor={magenta},
   linkbordercolor = {white},
    linktocpage=true
}
\usepackage[non-sorted-cites]{amsrefs}
\newcommand{\lpar}{(}
\newcommand{\rpar}{)}
\newcommand{\gl}{{\mathfrak g \mathfrak l}}
\newcommand{\so}{{\mathfrak s \mathfrak o}}
\renewcommand{\u}{{\mathfrak u}}
\newcommand{\su}{{\mathfrak s  \mathfrak u}}
\newcommand{\ssl}{{\mathfrak s \mathfrak l}}
\newcommand{\ssp}{{\mathfrak s \mathfrak p}}
\newcommand{\g}{{\mathfrak g}}         
\newcommand{\h}{{\mathfrak h}}         

\newcommand{\cx}{{\mathbb C}}
\newcommand{\diag}{\operatorname{diag}}
\newcommand{\ad}{\operatorname{ad}}
\newcommand{\Ad}{\operatorname{Ad}}
\newcommand{\tr}{\operatorname{tr}}

\newcommand{\im}{\operatorname{Im}}

\newcommand{\codim}{\operatorname{codim}}
\newcommand{\Lie}{\operatorname{Lie}}
\newcommand{\Hom}{\operatorname{Hom}}
\newcommand{\End}{\operatorname{End}}

\newcommand{\Ker}{\operatorname{Ker}}

\newcommand{\rank}{\operatorname{rank}}

\newcommand{\Hilb}{\operatorname{Hilb}}

\newcommand{\Mat}{\operatorname{Mat}}

\newcommand{\Spec}{\operatorname{Spec}}
\newcommand{\Id}{\operatorname{Id}}

\newcommand{\ol}{\overline}

\numberwithin{equation}{section}

\newtheorem{theorem}{Theorem}[section]
\newtheorem*{theorem*}{Theorem}

\newtheorem{lemma}[theorem]{Lemma}

\newtheorem{corollary}[theorem]{Corollary}

\newtheorem{proposition}[theorem]{Proposition}

\theoremstyle{remark}

\newtheorem{remark}[theorem]{Remark}

\newtheorem{definition}[theorem]{Definition}

\newtheorem{example}[theorem]{Example}

\newtheorem*{ack}{Acknowledgment}

\newcommand{\C}{{\mathbb{C}}}

\newcommand{\N}{{\mathbb{N}}}

\newcommand{\R}{{\mathbb{R}}}

\newcommand{\Z}{{\mathbb{Z}}}

\newcommand{\oH}{{\mathbb{H}}}

\newcommand{\oN}{{\mathbb{N}}}

\newcommand{\oP}{{\mathbb{P}}}

\newcommand{\oR}{{\mathbb{R}}}
\newcommand{\oS}{{\mathbb{S}}}

\newcommand{\oV}{{\mathbb{V}}}

\newcommand{\oZ}{{\mathbb{Z}}}

\newcommand{\sA}{{\mathcal{A}}}   
\newcommand{\sB}{{\mathcal{B}}}
\newcommand{\sC}{{\mathcal{C}}}   

\newcommand{\sF}{{\mathcal{F}}}

\newcommand{\sL}{{\mathcal{L}}}   
\newcommand{\sM}{{\mathcal{M}}}   
\newcommand{\sN}{{\mathcal{N}}}
\newcommand{\sO}{{\mathcal{O}}}
\newcommand{\sP}{{\mathcal{P}}}
\newcommand{\sQ}{{\mathcal{Q}}}

\newcommand{\sT}{{\mathcal{T}}}

\newcommand{\sX}{{\mathcal{X}}}

\newcommand{\sZ}{{\mathcal{Z}}}

\newcommand{\fB}{{\mathfrak{b}}}

\newcommand{\fH}{{\mathfrak{h}}}

\newcommand{\fK}{{\mathfrak{k}}}

\newcommand{\fM}{{\mathfrak{m}}}

\newcommand{\fT}{{\mathfrak{t}}}
\newcommand{\fU}{{\mathfrak{u}}}

\DeclareMathOperator*{\dprime}{\prime\hspace{-0.5pt} \prime}

\begin{document}

\title{Hypertoric varieties, $W$-Hilbert schemes, and Coulomb branches}
\author{Roger Bielawski \and Lorenzo Foscolo}
\address{Institut f\"ur Differentialgeometrie,
Leibniz Universit\"at Hannover,
Welfengarten 1, 30167 Hannover, Germany}
\address{Dipartimento di Matematica,
Sapienza Universit\`a di Roma,
Piazzale Aldo Moro 5,
00185 Roma, Italy}



\begin{abstract} We study transverse equivariant Hilbert schemes of affine hypertoric varieties equipped with a symplectic action of a Weyl group.
In particular, we show that the Coulomb branches of Braverman, Finkelberg, and Nakajima can be obtained either as such Hilbert schemes or  Hamiltonian reductions thereof. Furthermore, we propose that the Coulomb branches for representations of non-cotangent type are obtained analogously.
\par
We also investigate the putative complete hyperk\"ahler metrics on these objects. We describe their twistor spaces and, for a large class of $W$-invariant  hypertoric varieties (which includes all  Coulomb branches of cotangent type), we show that the hyperk\"ahler metric can be described as the natural $L^2$-metric on a moduli space of solutions to {\em modified} Nahm's equations on an interval  with poles at both ends and a discontinuity in the middle,  with the latter described by a new object: a {\em hyperspherical variety} canonically associated to a hypertoric variety.
\end{abstract}

\maketitle

\thispagestyle{empty}

\tableofcontents

\section{Introduction}

 Affine hypertoric varieties are symplectic quotients of some $\cx^{2d}\times T^\ast(\cx^\ast)^n$ (with its standard symplectic form)
by $(\cx^\ast)^d$ which acts in the standard way on the first factor, and via a homomorphism $(\cx^\ast)^d\to (\cx^\ast)^n$ on the second one. Such a quotient $X$ is a $2n$-dimensional affine  symplectic variety equipped with a Hamiltonian algebraic action of  $T=(\cx^\ast)^n$ and   a flat moment map $\mu:X\to \h^\ast\simeq \h$, where $\h=\Lie(T)$. 
Suppose now that a Weyl group $W$ acts on $T$ by reflections. We say that $X$ is a {\em $W$-invariant } if the action of $T$ on $X$ extends to a symplectic action of the corresponding semidirect product  $T\rtimes W$. In the special case, when already the homomorphism  $(\cx^\ast)^d\to T$ is $W$-equivariant, we say that $X$ is  {\em strongly $W$-invariant}.
\par
Recall now that to any scheme $X$ (of finite type over $\cx$) equipped with a generically free algebraic action of a finite group $W$, one can associate its equivariant Hilbert scheme $W\text{-}\Hilb(X)$ consisting of $W$-invariant $0$-dimensional subschemes $Z$ of $X$ such that the induced representation of $W$ on $H^0(Z,\sO_Z)$ is the regular representation. This scheme does not have to be integral, even if $X$ is. For this and other reasons it is often better to consider the subscheme $\Hilb^W(X)$ defined as the closure of the locus of free orbits in $W\text{-}\Hilb(X)$.
\par
If $X$ is also equipped with an equivariant map $\mu:X\to \h$, with $W$ acting by reflections on $\h$, then we can consider the transverse equivariant  Hilbert schemes $W\text{-}\Hilb_\mu(X)$ and $\Hilb^W_\mu(X)$, consisting of those  $Z\in W\text{-}\Hilb(X)$ (resp.\  $Z\in \Hilb^W(X)$) for which $\mu|_Z$ is an isomorphism onto its image. These are schemes over  $\h/W$. Both of them are regular if $X$ is regular. 
\par
The subject of this paper are varieties $\Hilb_\mu^W(X)$ corresponding to affine $W$-invariant hypertoric varieties $X$. One reason for our interest in these objects is that they provide an explicit realisation of Coulomb branches of $3$-dimensional $N=4$ supersymmetric gauge theories \cite{Nak}. Let us recall that such a quantum gauge theory is associated with the choice of a compact Lie group\footnote{We shall denote this gauge group by $K^\vee$, since it is its Langlands dual which will play a more prominent role in the present paper.}  and its quaternionic representation, and  physics predicts that the moduli space of vacua is a (generally singular) hyperk\"ahler space. While the quantum gauge theory does not have a mathematical definition, we still might expect to construct these hyperk\"ahler spaces rigorously. The moduli space of vacua has in general different branches. One of these, the Higgs branch, is mathematically well-defined, since it is simply the hyperk\"ahler quotient of the quaternionic representation space by the gauge group $K^\vee$. There is, however, 
another distinguished branch, the Coulomb branch.
Despite many well-known hyperk\"ahler manifolds (the Taub-NUT and Atiyah-Hitchin metrics, for example) having been identified as such Coulomb branches, a general definition is still lacking. In the case when the quaternionic representation is of the form $V\oplus V^\ast$, where $V$ is a complex representation of $K^\vee$, Braverman, Finkelberg, and Nakajima \cite{BFN} have constructed the associated Coulomb branches as Poisson\footnote{A recent result of Bellamy \cite{Bell} shows that these varieties are symplectic.} affine varieties via the equivariant Borel-Moore homology of certain moduli stacks.  In the case of arbitrary quaternionic representations, constructions of Coulomb branches have been proposed in  \cite{BDFRT} and in \cite{Tel2}.
\par
In the present paper, we show that the Braverman-Finkelberg-Nakajima  varieties can be realised as  transverse equivariant Hilbert schemes of  strongly $W$-invariant affine hypertoric varieties (in most cases), or Poisson quotients thereof. The latter case occurs only if $G=K^\cx$ has a direct factor ${\rm SO}_{2k+1}(\cx)$. Furthermore, we propose that the Coulomb branches for quaternionic representations of non-cotangent type are obtained as transverse equivariant Hilbert schemes or Poisson quotients  of more general affine $W$-invariant hypertoric varieties. In addition, we shall show that the hyperk\"ahler structure of Coulomb branches of cotangent type should be realised via a moduli space of solutions to modified  Nahm's equations on on an interval  with poles corresponding to a regular $\su(2)$-triple in $\fK$  at both ends and  certain discontinuity in the middle.
\par
Our main complex-symplectic result is as follows.
\begin{theorem*}
\begin{enumerate}
\item 
Let $X$ be an affine strongly $W$-invariant hypertoric variety with moment map $\mu: X\to \h^\ast\simeq\h$. The transverse $W$-Hilbert scheme $\Hilb^W_\mu (X)$ is equal to $W\text{-}\Hilb_\mu(X)$ and it is an affine symplectic\footnote{We show that $\Hilb^W_\mu (X)$ is a normal affine variety with a symplectic form on its regular locus. The fact that its  singularities are symplectic follows then from Bellamy's Lemma 2.1 \cite{Bell}.} variety with a flat morphism $\overline{\mu}: \Hilb^W_\mu (X)\to \h/W$.
\item For each connected reductive Lie group $G$ with Weyl group $W$ and a representation $V$ of the Langlands dual group $G^\vee$, there exists a strongly $W$-invariant hypertoric variety $X$ such that the Bravermann-Finkelberg-Nakajima Coulomb branch $\mathcal{M}_C(G^\vee, V)$ is isomorphic, as a Poisson affine variety, to the Poisson quotient of $\Hilb^W_\mu (X)$ by $(\cx^\ast)^N$, where $N$ is the number of direct ${\rm SO}_{2k+1}(\cx)$ factors in $G$.
\end{enumerate}
\end{theorem*}
 More precisely, let $G$ be a connected reductive algebraic group and $V$ a complex representation of its Langlands dual. The maximal torus of $G$ and the weights of $V$ determine a $W$-invariant affine hypertoric variety $X(G,V)$, and we prove that its transverse $W$-Hilbert scheme $\Hilb^W_\mu(X)$ is the Coulomb branch associated to $(G,V)$, {\em except} if $G$ has direct ${\rm SO}_{2k+1}(\cx)$-factors. In this case we replace each such factor by ${\rm Spin}^c_{2k+1}(\cx)= {\rm Spin}_{2k+1}(\cx)\times_{{\scriptscriptstyle\oZ_2}}\cx^\ast$,
 form the strongly $W$-invariant affine hypertoric  variety $X(\tilde G,V)$ corresponding to this new group $\tilde G$ (and the same $V$), and form the Poisson quotient of 
 $\Hilb^W_\mu\bigl(X(\tilde G,V))$ by the centre of the product of the ${\rm Spin}^c_{2k+1}(\cx)$.
The reason for this different behaviour 
is essentially due to the presence of coroots that are divisible in the coweight lattice.
\par
The proof of the above theorem, given in \S\ref{Coul} and building on results established in the previous sections, uses two main tools. One is the idea of Bravermann, Finkelberg, and Nakajima \cite[Theorem 5.26]{BFN} that $\mathcal{M}_C(G^\vee, V)$ is uniquely determined amongst affine varieties over $\h/W$ (satisfying additional technical assumption related to  normality) by its structure on the open set lying over a certain codimension $2$ locus in $\h/W$ determined by the weights of $V$ (or the defining hyperplanes of the hypertoric variety $X$ in our construction) and the reflection hyperplanes for the $W$-action. Very roughly, $\mathcal{M}_C (G^\vee, V)$ is isomorphic to $T^\ast T/W$ ($T$ is the maximal torus of $G$) outside of these hyperplanes, while at generic points along the reflection hyperplanes it has local models $(T^\ast T' \times D_k)/\Z_2$ or $T^\ast T'\times D_k$ (the latter only when $G^\vee$ contains a direct ${\rm SO}_{2k+1}(\C)$ factor), where $T'$ is a corank $1$ subtorus of $T$ and $D_k$ is a $D_k$ surface with $k$ determined by the weights of $V$. We show that this is precisely the same structure as (a Poisson $(\C^\ast)^N$-quotient of) $\Hilb^W_\mu(X)$. 
In order to be able to apply the uniqueness result by Bravermann, Finkelberg, and Nakajima we need to establish the normality of $\Hilb^W_\mu(X)$, as well as the fact that the complement of the above open set has codimension $2$. The latter follows from the flatness of $\bar\mu$, which is  established in \S\ref{Snormal}. Once we have this, we prove the normality of $\Hilb^W_\mu(X)$ by following an idea of Bravermann, Finkelberg, and Nakajima \cite[Lemma 6.12]{BFN}.
\par
The picture, described above, of a Coulomb branch being determined by its structure over the complement of (the projection of) root hyperplanes for $G$ and hyperplanes orthogonal to weights of $V$ has been also known to physicists \cite{BDG}. Our work suggests another possibility for an axiomatic description of these Coulomb branches and of the more general spaces $\Hilb_\mu^W(X)$ we consider: a Delzant type theorem for affine symplectic varieties of dimension $2\dim \h$ admitting a Hamiltonian action of the universal centraliser in $G$ (this is an abelian group scheme over $\h/W$).
\par
We hope that the uniform geometric description of our main theorem will be useful for further study of the geometry of the Bravermann-Finkelberg-Nakajima  Coulomb branches. In fact, since we can start from an arbitrary strongly $W$-invariant hypertoric variety $X$,  our construction associates a space $\mathcal{M}_C$ to every connected reductive Lie group $G$ and every element\footnote{The existence of a larger class of spaces with properties similar to Coulomb branches has been already observed both in physics \cite{CFHM} and in mathematics \cite{NW}, but the variety of spaces and their uniform description given here are new. } of the \emph{representation ring} of the Langlands dual group $G^\vee$, not just a single representation $V$. Conversely, the same space $M$ can be realised as a Coulomb branch in many different ways: for example, in dimension two Coulomb branches are $D_k$ singularities, where $k$ only depends on the sum of dominant weights of $V$. This redundancy might complicate identifying properties of the space $M$. For instance, Bravermann, Finkelberg, and Nakajima define natural deformations (and resolutions) of a Coulomb branch arising in the presence of a non-trivial flavour symmetry group (by ``turning on masses'' for the flavour group, in physics jargon).
These deformations depend on the particular presentation of the space in question as a Coulomb branch. In our construction, deforming the hypertoric variety $X$ so that it remains strongly $W$-invariant produces natural deformations of $\Hilb^W_\mu (X)$. Based on the $2$-dimensional case, it is tempting to conjecture that, modulo the subtlety arising when $G$ has ${\rm SO}_{2k+1}(\C)$ factors, all deformations of $\Hilb^W_\mu (X)$ arise in this way. Similarly, the presymplectic quotient description of $\Hilb^W_\mu (X)$ in \S\ref{sq} could lead to a construction of natural resolutions of $\Hilb^W_\mu (X)$.
\par
The construction of the Coulomb branch described above generalises to the case when the quaternionic representation $\mathbb{V}$ of $G^\vee$ is no longer of the form $V\oplus V^\ast$. A quaternionic representation is self-dual and, hence, the nonzero weights of $\mathbb{V}$ come in pairs $\pm \alpha_1,\dots,\pm \alpha_d$. We choose one element of each pair and consider the corresponding hypertoric variety $X$ (the isomorphism type of $X$ does not depend on this choice). Unlike in the case $\mathbb{V}=V\oplus V^\ast$, $X$ does not always admit an action of $W$; in general, it is acted upon  by some extension of $W$ by the maximal torus $T$ of $G$. The vanishing of the corresponding extension class in $H^2(W;T)$ is an obstruction to the existence of the Coulomb branch corresponding to $\mathbb{V}$.  If this obstruction vanishes, i.e.\ if $X$ is $W$-invariant hypertoric and $G$ has no direct ${\rm SO}_{2k+1}(\cx)$-factors, we propose to define the Coulomb branch corresponding to $\mathbb{V}$ as in above Theorem, i.e.\ as $\Hilb^W_\mu (X)$. 
In the case when $G$ does have such factors, the construction is more complicated and explained in \S\ref{genquat} and Appendix  \ref{appendix:ht-lifts}.

\vspace{3mm}

\textbf{Towards a hyperk\"ahler structure on $\Hilb^W_\mu (X)$.}
 Our interest in $\Hilb^W_\mu (X)$ was motivated (originally without knowledge of a connection to Coulomb branches) by a construction of complete non-compact hyperk\"ahler (real) $4$-manifolds with ALF asymptotics of dihedral type (a.k.a.\ $D_k$ ALF metrics) from $\Z_2$--invariant ALF spaces of cyclic type and the simplest dihedral ALF space, the Atiyah-Hitchin metric. This construction was suggested by Sen \cite{Sen} in the physics literature and recently made rigorous by Schroers and Singer \cite{SS} (similar ideas were implemented in the compact setting by the second-named author \cite{Fos}). A natural question is whether this construction can be generalised to higher dimensions. 
Indeed, for a smooth  $W$-invariant $X$ we expect $\Hilb_\mu^W(X)$ to carry a natural complete QALF hyperk\"ahler metric (QALF stands for ``quasi asymptotically locally flat",  the precise definition of which we leave to a future work, but see \S\ref{asymptotics}). Furthermore, we expect that all complete QALF hyperk\"ahler manifolds arise as hyperk\"ahler quotients of such $\Hilb_\mu^W(X)$ by tori.
\par
In this paper we take first steps towards verifying these conjectures. First of all, it is almost automatic to define a singular model of the twistor space of $\Hilb_\mu^W(X)$: one applies the functor $\Hilb_\mu^W$  to the fibres of the singular model of the twistor space of $X$. This should give a hyperk\"ahler metric on an appropriate space of sections of the twistor space, which we would like to identify with   $\Hilb_\mu^W(X)$. Proving this identification directly is hopeless, as is proving the completeness of the metric. Our main contribution is in showing that this putative complete hyperk\"ahler metric on  $\Hilb_\mu^W(X)$ should arise as the natural metric on  a moduli space of solutions to {\em modified}\footnote{The modification arises by changing the flat hyperk\"ahler structure on the space of $\R^3$-invariant connections on $I\times \R^{3}$.} Nahm's equations. The solutions are defined on an interval, with poles at both ends and a discontinuity in the middle.  The discontinuity is described by another object we associate to a strongly $W$-invariant hypertoric variety: a {\em hyperspherical variety} $M_G(X)$. Here $G$ is a  connected reductive algebraic group canonically associated to $(T,W)$, and $M_G(X)$ is a normal affine $G\times G$-Hamiltonian variety of dimension $2\dim G$  such that $\cx[M_G(X)]^{G\times G}\simeq \cx[\h]^W$.  Moreover, $\Hilb^W_\mu(X)$ is naturally a subvariety of $M_G(X)$.  The variety $M_G(X)$ generalises $T^\ast\Mat_{n,n}(\C)$ with its natural $U(n)\times U(n)$-action. We construct a twistor space for $M_G(X)$ and show that this defines a pseudo-hyperk\"ahler manifold with a tri-Hamiltonian $K\times K$-action ($K$ is the maximal compact subgroup of $G$), which can be used to define the above moduli space of solutions to modified Nahm's equations. Thus statements about existence and completeness of a natural hyperk\"ahler metric on $\Hilb_\mu^W(X)$ are equivalent to the analogous statements for $M_G(X)$.
Unlike the metric on $\Hilb_\mu^W(X)$, the one on  $M_G(X)$ is in most cases algebraic and therefore one can hope that it admits a finite-dimensional construction as a (stratified) hyperk\"ahler manifold, e.g.\ as a hyperk\"ahler quotient of a vector space, or a hyperk\"ahler submanifold of such a quotient. Indeed, we are able to verify  that this is the case for several $X$.

\vspace{3mm}

\textbf{Plan of the paper.}
The article is organised as follows. In the next section we introduce  equivariant transverse Hilbert schemes and prove that this construction preserves two key scheme-theoretic properties:  being an affine scheme and the flatness of the natural morphism to $\cx^n/W$. The following section describes transverse equivariant Hilbert schemes as weighted affine blow-ups. In \S\ref{tori} we consider the case of the transverse Hilbert scheme of $T\times\cx^n$, where $T$ is a torus, and a Weyl group $W$ acts on both $T$ and $\cx^n$ by reflections. The transverse Hilbert scheme is then an abelian group scheme $\sT$ over $\cx^n/W$, and if $T$ acts on an affine scheme $\pi:X\to \cx^n$, then  $\sT$ acts on $\Hilb^W_\pi(X)$. We give a sufficient condition for the finite generation of the algebra of $\sT$-invariants, and, hence, for the existence of a well-defined GIT quotient $\Hilb^W_\pi(X)/\!\!/ \sT$ (Proposition \ref{fingen}).
\par
In Section \ref{htv} we discuss affine hypertoric varieties, $W$-invariant hypertoric varieties, and their transverse $W$-Hilbert schemes. The main result in this section is a group-cohomological characterisation of $W$-invariant hypertoric varieties (Corollary \ref{obstruction}). The subsequent section is devoted to the study of transverse Hilbert schemes of strongly $W$-invariant hypertoric varieties. In particular, we prove part (1) of the main theorem above (Corollaries \ref{flat2} and \ref{W=W}, and Theorem \ref{normal}). In \S\ref{sq} we show that if a strongly $W$-invariant affine hypertoric variety $X$ is a symplectic quotient of $Y=\cx^{2d}\times T\times \h$  
by $(\cx^\ast)^d$, then 
$\Hilb^W_\mu(X)$  is a presymplectic quotient of $\Hilb^W_\mu(Y)$ by the abelian group scheme $\Hilb_\mu^W\bigl((\cx^\ast)^d\times \h\bigr)$ over $\h/W$. In \S\ref{Coul} we identify Coulomb branches of Braverman, Finkelberg, and Nakajima with transverse Hilbert schemes of strongly $W$-invariant hypertoric varieties or with their Poisson quotients (Theorem \ref{C=W}). We also propose, along the same lines, a construction of Coulomb branches of non-cotangent type.
\par
In \S\ref{hk-Hilb^W} we discuss the twistor space of $\Hilb^W_\mu(X)$, the relation between its twistor lines, $W$-invariant curves in the twistor space of $X$, and Nahm's equations. We also prove that the complete hyperk\"ahler geometry on  $\Hilb^W_\mu(X)$ can be described via Nahm's equations in two basic cases: $X=T\times \h$ and $X=\cx^{2d}$. In \S\ref{hypers} we extend this description to a large class (possibly all, but in any case containing  all strongly $W$-invariant varieties) of $W$-invariant hypertoric varieties. More precisely, we show that the hyperk\"ahler geometry of $\Hilb^W_\mu(X)$ can be described via modified Nahm's equations on an interval  with matching conditions determined by a complex space $M_G^\circ(X)$. We describe the symplectic and (pseudo)-hyperk\"ahler geometry of $M_G^\circ(X)$. In \S\ref{completions} we show that, in the case of strongly $W$-invariant hypertoric varieties, $M_G^\circ(X)$ is quasiaffine. In general, we do not have an explicit description of its affine completion $M_G(X)$, but in \S\ref{minuscule} we do give such a description for those $X$ which can be diagonally embedded in $\bigoplus_{i=1}^l \End(V_i)$, where  each $V_i$ 
is a minuscule representation of $G$. This is, in particular, the case for a class of examples corresponding to singular monopoles on $\oR^3$, and in \S\ref{monopoles} we work out the explicit description of the hyperspherical variety $M_G(X)$ in this case.
\par
In  \S\ref{asymptotics} we discuss briefly the expected asymptotic behaviour of hyperk\"ahler metrics on $\Hilb^W_\mu(X)$, generalising the description  given by Bullimore, Dimofte, and Gaiotto \cite[\S4.1]{BDG} for Coulomb branches.  In Appendix \ref{appendix:B} we introduce the modified Nahm equations, while in Appendix \ref{appendix:ht-lifts} we discuss a special case of the construction of Coulomb branches in the case when $G$ has direct ${\rm SO}_{2k+1}(\C)$-factors.

\begin{remark} All schemes are of finite type over $\cx$, unless stated otherwise.
\end{remark}

\begin{ack} The work of Lorenzo Foscolo is supported by EPSRC New Horizon Grant EP/V047698/1 and a Royal Society University Research Fellowship. 
\par
Both authors thank Gwyn Bellamy, Amihay Hanany, and Constantin Teleman for helpful comments and questions. We are also grateful to Kifung Chan and Tom Gannon for pointing out errors in an earlier version of the paper.\end{ack}

\section{Equivariant transverse Hilbert schemes\label{eths}}

Let $W$ be a finite group acting  on a scheme $X$. The $W$-Hilbert scheme  $W\text{-}\Hilb(X)$ consists of $Z\in \bigl(\Hilb^{|W|}(X)\bigr)^W$ such that the induced representation of $W$ on $H^0(Z,\sO_Z)$ is the regular representation (cf.\ \cite{Bertin}; see also \cite{Blume} for a definition not requiring $\Hilb^{|W|}(X)$).
In addition, if $W$ acts generically freely, then $\Hilb^W(X)\subset W\text{-}\Hilb(X)$ is the (scheme-theoretic) closure of the subscheme of free orbits.
\par
The first of these notions has better universality properties, while the second one has better scheme-theoretic properties. In particular, if $X$ is reduced and irreducible, then so is $\Hilb^W(X)$.
\begin{example} Let $X=\cx^n$ and let $W$ be a reflection group. The ring of invariant polynomials is freely generated by certain $p_1(z),\dots,p_n(z)$.  For every $c=(c_1,\dots,c_n)$, the ideal 
\begin{equation}I_c=(p_1(z)-c_1,\dots, p_n(z)-c_n)\label{I_c}\end{equation}
defines an element of $W\text{-}\Hilb(X)$ and induces an isomorphism
$$W\text{-}\Hilb(\cx^n)\simeq \cx^n/W\simeq \cx^n.$$\end{example}
\begin{remark} The coordinate ring of the $0$-dimensional subscheme $Z$ defined by ideal \eqref{I_c}, i.e. $R_c=\cx[z_1,\dots,z_n]/I_c$, is known as the {\em (deformed) ring of coinvariants}. If $z\in Z$ belongs to a subspace $V$ invariant under a reflection subgroup $W^\prime \subset W$, then  $R_c\simeq \bigoplus_k R^\prime_0$, where $R^\prime_0$ is the (undeformed) ring of coinvariants for $W^\prime$ acting on $V$, and $k=|W|/|W^\prime|$.
\label{R_c}
\end{remark}
\begin{remark} In the setting of the last example, we denote by $\tilde\delta(z)$ the defining polynomial of the union  of hyperplanes fixed by reflections in $W$, and by $\delta$ the defining polynomial of its image in $\cx^n/W$. In other words, $\delta^{-1}(0)$ is the branch divisor  of the natural map $\cx^{n}\to \cx^{n}/W\simeq \cx^n$.\label{delta}\end{remark}
Example \ref{I_c} generalises as follows:
\begin{proposition} If the quotient morphism $p:X\to X/W$ is flat, then $W\text{-}\Hilb(X)\simeq X/W$.
\end{proposition}
\begin{proof} In this setting flatness is equivalent to $p_\ast\sO_X$ being a locally free sheaf (cf.\ \cite[Lemma 29.48.2]{stacks}). Now the proof of Theorem 4.11 in \cite{Bertin} works without changes.
\end{proof}
Let now $S$ be a $W$-scheme with $S\to S/W$ flat and let $X$ be a $W$-scheme equipped with a surjective $W$-equivariant morphism $\pi:X\to S$.
\begin{definition} The equivariant transverse Hilbert scheme  $W\text{-}\Hilb_\pi(X)$ is an open subscheme of $W\text{-}\Hilb(X)$  consisting of $Z$ such that $\pi|_Z$ is a scheme-theoretic isomorphism onto its image. If $S$ is the closure of the locus of free $W$-orbits, then $\Hilb_\pi^W(X)$ is the intersection of $W\text{-}\Hilb_\pi(X)$ and $\Hilb^W(X)$. \end{definition}
\begin{remark} There is a natural morphism $\bar\pi: W\text{-}\Hilb_\pi(X)\to S/W$, $Z\to \pi(Z)$. 
\end{remark}
\begin{remark} $W\text{-}\Hilb_\pi(X)$ parametrises pairs $(D,\phi)$, where $D\in W\text{-}\Hilb(S)=S/W$ and $\phi:D\to X$ is a $W$-equivariant section of $\pi$.
\label{pairs}\end{remark}
 \begin{remark} The functorial interpretation of $W\text{-}\Hilb(X)$ shows that $W\text{-}\Hilb_\pi(X)$ represents the  functor $\underline{W\text{-}\Hilb_\pi(X)}:\text{({\it schemes})}^\circ\to \text{({\it sets})}$ given by
 \begin{multline*}\underline{W\text{-}\Hilb_\pi(X)}(T) = \{\text{closed $W$-subschemes $Z$ of $T\times X$},\\  \text{$\pi|_Z:Z\to \pi(Z)\in \underline{W\text{-}\Hilb(S)}(T)$ is an isomorphism}\}.
 \end{multline*}
 \label{functor}
 \end{remark}
\begin{remark} Unlike the full $W$-Hilbert scheme $W\text{-}\Hilb$,  $W\text{-}\Hilb_\pi$ is functorial in the category of $W$-schemes over $S$. Indeed, if $f:X\to X^\prime$ is a $W$-equivariant morphism, covering the identity on $S$, then we obtain a natural transformation of functors $\underline{W\text{-}\Hilb_\pi(X)}\to \underline{W\text{-}\Hilb_{\pi^\prime}(X^\prime)}$ by sending a $Z$ as  in the previous remark to $Z^\prime=(\Id_T\times f)(Z)$. Since $\pi^\prime\circ(\Id_T\times f)=\pi$, $\pi^\prime|_{Z^\prime}:Z^\prime\to \pi^\prime(Z^\prime)$ is an isomorphism. In the interpretation of Remark \ref{pairs} the induced morphism $\bar f: W\text{-}\Hilb_\pi(X)\to W\text{-}\Hilb_{\pi^\prime}(X^\prime)$ sends a section $\phi:D\to X$ to $f\circ \phi:D\to X^\prime$.\label{functorial}
\end{remark}
The last two remarks imply immediately:
\begin{proposition} Let $S$ be as above and let $X\to S$ and $Y\to S$ be two $W$-equivariant surjective morphisms. Then
 $$W\text{-}\Hilb_\pi(X\times_S Y)\simeq W\text{-}\Hilb_\pi(X)\times_{S/W} W\text{-}\Hilb_\pi(Y).$$
 \qed\label{fibredpro}
\end{proposition}
\begin{proposition} If $X$ and $S$ are affine and $\cx[S]$ is free over $\cx[S]^W$, then  $W\text{-}\Hilb_\pi(X)$ and $\Hilb_\pi^W(X)$ are also affine.
\label{affine}\end{proposition}
\begin{proof} Without loss of generality we may assume that $X$ is a $W$-invariant subscheme of  $\cx^N\times S$, where $\cx^N$ is a representation of $W$ and $\pi$ is the restriction of the projection onto $S$. Since the functor $W\text{-}\Hilb$ commutes with the base change \cite{Blume}, it is enough to show that $W\text{-}\Hilb_\pi(\cx^N\times S)$ is affine. Let $u_1,\dots, u_N$ be linear coordinates on $\cx^N$. If $c\in S/W$ and $D_c$ is a $0$-dimensional subscheme of $S$ defined by the maximal ideal $\fM_c\in \C[S]^W$, then the $0$-dimensional subspaces  of $\cx^{N}\times S$ which map isomorphically onto  $D$ are cut out by ideals generated by $u_i-\phi_i$, $i=1,\dots,N$, where $\phi_i\in\C[D_c]$. A choice of a $\cx[S]^W$-basis of $\cx[S]$ yields a trivialisation of the vector bundle $E\to S/W$ with fibres $\cx[D_c]$, i.e.\ $E\simeq S/W\times R$, where $R$ is the regular representation of $W$.  It follows that
\begin{multline*} W\text{-}\Hilb_\pi\bigl(\cx^{N}\times S\bigr)\simeq \Hom_{\cx\text{-}{\rm alg}}\bigl(\cx[u_1,\dots,u_N],R)^W\times S/W\simeq \\ \simeq \Hom_\cx\bigl(\cx^N,R\bigr)^W\times S/W\simeq \cx^N\times S/W.\end{multline*}
\end{proof}
We shall be mostly interested in the case when $S=\cx^n$ and $W$ is a reflection group. The assumption of the above proposition is then satisfied, and the proof yields an explicit description of $W\text{-}\Hilb_\pi(X)$ for affine $X$. 
A general $X\to \cx^n$ is given by gluing together affine schemes, and $W\text{-}\Hilb_\pi(X)$ is obtained by induced gluing of the corresponding equivariant transverse Hilbert schemes of the affine pieces.
\begin{remark} The description of  $W\text{-}\Hilb_\pi(\cx^N\oplus\cx^n)$ given in the proof of Proposition \ref{affine} can be rephrased as follows. Let $\sP^W$ be the space of $W$-equivariant polynomial maps from $\cx^n$ to $\cx^N$. These are known as $W$-covariants of type $\cx^N$ and they form a free module over $\cx[\cx^n]^W$ of rank $N$ \cite{Che}. As remarked above (Remark \ref{pairs}), $W\text{-}\Hilb_\pi(\cx^N\oplus\cx^n)$ parametrises equivariant sections of $\pi$ restricted to $0$-dimensional subschemes $D\in W\text{-}\Hilb(\cx^n)\simeq \cx^n/W$. These sections are elements of $\sP^W$ and restricting them to the subschemes $D$ amounts to replacing $\sP^W$ with $\sP^W/\sP^W_+\simeq \cx^N$, where $\sP^W_+$ is the submodule generated by $\sP_W$ and elements of $\cx[\cx^n]^W$ with zero constant term. 
\label{covariants}
\end{remark}
\begin{proposition} If $X,S$, and $S/W$ are regular, then $W\text{-}\Hilb_\pi(X)$ is regular. Consequently,  $\Hilb_\pi^W(X)=W\text{-}\Hilb_\pi(X)$ in this case.\label{smooth}
\end{proposition}
\begin{proof} Since $S$ and $S/W$ are regular and $p:S\to S/W$ is flat, all fibres of $p$ are lci. Therefore any $0$-dimensional $Z$ in $\Hilb^{|W|}(X)$, such that $\pi|_{Z}$ is an isomorphism, is lci. If $X$ is regular, then so is the lci locus  of $\Hilb^{|W|}(X)$. Since the $W$-invariant  part of a regular scheme is regular, the claim follows.\end{proof}
If $X$ is singular, the two notions may be different:
  \begin{example} Let $X$ be the surface $\{(x,y,z)\in \cx^3; x^2+y^2+z^2=0\}$ and let $\pi(x,y,z)=z$. Consider two $\oZ_2$-actions: a) $(x,y,z)\mapsto (x,-y,-z)$, b) $(x,y,z)\mapsto (-x,-y,-z)$. According to the above description, $0$-dimensional subschemes of $X$ mapping isomorphically onto an element of $\cx/\oZ_2$ defined by $z^2-c=0$ are described by:
  $$\{(x_0,x_1,y_0,y_1)\in \cx^4; (x_0+x_1z)^2+(y_0+y_1z)^2+ z^2=0 \mod z^2-c\}.$$
  The $\oZ_2$-invariant parts are given by $x_1=y_0=0$ in  case a), and by $x_0=y_0=0$ in case b). Thus $\oZ_2\text{-}\Hilb_\pi(X)$ is, respectively, $\{(x_0,y_1,c); x_0^2+y_1^2c+c=0\}$ and  $\{(x_1,y_1,c); (x_1^2+y_1^2)c+c=0\}$. In case b) the resulting surface is reducible. On the other hand $\Hilb_\pi^{\oZ_2}(X)$ is $\cx\times \{(x_1,y_1); x_1^2+y_1^2+1=0\}$ in this case.
  \label{xyz}
  \end{example}
 We compute three more examples of $W\text{-}\Hilb_\pi(X)$.
 \begin{example}
Consider a $\oZ_2$-invariant deformation $X_\tau$ of the $A_1$-singularity:
$$\{(x,y,z); xy=z^2-\tau^2\}.$$
The $\oZ_2$-action is given by 
$(x,y,z)\mapsto (y,x,-z),$
and $\pi(x,y,z)= z$. Then  $\oZ_2\text{-}\Hilb_\pi(X_\tau)$ is given by 
$$\{(x(z),y(z),q(z)); x(z)y(z)=z^2-\tau^2\mod z^2-c\}$$
where  $x(z),y(z)$ are linear polynomials interchanged by the $\oZ_2$-action.  Writing $x(z)=a+b z$, we obtain the equation
$$(a+bz)(a-bz)=z^2-\tau^2 \mod z^2-c \iff a^2-b^2c=c-\tau^2,$$
i.e.\ a deformation of the $D_2$-singularity. More generally, if $X$ is a $\oZ_2$-invariant $A_{2k-1}$-surface given by the equation 
$xy=\sum_{i=0}^k\lambda_i z^{2i}$, then $\oZ_2\text{-}\Hilb_\pi(X)$ is a $D_{k+1}$-surface with equation $a^2-b^2c=\sum_{i=0}^k\lambda_i c^{i}$. Notice that we do not get the full family of deformations of the $D_{k+1}$-singularity, but only the subfamily admitting the $\oZ_2$-action $(a,b,c)\mapsto (-a,-b,c)$.
\label{D_2}
\end{example}
\begin{example} Let $X=\cx^\ast\times\cx^n$ with $\Sigma_n$ acting in the standard way on $\cx^n$ and as $\sigma.t=t^{s(\sigma)}$ on $\cx^\ast$. $X$ is an affine subvariety of $\cx^2\times\cx^n$ with coordinates $x,y,z_1,\dots,z_n$ given by $xy=1$. Let $q(z)=\prod_{i<j}(z_i-z_j)$, i.e.\ the polynomial on which $\Sigma_n$ acts via the sign representation.  Setting $x=a+q(z)b$, $y=a-q(z)b$ we conclude that $\Sigma_n\text{-}\Hilb_\pi(X)$ is the affine variety defined by $a^2-q(z)^2b^2=1$, i.e.\ $a^2-\delta(c)b^2=1$,  where $\delta(c)=q(z)^2$ is the polynomial defined in Remark \ref{delta}.\label{D_1} 
\end{example}
\begin{example} Let $X=\cx^n\oplus\cx^n$ with $\Sigma_n$ acting in the standard way on both summands, and $$\pi(u_1,\dots,u_n,v_1,\dots,v_n)=(u_1v_1,\dots,u_nv_n).$$
According to the above description, we have to replace the $u_i$ and the $v_i$  $\Sigma_n$-equivariantly by  polynomials in $\cx[z_1,\dots,z_n]/ \cx[z_1,\dots,z_n]^{\Sigma_n}$.  This means that
\begin{equation} u_i=\sum_{j=1}^{n}a_jz_i^{j-1},\enskip v_i=\sum_{j=1}^{n}b_jz_i^{j-1},\quad i=1,\dots,n,\label{uv}\end{equation}
for some scalars $a_j,b_j$.
We can identify $\cx^n/\Sigma_n$ with the Slodowy slice $\mathscr{S}_n$  to a regular nilpotent orbit (or, equivalently, with the set of companion matrices). If we view vectors $u$ and $v$ as diagonal matrices and conjugate them by the Vandermonde matrix, then the resulting matrices $A,B$ are of the form $A=\sum_{j=1}^n a_jS^{j-1}$, $B=\sum_{j=1}^n b_iS^{j-1}$, where $S\in \mathscr{S}_n$ has eigenvalues $z_1,\dots,z_n$.  The equations $u_iv_i=z_i$, $i=1,\dots,n$, yield $AB=S$. Thus  
$$ \Sigma_n\text{-}\Hilb_\pi(X)\simeq \bigl\{(S,A,B)\in \gl_n(\cx)^{\oplus 3};\, S\in \mathscr{S}_n,\; [A,S]=0,\; [B,S]=0,\; AB=S\bigr\}.
$$
\label{bigHilb}
\end{example}

\subsection{Localisation}
In the following we assume that $S=\cx^n$  and we work in the analytic category. 
Let $t\in\cx^n$ be a point with stabiliser $W^\prime \subset W$, and let $U$ be its neighbourhood   satisfying $wU\cap U=\emptyset$ for all $w\in W\backslash W^\prime$. Denote by $\phi^{W^\prime W}$ the induced isomorphism between the image of $U$ in $\cx^n/W$ and the image of $U$ in $\cx^n/W^\prime$.
 Denote by $p:\cx^n\to\cx^n/W$ and $\bar\pi:\Hilb_\pi^W(X)\to \cx^n/W$ the canonical projections. Since a $W$-orbit intersects $\bar\pi^{-1}(U)$ in a $W^\prime$-orbit, we conclude:
 \begin{proposition} Let the stabiliser of $t\in \cx^n$ be $W^\prime\subset W$ and let $U$ and $\phi^{W^\prime W}$ be as above. There exists a natural isomorphism $L^{W^\prime W}$ between $\bar\pi^{-1}(p(U))$ in $\Hilb_\pi^W(X)$ and $\bar\pi^{-1}(p(U))$ in $\Hilb_\pi^{W^\prime}(X)$, which covers $\phi^{W^\prime W}$. Moreover, if  the stabiliser of $t^\prime\in U$ is $W^{\dprime}\subset W^\prime$, then $L^{W^{\dprime}W}=L^{ W^{\dprime}W^\prime}\circ L^{W^\prime W}$ for an analogous smaller neighbourhood of $t^\prime$.\hfill $\Box$.\label{local}
 \end{proposition}
 \begin{remark} In the algebraic category we have an analogous isomorphism between local rings $\sO_{M,Z}$ and $\sO_{M^\prime,Z^\prime}$, where $M=\Hilb_\pi^W(X)$, $M^\prime=\Hilb_\pi^{W^\prime}(X)$, $Z=\bar\pi^{-1}(p(t))\subset M$, $Z^\prime=\bar\pi^{-1}(p(t))\subset M^\prime$. Alternatively, we 
 can replace ``neighbourhoods" with ``\'etale neighbourhoods" in the above proposition.
 \end{remark}

\section{Transverse blow-ups\label{blow-up}}

One of our technical tools in proving algebro-geometric properties of transverse Hilbert schemes are weighted transverse blow-ups which we now introduce.
\par
If $\Delta$ is a codimension one\footnote{i.e.\ every irreducible component has codimension $1$.} closed subscheme of a scheme $X$, and $Z$ a closed subscheme of $\Delta$, then we can blow up $Z$ {\em transversely} to $\Delta$, i.e.\ we blow up $Z$ and remove the strict transform of $\Delta$. If $X$ is affine with coordinate ring $\cx[X]$, 
$\Delta=V(\delta)$ with $\delta\nmid 0$, and $I_Z=(\delta,f_1,\dots,f_k)$, then the resulting scheme $\tilde X$ is also affine and can be described as the closure in $X\times \cx^k$ of the graph of
$$ \phi:X\backslash\Delta\to \cx^k, \quad \phi(x)=\bigl(f_1(x)\delta^{-1}(x), \dots, f_k(x)\delta^{-1}(x)\bigr).$$
In other words $\tilde X=\Spec \cx[X][f_1\delta^{-1},\dots,f_k\delta^{-1}]$. More generally, we can consider schemes of the form $\tilde X=\Spec \cx[X][f_1\delta^{ -m_1},\dots, f_k\delta^{ -m_k}]$, where $\delta\nmid 0$ and $m_i\in \oN$, and call such a scheme the {\em weighted affine blow-up of $X$ along $Z=V(\delta,f_1,\dots,f_k)$ transverse to $\Delta=V(\delta)$} (with weights $m_1,\dots,m_k$), or simply a transverse blow-up of $X$, whenever the data of $\delta$, $Z$, $m_1,\dots,m_k$, is understood. 
\begin{remark} One can always assume  that $\delta\nmid f_i$, $i=1,\dots,k$ by changing the subscheme $Z$ along which we blow up.
\end{remark}
The natural morphism $p:\tilde X\to X$ does not have to be surjective (e.g.\  $f_1=\dots=f_k=1$). We note the following simple sufficient condition for surjectivity:
\begin{lemma} Let $\tilde X$ be a weighted affine transverse blow-up of $X$ along $V(\delta,f_1,\dots,f_k)$ transverse to $\Delta=V(\delta)$ ($\delta\nmid 0$). Assume that $\delta\nmid f_i$, $i=1,\dots,k$, and set $I=(f_1,\dots,f_k)$. Let $x\in X$ be such that the corresponding maximal ideal  $\fM_x\subset \cx[X]$, satisfies $\delta^r\not \in I_{\fM_x}$ for any $r\in \oN$.  Then $x$ belongs to the image of  the natural morphism $\tilde X\to X$.\label{surj}\end{lemma}
\begin{proof} The ideal of $\cx[\tilde X]$ generated by the image of $\fM_x$ is a maximal ideal when $x$ is in the image of $\tilde X\to X$ or the whole $\cx[\tilde X]$ otherwise. Hence the point $x$ is not in the image of $\tilde X\to X$ if and only if $1\in \fM_x\cx[\tilde X]$, i.e.\ if there exist $y_1,\dots,y_l\in \fM_x$ and $z_1,\dots,z_l\in \cx[\tilde X]$ such that $1=y_1z_1+\dots+y_lz_l$. We can write each $z_i$ as $z_i=x_i+w_i\delta^{-k_i}$ with $x_i\in \cx[X]$, $w_i\in I$,
$\delta\nmid w_i$. Since $1\not\in \fM_x$, $k=\max_i k_i$ must be positive. We conclude that $\delta^k(1-\sum_{i=1}^lx_iy_i)\in I$. Since $\fM_x$ is maximal,  $1-\sum_{i=1}^lx_iy_i\not\in \fM_x$, and the proof is complete.
\end{proof}
The following proposition shows that affine blow-ups of a scheme $X$ transverse to $\Delta=V(\delta)$ occur when considering affine schemes isomorphic to $X$ outside of $\Delta$
\begin{proposition}  Let $p:\tilde X\to X$ be a morphism of nonempty affine schemes, and let  $\delta\in \cx[X]$, $\delta\nmid 0$. 
$\tilde X$ is an affine blow-up of $X$ transverse to $\Delta=V(\delta)$ if and only if $p^\sharp(\delta)\nmid 0$ and $\cx[X][\delta^{-1}]\simeq \cx[\tilde X][p^\sharp(\delta)^{-1}]$.  \label{tbu}  
\end{proposition}
\begin{proof} The ``only if" part is clear from the definition. We prove the ``if" part. Let us write $\tilde\delta=p^\sharp(\delta)$ and suppose that $\cx[X][\delta^{-1}]\simeq \cx[\tilde X][{\tilde\delta}^{-1}]$. Since $\delta\nmid 0$ and  $\tilde\delta\nmid 0$, $\cx[X]$,$\cx[\tilde X]$ embed into $\cx[X][\delta^{-1}],\cx[\tilde X][\tilde\delta^{-1}]$,
respectively. Moreover for every $\tilde a\in\cx[\tilde X]$ there exists unique $a\in \cx[X]$ with $\delta\nmid a$ and $m\geq 0$ such that $\tilde a=p^\sharp(a)/{\tilde\delta}^m$. We can therefore find generators $\tilde a_i,\tilde b_j$, $i=1,\dots,k$, $j=1,\dots,l$, such that $\tilde b_j=p^\sharp(b_j)$ and ${\tilde \delta}^{m_i}\tilde a_i=p^\sharp(a_i)$ for some $m_i\in \oN$, $a_i,b_j\in \cx[X]$, and $\delta\nmid a_i$ for all $i$. We obtain a surjective map $A=\cx[X][a_1\delta^{-m_1},\dots,a_k\delta^{-m_k}]\to \cx[\tilde X]$. The kernel of this map must be contained in the kernel of the composition $A\to \cx[X][\delta^{-1}]\to \cx[\tilde X][{\tilde\delta}^{-1}]$. Since the first map is injective and the second one an isomorphism, the proof is complete.
\end{proof}
In particular:
\begin{proposition}
Let $X$ be a reduced affine scheme and suppose that $\delta\in \cx[X]$ is an element such that $\delta\nmid 0$ and $\cx [X][\delta^{-1}]$ is normal. Then the normalisation $X^\nu$ of $X$ is a blow-up of $X$ transverse to $\Delta =V(\delta)$.\label{normal-bu}
\end{proposition}
\begin{proof}
Under our assumptions on $X$, the coordinate ring of the normalisation $X^\nu$ is the integral closure of $\cx[X]$ in the total ring of fractions of $\cx[X]$. 
The image of $\delta$ in $\cx[X^\nu]$ cannot divide $0$ since $\delta$ is invertible in the total ring of fractions of $\cx[X]$. 
The normality of $\C[X][\delta^{-1}]$ implies that $\cx[X]\subset\cx[X^\nu]\subset\cx[X][\delta^{-1}]$ and therefore we are in the situation of Proposition \ref{tbu}.
\end{proof}
\subsection{Transverse $W$-Hilbert schemes and transverse blow-ups\label{extended}}
Let us return to the setting of the previous section, i.e.\  $X$ is equipped with a surjective $W$-equivariant morphism $\pi:X\to S$, where $S$ is an affine $W$-scheme such that $W\text{-}\Hilb(S)\simeq S/W$. We assume that $S$ is the closure of its open subscheme $S^{\rm free}$ consisting of free orbits, and that the complement of $S^{\rm free}$ is a codimension $1$ subscheme cut out by $\delta\in \cx[S]^W$ (this is the case for $S=\cx^n$). We set $\Delta=V(\pi^\sharp(\delta))\subset X/W$.
\begin{proposition} If $X$ is integral, then $\Hilb_\pi^W(X)$  is an affine blow-up of $X/W$, transverse to $\Delta$.\label{affineblowup}
\end{proposition}
\begin{proof} By definition, $X/W\backslash V(\pi^\sharp(\delta))$ is the set of $W$-orbits which map to free orbits in $Y$. Therefore $\cx[X/W][\pi^\sharp(\delta)^{-1}]\simeq \cx[\Hilb_\pi^W(X)][\bar\pi^\sharp(\delta))^{-1}]$. Since $X$ is integral, so is $\Hilb_\pi^W(X)$, and, consequently, $\bar\pi^\sharp(\delta)\nmid 0$. The statement follows now from Proposition \ref{tbu}.\end{proof} 
\begin{remark} In the case $S=\cx^n$ and $W$ a reflection group we can describe this transverse blow up more precisely.
Let $X\subset \cx^N\oplus \cx^n$ with coordinates  $u_1,\dots,u_N$ on $\cx^N$, and $W$ acting linearly on $\cx^N$. Then $ W\text{-}\Hilb_\pi(X)$  is a subscheme of $\cx^N$ given by a $W$-equivariant substitution $u_i=\phi_i(z)$, $\phi_i$ polynomials representing coinvariants. 
If we decompose $\cx^N$ into irreducibles, then on each $k$-dimensional irreducible $V$ we can write $(u_1,\dots,u_k)^T=M_V(z)(a_1,\dots,a_k)^T$, where $M_V(z)$ is a matrix of polynomials representing the $k$ copies of $V$ in $R_0=\cx[z_1,\dots,z_n]/\cx[z_1,\dots,z_n]^W$. This  means   $ \det M_V(z)(a_1,\dots,a_k)^T= M_V(z)_{\rm adj}(u_1,\dots,u_k)^T$, and, since $ \det M_V(z)$ represents $\Lambda^k V$, it is of the form $q^m$ where either $q$ (if $\Lambda^k V$ is the trivial representation) or $q^2$ (if $\Lambda^k V$ is the sign representation) is a factor of $\delta$. Dividing both sides by $q^m$ represents the $a_i$ as elements of $\cx[X]^W[\delta^{-1}]$. In particular, the ideal of the subscheme $Z$ in which we blow up is the image in $\cx[X]^W$ of the ideal in $ \cx[\C^N\oplus\cx^n]^W$ generated by $(V^\ast\otimes\cx[z_1,\dots,z_n])^W$, where $V$ is the direct sum of all nontrivial irreducibles in $\cx^N$.\label{M_V}\end{remark}
\begin{example} Let $X=\cx^\ast\times \cx$ with the $\oZ_2$-action $(t,z)\mapsto (t^{-1},-z)$. $X/\oZ_2$ is the $D_2$-singularity $a^2-b^2c+c=0$ ($a=(t-t^{-1})z/2$, $b=t+t^{-1}$, $c=z^2$). We know from Example \ref{D_1} that 
$\Hilb_\pi^{\oZ_2}(X)$, where $\pi(t,z)=z$, is the $D_1$ surface $x^2-y^2c=1$. It is obtained from 
$X/\oZ_2$ by blowing up the subscheme $(c,a)$ transversely to $(c)$ ($y=a/c$, $x=b$).\label{D1D2}\end{example}
\begin{example} The last example generalises to $\Hilb_\pi^{\oZ_2}(X)$ for any $\oZ_2$-invariant $A_{2k-1}$-surface $X$, $k\geq 1$. If  $X$ is given by the equation $xy=\prod_{i=1}^k(\tau_i^2-z^2)$, $\tau_i\in \C$, then $\Hilb_\pi^{\oZ_2}(X)$ is a $D_{k+1}$-surface with equation
 $a^2-b^2c=\prod_{i=1}^k(\tau_i^2-c)$. Substituting $a+\prod_{i=1}^k\tau_i=cw$ yields
 $$ c\bigl(b^2-cw^2+2w\prod_{i=1}^k\tau_i+p(c)\bigr)=0,$$
 where $p(c)$ is a polynomial of degree $k-1$. Hence the affine blow-up of $\Hilb_\pi^{\oZ_2}(X)$ along $(c,a+\prod_{i=1}^k\tau_i)$ transverse to $(c)$ is a $D_k$-surface with equation $b^2-cw^2+2w\prod_{i=1}^k\tau_i+p(c)=0$. We remark that the $D_0$-surface $u^2-v^2c=v$ can also be obtained as a transverse blow-up  of the $D_1$-surface $x^2-y^2c=1$.  The substitution $x-1=2cv$ (and $y=2u$) does the trick.
\label{D_k-D_{k-1}}
\end{example}
\begin{remark} Although $\Hilb_\pi^W(X)$ is the correct object for our purposes, it does have a drawback from a purely geometric point of view: the induced morphism $\bar\pi$ need not be surjective. This is, for example, the case when $\pi:X\to \cx^n$ and $W$ acts freely on $X$. We can remedy this by using the interpretation of $\Hilb^W(X)$ given in \cite[\S4.1.5]{Bertin}: it is a scheme which flatifies the sheaf $p_\ast\sO_X$ over $X/W$ ($p:X\to X/W$ is the natural projection). 
The sheaf $p_\ast\sO_X$ decomposes as 
$$\bigoplus_{\chi\in {\rm Irrep}(W)} \sF_\chi\otimes V_\chi,\enskip \text{where}\enskip \sF_\chi={\mathcal Hom}_{\sO_X}(V_\chi\otimes\sO_X,p_\ast\sO_X)^W.$$
Bertin (op.\ cit.) shows that $\Hilb^W(X)$ is isomorphic to the fibre product over $X/W$ of $G_\chi$, where $G_\chi=X/W$ if $\chi$ is a trivial representation and a Grassmann blow-up which makes $\sF_\chi/{ Tors}(\sF_\chi)$ locally free otherwise.
\par
Let now $\pi:X\to S$ be a surjective $W$-equivariant morphism with $S\to S/W$ flat. Let $\Delta\subset X$ be the preimage of $S\backslash S^{\rm free}$. We can now modify the definition of the  transverse $W$-Hilbert scheme as follows: $\ol{\Hilb}^W_\pi(X)$ is  the fibre product over $X/W$ of $(G_\chi)_\pi$, where $(G_\chi)_\pi=G_\chi=X/W$ if $\chi=1$ or the sheaf $\sF_\chi/{ Tors}(\sF_\chi)$ is locally free, and to the complement of the strict transform of $\Delta$ in $G_\chi$ otherwise.
\par
We can describe $\ol{\Hilb}^W_\pi(X)$ more explicitly in the case when $X$ is affine and $S=\cx^n$. Choose a $W$-invariant set $\{u_1,\dots,u_N,z_1,\dots,z_n\}$ of generators of $\C[X]$, where $z_1,\dots,z_n$ are coordinates on $\C^n$. According to Remark \ref{M_V},
$\Hilb_\pi^W(X)$ is a transverse blow-up of $X/W$ along a subscheme cut out by   $(V^\ast\otimes\cx[z_1,\dots,z_n])^W$, where $V$ is the direct sum of all nontrivial irreducibles in $\cx^N$. $\ol{\Hilb}_\pi^W(X)$ is then a transverse blow-up of $X/W$ along a smaller subscheme cut out by  $(V_{\rm nf}^\ast\otimes\cx[z_1,\dots,z_n])^W$, where $V_{\rm nf}$ is the direct sum of all nontrivial irreducibles in $\cx^N$ such that the corresponding  isotypic component $\sF_\chi/{ Tors}(\sF_\chi)$ is not locally free.
 \end{remark}

\section{Torus actions\label{tori}}

In this section $Y=\cx^n$ and $W$ acts on it by  reflections.
Let $T\simeq (\cx^\ast)^m$ be an algebraic torus equipped with an action of $W$ given by a homomorphism $\psi:W\to {\rm GL}_m(\oZ)$. 
According to the description in \S\ref{eths}, $0$-dimensional subschemes (not necessarily $W$-invariant) of $ T\times \cx^n$ mapping isomorphically onto $Z_c$ defined by \eqref{I_c} correspond to elements of $\bigl(R_c^\ast)^m$, where 
$R_c^\ast$ denotes the group of invertible elements of the ring $R_c$. The induced action of $W$ on $\bigl(R_c^\ast)^m$ is the composition of the same homomorphism $\psi$ and the action of $W$ on $R_c$. Thus 
$W\text{-}\Hilb_\pi(T\times \cx^n)$ is isomorphic to the $W$-invariant part of the fibre bundle over $\cx^n/W$ with fibre $\bigl(R_c^\ast)^m$ over $c\in \cx^n/W$.  In particular, $W\text{-}\Hilb_\pi(T\times \cx^n)$  is
an abelian group scheme over $\cx^n/W$. We shall denote it by  ${\sT}_\psi$. Observe that the fibre of  ${\sT}_\psi$ over a free $W$-orbit $c$  is isomorphic to $T$. For other $c$ the fibre is isomorphic to a semidirect product of a torus and a unipotent abelian group.
\begin{example} Let ${\sT}_\psi$ be the abelian group scheme corresponding to $T\times \cx^n$ as above, and let $\phi$ be a $W$-invariant element of the dual $\fT^\ast$ of the Lie algebra of $T$. Then $\phi$
defines canonically an element of $\Lie({\sT}_\psi)^\ast$ via $ u(z)\mapsto \phi\bigl(u(z)\bigr)$, where $u(z)\in \Hom_\cx(\cx[\fT],R_c)^W$. In particular, if $T$ is the maximal torus of ${\rm GL}_n(\cx)$ with the standard $\Sigma_n$-action, and $\phi$ is given by a central element $\lambda {\rm I}\in \fT$, i.e.\ 
$\phi(h)=\lambda\tr(h)$, then the induced element of $\Lie({\sT}_\psi)^\ast$ can be calculated as in Example \ref{bigHilb}:
$$ (a_1,\dots,a_n)\mapsto \lambda\tr \diag\bigl(\sum_{j=1}^{n}a_jz_1^{j-1},\dots, \sum_{j=1}^{n}a_jz_n^{j-1}\bigr)=\sum_{j=1}^na_j\lambda\tr S^{j-1}.$$
In other words a central element $\lambda{\rm I}$ of $\fT$ defines an element of $\Lie({\sT}_\psi)^\ast$ given, on the canonical basis ${\rm I},S,\dots,S^{n-1}$ of $\Lie{\sT}_\psi$, by $S^m\to\lambda\tr S^m$ .
\label{central}
\end{example}
\subsection{The case $T\times\h$\label{sTG}}
 Let $T$ be an algebraic torus with $\Lie(T)=\h$ and suppose that the action of $W$ on $\h$ arises from a linear action on $T$. 
 For any reflection $w\in W$, the fixed point set $T^w$ is a codimension $1$ subgroup of $T$ and we denote by $\chi_w$ a generator of $\{\chi\in \sX^\ast(T); \chi(T^w)=\{1\}\}$. 
Let $h_w\in \h$ satisfy  $\chi_w(h_w)=2$ and $ w.x= x-x(h_w)\chi_w$  for all $x\in \h^\ast$. Then $(\sX^\ast(T),\{\chi_w\}, \sX_\ast(T),\{h_w\})$ is a root datum which determines \cite[\S\S 8-10]{Spr}, up to isomorphism, a connected complex reductive Lie group $G_{{\scriptscriptstyle T,W}} $ which has $T$ as a maximal torus. The following example shows that if we start with a maximal torus $T$ of a connected semisimple Lie group $G$, then $G_{{\scriptscriptstyle T,W}}$ is not necessarily equal to $G$. 
\begin{example} Let $T=\cx^\ast$ be a $1$-dimensional torus with the action of $W=\Z_2$ given by $t\mapsto t^{-1}$. Then $T^w=\{\pm 1\}$ and $\chi_w(t)=t^2$. Therefore $\chi_w$ is not primitive in $\sX^\ast(T)$ and $G_{{\scriptscriptstyle T,W}}\simeq {\rm SL}_2(\cx)$. Thus $G\mapsto (T,W)\mapsto G_{{\scriptscriptstyle T,W}}$ does not recover $G={\rm P{\rm SL}}_2(\cx)$. 
\end{example}
In general, $G_{{\scriptscriptstyle T,W}}$ is not equal to the original $G$ if $\chi_w$ is a nontrivial multiple of the root $\theta_w$ of $G$ corresponding to $w$. Equivalently, this means that the corresponding coroot $v_{\theta_w}$ of $G$ is divisible in the cocharacter lattice $\sX_\ast(T)$.
The following  
result of Jackowski, McClure, and Oliver \cite[Prop. 3.2]{JMO} determines when this is the case:
\begin{proposition} Let $K$ be a compact connected Lie group, $T$ its maximal torus, and $R\subset X^\ast(T)$ its set of roots.
Assume that $\h=Lie(T)$ has been given a $W$-invariant inner product, and denote,  for any $\theta\in R$,  by $v_\theta\in \sX_\ast(T)$ the corresponding coroot, i.e.\ $v_\theta\perp\Ker\theta$ and $\theta(v_\theta)=2$. Then $\sX_\ast(T)\cap \Ker(\theta)^\perp$ is equal to $\oZ\cdot v_\theta$ or $\oZ\cdot\frac{1}{2} v_\theta$, where the second possibility occurs if and only if $G$ contains a direct factor $SO(2n+1)$ for some $n\geq 1$, and $\theta$ is a short root of such a factor.\label{JMO} 
\end{proposition}
We conclude, therefore, that the construction $G\mapsto (T,W)\mapsto G_{{\scriptscriptstyle T,W}}$
replaces a short root $\theta$ in any ${\rm SO}_{2n+1}(\cx)$-factor with $2\theta$ (and $v_\theta$ with $\frac{1}{2} v_\theta$). In other words, $G_{{\scriptscriptstyle T,W}}$  is obtained from $G$ by replacing any direct factor ${\rm SO}_{2n+1}(\cx)$ with ${\rm Sp}_{2n}(\cx)$. 
\par
We now consider the corresponding abelian group scheme $\sT_\psi$ over $\h/W$, which we shall denote by $\sT_{{\scriptscriptstyle T,W}}$. In the case $T=(\cx^{\ast})^n$ with the standard action of $\Sigma_n$,  the same argument as in Example \ref{bigHilb} shows that $\sT_{{\scriptscriptstyle T,W}}$ is isomorphic to the universal centraliser $\{(g,S)\in {\rm GL}_n(\cx)\times \mathscr{S}_n;\, gSg^{-1}=S\}$. In general we have:
\begin{theorem} $\sT_{{\scriptscriptstyle T,W}}$ is isomorphic to the universal centraliser 
 \begin{equation} {\sZ}_{G}=\{(g,S)\in G\times \mathscr{S}_{\g};\, \ad_g(S)=S\},\label{ucentr}\end{equation}
where $G=G_{{\scriptscriptstyle T,W}}$ and $\mathscr{S}_{\g}\subset \g$ is the Slodowy slice to a regular nilpotent orbit.\label{Slodowy}
\end{theorem}
\begin{proof} It is enough to show that the two varieties are analytically isomorphic. Let $\bar\pi_1$ (resp.\ $\bar\pi_2$) denote the natural submersion $\sT_{{\scriptscriptstyle T,W}}\to \h/W$  (resp.\ ${\sZ}_{G}\to \h/W$). Let $F\subset \h/W$ be the zero set of the resultant of $\delta$, i.e.\ the image of the locus of intersections of two or more reflection hyperplanes in $\h$. Since  both $\sT_{{\scriptscriptstyle T,W}}$ and $ {\sZ}_{G}$ are regular and affine and
 the complement of $\bar\pi_i^{-1}(F)$, $i=1,2$,  has codimension $2$ in both manifolds,
 Hartogs' theorem implies that it is enough to show that there is an isomorphism between the complement of $\bar\pi_1^{-1}(F)$ and the complement of $\bar\pi_1^{-1}(F)$. The complements of $\pi_i^{-1}((\delta))$, $i=1,2$, are naturally isomorphic, since they are both naturally isomorphic to $ \bigl(T\times(\h\backslash (\delta))\bigr)/W$. Proposition \ref{local} implies that we only need to show that this isomorphism extends to a global isomorphism in the case $W=\oZ_2$. Then  $G=G_{T,\oZ_2}$ has semisimple rank $1$, and hence it is isomorphic to ${\rm SL}_2(\cx)\times (\cx^\ast)^{r-1}$ or to ${\rm GL}_2(\cx)\times (\cx^\ast)^{r-2}$ (\cite[Thm. 20.33]{Milne}; the case ${\rm PSL}_2(\cx)\times (\cx^\ast)^{r-1}$ is excluded as explained above). The isomorphism in the first case follows from  Example \ref{D_1}, and in the second case from the remark just before the statement of the theorem.
\end{proof}

\subsection{Quotients\label{QQuotients}}
Let now ${\sT}_\psi$ be as at the beginning of the section, i.e.\  equal to $W\text{-}\Hilb_\pi(T\times \cx^n)$, where a linear action of $W$ on $T$ is given by a homomorphism $\psi$. Let $X$ be an affine scheme equipped with an action of  $T\rtimes_\psi W$ and a $W$-equivariant homomorphism $\pi:X\to \cx^n$. The action of $T$ induces a homomorphism  
 (cf.\ Remark \ref{functorial}) $W\text{-}\Hilb_\pi(X\times T)\to W\text{-}\Hilb_\pi(X)$. On the other hand (Proposition \ref{fibredpro})
  $W\text{-}\Hilb_\pi(X\times T)\simeq W\text{-}\Hilb_\pi(X)\times_{\cx^n/W} {\sT}_\psi$. We thus obtain an action of ${\sT}_\psi$ on $W\text{-}\Hilb_\pi(X)$, and hence on $\Hilb_\pi^W(X)$. 
If the ring  $\cx\bigl[\Hilb_\pi^W(X)\bigr]^{{\sT}_\psi}$ is finitely generated, we can form the affine scheme $\Hilb_\pi^W(X)/\!\!/ {\sT}_\psi={\rm Spec}\;\cx\bigl[\Hilb_\pi^W(X)\bigr]^{{\sT}_\psi}$. The next proposition shows that this scheme  is always  a transverse blow-up of $\Hilb_\pi^W(X/\!\!/T)$ in the sense of \S\ref{blow-up}. 
\begin{proposition}  There is a natural injective homomorphism
\begin{equation} \cx\bigl[\Hilb_\pi^W(X/\!\!/T)\bigr]\longrightarrow \cx\bigl[\Hilb_\pi^W(X)\bigr]^{{\sT}_\psi}, \label{embedding}\end{equation}
and the two rings are isomorphic away from $\Delta=(\delta)$ (i.e.\ an isomorphism of rings localised away from $\Delta$).
\label{above}\end{proposition}
\begin{proof} 
Let $g_1,\dots,g_k$ be generators of $\cx[X]^{T}$.  $\Hilb_\pi^W(X/\!\!/T)$ is a subscheme of $\cx^k\oplus\cx^n$ with the relations between the $g_i$ and coordinates $b_1,\dots,b_k$ on $\cx^k$ given by $g_i=\sum b_j\lambda_{ij}(z)$ for some polynomials $\lambda_{ij}(z)$. Similarly the relations between the coordinates $u_1,\dots,u_N$ on $\cx^N$ and $a_1,\dots,a_N$ on $W\text{-}\Hilb_\pi(\cx^{N+n})$ is $u_i= \sum a_j\phi_{ij}(z)$ for some polynomials $\phi_{ij}(z)$. Each $g_i$ is a $T$-invariant polynomial in the $u_i$ and $z_i$. Substituting for $u_i$, we conclude that each $b_j$ is a polynomial in the $a_i$ and $c_i$ (coordinates on $\cx^n/W$). Since the $g_i$ are $T$-invariant, the $b_j$ must be ${\sT}_\psi$-invariant. Thus we obtain a homomorphism
$$ \cx\bigl[W\text{-}\Hilb_\pi(\cx^{N+n}/\!\!/T)\bigr]\longrightarrow \cx\bigl[W\text{-}\Hilb_\pi(\cx^{N+n}\bigr]
^{{\sT}_\psi} .$$
Any $f$ in the ideal of $\Hilb_\pi^W(X/\!\!/T)$ maps to a ${\sT}_\psi$-invariant element of the ideal of $\Hilb_\pi^W(X)$. We thus have a homomorphism \eqref{embedding}.
It is an isomorphism away from $\Delta=(\delta)$ since $W$ acts freely there and ${\sT}_\psi$ is simply $(\cx^n/W)\backslash \Delta)\times T$.  Since $\Hilb_\pi^W(X/\!\!/T)$ is the closure of the locus of free $W$-orbits, this homomorphism is injective. 
\end{proof}
We shall give  a simple sufficient condition for the homomorphism \eqref{embedding} to be an isomorphism.
First of all, we have the following lemma:
\begin{lemma} Let $X$ be integral and $H=V(h)$ be a $W$-invariant hypersurface with $(h,\tilde\delta)=1$. Then $\cx\bigl[\Hilb_\pi^W(X)\bigr][h^{-1}]\simeq \cx \bigl[\Hilb_\pi^W(X\backslash H)\bigr]$.\end{lemma}
\begin{proof} 
This follows from Proposition \ref{affineblowup}, since the transverse blow-up of a basic open subset $U(h)$ along $Z\cap U(h)$ is the corresponding basic open  subset $U(h)$ of the transverse blow-up.
\end{proof}
\begin{proposition}  Let $X$ be integral and suppose that there exists a $T\rtimes_\psi W$-invariant  hypersurface $H=V(h)\subset X$ such that:
\begin{itemize}
\item[(i)] $X\backslash H$  is a trivial principal $T$-bundle, i.e.\ there exists  an affine scheme $B$ and a $T\rtimes_\psi W$-equivariant isomorphism $X\backslash\pi^{-1}(H)\simeq B\times T$;
\item[(ii)] either $(\delta,h)$ or $(h,\delta)$ is a regular sequence in $\cx\bigl[\Hilb_\pi^W(X/\!\!/T)\bigr]$.
\end{itemize}
Then $\cx\bigl[\Hilb_\pi^W(X)\bigr]^{{\sT}_\psi}\simeq \cx\bigl[\Hilb_\pi^W(X/\!\!/T)\bigr]$.\label{fingen}
\end{proposition}
\begin{remark} The result remain true if $X\backslash H\to B$ is an \'etale principal bundle.\label{epb}
\end{remark}
\begin{proof} Let $R=\cx\bigl[\Hilb_\pi^W(X/\!\!/T)\bigr]$ and $S=\cx\bigl[\Hilb_\pi^W(X)\bigr]^{{\sT}_\psi}$.
According to Proposition \ref{embedding}, $R$ is a subring of $S$. Therefore $R[h^{-1}]$ is a subring of $S[h^{-1}]$. Owing to the above lemma, $R[h^{-1}]\simeq \cx\bigl[\Hilb^W_\pi(B)\bigr]$ and $\cx\bigl[\Hilb_\pi^W(X)\bigr][h^{-1}]\simeq \cx\bigl[\Hilb_\pi^W(X\backslash V(h))\bigr]$. The latter isomorphism induces  
 an injective homomorphism
$$ S[h^{-1}]=\cx\bigl[\Hilb_\pi^W(X)\bigr]^{{\sT}_\psi}[h^{-1}]\to \cx\bigl[\Hilb_\pi^W(X\backslash V(h))\bigr]^{{\sT}_\psi}.$$
Since
$$\cx\bigl[\Hilb_\pi^W(X\backslash V(h))\bigr]^{{\sT}_\psi}\simeq \cx\bigl[\Hilb^W_\pi(B)\times_{\h/W} {\sT}_\psi\bigr]^{{\sT}_\psi} \simeq \Hilb^W_\pi(B)\simeq R[h^{-1}],
$$
we obtain an inverse of the embedding $R[h^{-1}]\to S[ h^{-1}]$. Therefore $R[ h^{-1}]\simeq S[ h^{-1}]$.
Hence, if  $s\in S$, then $ s=r_1 h^{-k}$  for some $r_1\in R$, $k\in \oN$,  $h\nmid r_1$. Similarly,  Proposition \ref{above} implies that $s=r_2\delta^{-l}$ for some $r_2\in R$, $l\in \oN$,  $\delta\nmid r_2$.
Then $h^kr_2=\delta^l r_1$, and assumption (ii) implies that either $k=0$ or $l=0$, i.e.\  $s\in R$.
\end{proof}
\begin{corollary} $\cx[\sT_{\psi}\times_{\cx^n/W} \sT_{\psi}]^{\sT_{\psi}}\simeq \cx[\sT_\psi]$, where $\sT_\psi$ acts anti-diagonally  on the fibred product $\sT_{\psi}\times_{\C^n/W} \sT_\psi$.\label{sTsquared}
\end{corollary}
\begin{proof} Since $\sT_\psi\times_{\C^n/W} \sT_\psi\simeq \Hilb_\pi^W(T\times T\times\C^n)$, the claim follows from the last proposition  (with $h=1$) and Remark \ref{epb}.
\end{proof}
\subsection{Exact sequences\label{exact}}
An exact and $W$-equivariant sequence $1\to T_1\to T_2\to T_3\to 1$ of tori induces an exact sequence $1\to \sT_1\to \sT_2\to \sT_3$ of abelian group schemes $\sT_i=W\text{-}\Hilb_\pi(T_i\times \cx^n)$ over $\cx^n/W$. The homomorphism $h:\sT_2\to \sT_3$ does not have to be surjective, since the fibres of $ \sT_3$ do not have to be connected. The scheme-theoretic image of $h$ is $\sT_3$, but the set-theoretic image  is only an open subscheme of $\sT_3$. As the following example shows, neither of them has to be isomorphic to the (GIT) quotient of $\sT_2$ by $\sT_1$. 
\begin{example} Consider the standard $\oZ_2$-equivariant exact sequence 
$$1\to \cx^\ast\to \cx^\ast\times \cx^\ast\to \cx^\ast\to 1,$$
where the first map is $t\mapsto (t,t)$, and the second one is $(t,s)\mapsto ts^{-1}$.
The corresponding group schemes over $\cx/\oZ_2\simeq \cx$ are (cf.\ Example \ref{sTG}) $\sT_1= \cx^\ast \times \cx$,
$\sT_2=\{(g,S)\in {\rm GL}_2(\cx)\times \mathscr{S}_{\ssl_2(\cx)};\, gSg^{-1}=S\}$, and 
$\sT_3=\{(g,S)\in {\rm SL}_2(\cx)\times \mathscr{S}_{\ssl_2(\cx)};\, gSg^{-1}=S\}$.
The induced morphisms $\sT_1\to \sT_2$ and $\sT_2\to \sT_3$ are given by the diagonal embedding and by $g\mapsto (\det g)^{-1}g^2$.
We can write an element commuting with $S$ as $\gamma_0+\gamma_1 S$, $\gamma_0,\gamma_1\in \cx$,
and  then $\sT_2$ becomes the hypersurface $(\gamma_0^2-c\gamma_1^2)e=1$ in $\cx^4$, while $\sT_3$ is the $D_1$-surface $\gamma_0^2-c\gamma_1^2=1$. The image of $\sT_2\to \sT_3$ is the complement of $c=0$, $\gamma_0=-1$. We now compute $\sT_2/\!\!/\sT_1=\Spec\cx[\sT_2]^{\sT_1}$. The $\sT_1$-invariant polynomials on $\sT_2$ are $c, w=\gamma_0^2e,y=\gamma_1^2e,x=\gamma_0\gamma_1 e$. They satisfy equations $w=1+yc$, $wy=x^2$, and hence $\sT_2/\!\!/\sT_1$ is the $D_0$-surface $x^2-y^2c=y$.\label{D1-D0}
\end{example}
\begin{remark} The definitions and results in this section generalise easily to the case of extensions of tori by finite abelian groups.  We remark that if $\Gamma$ is a finite abelian group with an action of $W$ commuting with the action of $\Gamma$ on itself, then the fibre of the corresponding abelian group scheme
$W\text{-}\Hilb_\pi(\Gamma\times \cx^n)$ over a point $x\in \cx^n/W$, which is an image of $z\in \cx^n$, is $\Gamma^{W_1}$ where $W_1={\rm Stab}_W(z)$.\label{finite}
\end{remark}

\section{Hypertoric varieties\label{htv}}
 We denote by $T^d\simeq (\cx^\ast)^d$ the standard torus with the cocharacter lattice $\sX_\ast(T^d)=\oZ^d\subset \cx^d$.
 Let $T$ be an algebraic torus with Lie algebra $\h$, and let $\alpha:T^d\to T$ be a homomorphism determined by a $\oZ$-linear map $\alpha: \oZ^d\to \sX_\ast(T)$, $\alpha(e_k)=\alpha_k$.
We consider symplectic quotients of $T^\ast \cx^{d}\times T^\ast T\simeq \cx^{2d}\times T\times \h^\ast$ by the torus  $T^d$, which acts on $\cx^{2d}$ in the standard way and on $T$ via the homomorphism $\alpha$. We refer to such quotients as {\em (affine) hypertoric varieties}.
If $u_k,v_k$, $k=1,\dots,d$ are coordinates on $\cx^{2d}$, so that the symplectic form is $\sum_{k=1}^d du_k\wedge dv_k$, then the moment map equations are  
\begin{equation} u_kv_k+\lambda_k=\langle \alpha_k,z\rangle,\quad k=1,\dots,d,\label{moment}\end{equation}
for some scalars $\lambda_1,\dots,\lambda_d$. The (GIT) quotient $X$ of this level set $Z_\lambda$ by $T^d$ is an affine hypertoric variety. The action of $T$ on $Z_\lambda$ descends to a Hamiltonian action on $X$ with moment map $\mu:X\to \h^\ast$ induced by the projection $\cx^{2d}\times T\times \h^\ast\to \h^\ast$. We shall refer to $T$ 
as the {\em structure torus} of $X$. 
\par
The affine GIT quotient is easy to describe (cf.\ \cite{HS}). Let $t_j:T\to\cx^\ast$, $j=1,\dots,n$, be a $\oZ$-basis of characters, and $z_j$, $j=1,\dots,n$, the corresponding coordinates on $\h^\ast$. 
The coordinates $z_j$ are $T^d$-invariant, while  the ring of invariants 
$S=\cx[u_i,v_i,t_j]^{T^d}$ is generated by invariant monomials. The exponents $(a_i,b_i,\xi_j)$ of these monomials form a semigroup $M\subset (\oZ_{\geq 0})^{2d}\oplus\oZ^n$, given by:
\begin{equation} a_i-b_i+\sum_{j=1}^n \alpha_{i}^j\xi_j=0,\quad \forall\: i=1,\dots,d,
\label{abcd}\end{equation}
where $\alpha_i=(\alpha_i^1,\dots,\alpha_i^n)$ in the chosen basis.
In other words, $S$ is generated by $u_iv_i$, $i=1,\dots,d$, and  by the following invariant monomials $x_\xi$  indexed by  $\xi\in \sX^\ast(T)$: 
\begin{equation} x_\xi=\prod_{i=1}^d u_i^{a_i}v_i^{b_i}\prod_{j=1}^n t_j^{\xi_j},\enskip \text{where}\enskip 
a_i=\langle \alpha_i,\xi\rangle_-,\enskip b_i=\langle \alpha_i,\xi\rangle_+.\label{xxi}\end{equation}
The relations between these monomials are
\begin{equation} x_\xi\cdot x_{\xi^\prime}=x_{\xi+\xi^\prime}\prod_{i=1}^d (u_iv_i)^{\langle \alpha_i,\xi\rangle_+ +\langle \alpha_i,\xi^\prime\rangle_+ - \langle \alpha_i,\xi+\xi^\prime\rangle_+}   .\label{monomials}
\end{equation}
Observe that this description implies, in particular, that it does not matter whether we first impose equations \eqref{moment} and then quotient by $T^d$, or perform these operations in the opposite order.
\par
The Poisson structure of $X$ is computed easily from \eqref{xxi}:
\begin{gather}  \{z_i, z_j\}=0,\quad \{x_\xi, z_j\}=\xi_j x_\xi,\quad i,j=1,\dots,n,\label{Poisson1}\\
\{x_\xi, x_{\xi^\prime}\}=  x_{\xi+\xi^\prime}\sum_{k=1}^d c_k\prod_{i=1}^d\frac{ (\langle \alpha_i,z\rangle-\lambda_i)^{\langle \alpha_i,\xi\rangle_+ +\langle \alpha_i,\xi^\prime\rangle_+ - \langle \alpha_i,\xi+\xi^\prime\rangle_+}}{\langle \alpha_k,z\rangle-\lambda_k},\label{Poisson2}
\end{gather}
where $c_k=\langle \alpha_k,\xi\rangle_-\langle \alpha_k,\xi^\prime\rangle_+- \langle \alpha_k,\xi\rangle_+\langle \alpha_k,\xi^\prime\rangle_-$.

The following fact is well known; we give a proof for the sake of completeness.
\begin{proposition} A hypertoric variety $X$ is symplectic and Cohen--Macaulay. The moment map $\mu:X\to \h^\ast$ is flat. \label{flat-mu}
\end{proposition}
\begin{proof} As remarked above $X$ can be obtained by taking GIT quotient of $ Y=\cx^{2d}\times T\times \h^\ast$ by $T^d$ and then imposing the moment map equations \eqref{moment}. 
The GIT quotient of $Y$ by $T^d$ is Cohen\textemdash{}Macaulay owing to the Hochster--Roberts theorem and $X$ is a complete intersection in $Y/\!\!/ T^d$. Therefore $X$ is Cohen--Macaulay. It is normal since it is regular in codimension $1$. A Proj GIT quotient $\tilde X$ of  $ \cx^{2d}\times T\times \h^\ast$ by $T^d$ with respect to a generic character is a partial resolution of $X$ with cyclic symplectic singularities. Therefore $X$ is a symplectic variety. It remains to show the flatness of $\mu$. Each fibre of $\mu$ is a finite union of $T$-orbits (with at least one free orbit), hence $\mu$ is equidimensional. The flatness of $\mu$ follows now from the miracle flatness theorem. 
\end{proof}
\begin{remark} Let $\tilde T$ be an extension of $T$ by a finite (abelian) group $\Gamma$. In other words, the cocharacter lattice $\sX_\ast(\tilde T)$ is a sublattice of $\sX_\ast(T)$ with the quotient isomorphic to $\Gamma$. Suppose that the vectors $\alpha_k$ belong to $\sX_\ast(\tilde T)$. Then the hypertoric variety corresponding to $(T,\{\alpha_k\},\{\lambda_k\})$ is the quotient of the hypertoric variety corresponding to $(\tilde T,\{\alpha_k\},\{\lambda_k\})$ by $\Gamma\subset \tilde{T}$.\label{finiteq} 
\end{remark}
\begin{remark} In the original construction of \cite{BD} the vectors $\alpha_k$ were assumed to span $\h$. In this case the variety $X$ can be also obtained as a symplectic quotient of $\cx^{2d}$ by the kernel $N$  of $\alpha:T^d\to T$. If $\nu:\cx^{2d}\to \Lie(N)^\ast$ denotes the moment map for this action, then the level set $Z_\lambda$ given by \eqref{moment} is of the form $\nu^{-1}(0)\times T$, and quotienting by $T^d$ yield the same result as quotienting $\nu^{-1}(0)$ by $N$. Conversely, the general case can be reduced, modulo finite abelian quotients, to two simpler cases: $T\times \h^\ast$ and a hypertoric variety where the cocharacters $\alpha_k$ span $\h$. Indeed, let $T^d\to T$ be an arbitrary homomorphism with cokernel $T^\prime$ and let $T_0$ be the kernel of the projection $T\to T^\prime$. Since $T_0$ is isomorphic to $T^d/\Ker\alpha$, it is connected, i.e.\ a subtorus of $T$. Moreover  $T=(T_0 \times T^\prime)/\Gamma$ for an
abelian group $\Gamma$.
The cocharacters $\alpha_k$ belong to $\sX_\ast(T_0)$, hence to $\sX_\ast(T_0\times T^\prime)$, and the previous remark implies that our hypertoric variety $X$ is the quotient   by $\Gamma$ of the hypertoric variety $\tilde X$ corresponding to $(T_0 \times T^\prime,\{\alpha_k\},\{\lambda_k\})$. The variety $\tilde X$ is isomorphic to $X_0\times T^\ast T^\prime$, where $X_0$ is the hypertoric variety  corresponding to $(T_0,\{\alpha_k\},\{\lambda_k\})$. The cocharacters  $\alpha_k$ span $\h_0=\Lie(T_0)$.
\label{nu}
\end{remark}

We now address the question of isomorphism between symplectic quotients corresponding to different collections $(\alpha_k,\lambda_k)_{k=1,\dots,d}$ of cocharacters and scalars. By ``isomorphism" we mean an isomorphism in the category of affine $T$-Hamiltonian varieties, i.e.\ a $T$-equivariant biregular isomorphism $f:X\to X^\prime$ which preserves the Poisson structure and intertwines the moment maps, i.e.\ $\mu=\mu^\prime\circ f$. Clearly, the isomorphism type does not depend on the order of the $(\alpha_k,\lambda_k)$. Secondly, if any $\alpha_k=0$, then we obtain an isomorphic hypertoric variety by omitting $(\alpha_k,\lambda_k)$. Thirdly:
\begin{lemma} Suppose that, for some $k\in \{1,\dots,d\}$ and $m\in \Z^\ast$, $\alpha_k=m\beta_k$ with $\beta_k\in \sX^\ast(T)$. Replacing $(\alpha_k,\lambda_k)$ with $m$ copies of $(\beta_k,\lambda_k/m)$ results in an isomorphic hypertoric variety $X^\prime$.\label{multi}
\end{lemma}
\begin{proof} The isomorphism is described by a ring isomorphism $\cx[X]\to \cx[X^\prime]$ given by $x_\xi^\prime= x_\xi/m^{\langle \alpha_k,\xi^\prime\rangle_+} $ (observe that the exponents on the right hand side of \eqref{monomials} are invariant under the change $\alpha_i\mapsto -\alpha_i$). It preserves the Poisson structure due to formulae \eqref{Poisson1}-\eqref{Poisson2}.
\end{proof}
Let us call an affine hyperplane $H$ in $\h^\ast$ {\em integral} if there is a nonzero element in $\sX_\ast(T)$  normal to $H$. We then define an {\em integral multiarrangement} of hyperplanes in $\h^\ast$ to be a multiset $\{H_1,\dots,H_d\}$ of integral hyperplanes. The following is an improvement of results in \cite{BD,HP}.
\begin{proposition} There is a natural bijection between \begin{itemize}\item[1)] isomorphism classes of hypertoric varieties with a given structure torus $T$, and \item[2)] integral multiarrangements in $\h^\ast$.
\end{itemize}\label{bijection}
\end{proposition}
\begin{proof} Let $X$ be a hypertoric variety corresponding to a multiset $\{(\alpha_k,\lambda_k)\}$. Owing to the last lemma, we may assume that each $\alpha_k$ is primitive in $\sX_\ast(T)$. We then associate to $X$ the hyperplanes
$$ H_k=\{z\in \h^\ast; \langle \alpha_k,z\rangle=\lambda_k\},\quad k=1,\dots,d.$$
Conversely, let $\{H_k;k=1,\dots,d\}$ be a collection of integral hyperplanes in $\h^\ast$. Each $H_k$ has exactly two primitive normals in $\sX_\ast(T)$, which are negatives of each other. For each $k$ choose one of these, $\alpha_k$, so that $ H_k=\{z; \langle \alpha_k,z\rangle=\lambda_k\}.$ We obtain a hypertoric variety corresponding to $\{(\alpha_k,\lambda_k)\}$. The previous lemma implies that choosing instead $-\alpha_k$ for some $k$ results in an isomorphic variety, since we have to replace then $\lambda_k$ with $-\lambda_k$. Therefore the two maps are inverse to each other.
\end{proof}
\begin{remark} A multiarrangement can be viewed as a pair $(\sA,m)$, where $\sA=\{H_1,\dots,H_d\}$ is an ordinary arrangement (i.e.\ the hyperplanes $H_i$ are distinct), and $m:\sA\to \N$ is a multiplicity function. Different symplectic quotient realisations of the corresponding hypertoric variety $X$ correspond to different signed partitions $(m_1,\dots,m_r)$, $m_i:\sA\to \Z$, $m=\sum_{i=1}^r|m_i|$, of the multiplicity function $m$.
If the hyperplanes $H_k$ have equations $\langle \omega_k,z\rangle=\lambda_k$,  with primitive $\omega_k\in \sX_\ast(T)$, $k=1,\dots,d,$ and $(m_1,\dots,m_r)$ is a signed partition of $m$, then $X$ is obtained via the symplectic quotient construction corresponding to the multiset $\{(m_i(H_k)\omega_k, m_i(H_k)\lambda_k); k=1,\dots,d,\enskip i=1,\dots,r\}$.\label{multiarr}
\end{remark}
\begin{remark} Integral multiarrangements form an abelian monoid with the product given by the union of multiarrangements (as multisets) and with  the neutral element equal to the empty multiarrangement. Similarly, (isomorphism classes of) hypertoric varieties with given structure torus $T$ form an abelian monoid, with the product of $X_1$ and $X_2$ given by the symplectic quotient of $X_1\times X_2$ by the anti-diagonal $T$, and $T\times \h^\ast$ as the neutral element. It is easy to verify that the bijection of Proposition \ref{bijection} is a monoid isomorphism.\label{monoid}
\end{remark}
\begin{remark} It would be interesting to have an intrinsic characterisation of affine hypertoric varieties. It seems likely that these are affine symplectic varieties equipped with an effective Hamiltonian algebraic action of an algebraic torus of half dimension, such that the moment map is faithfully flat and has connected fibres.  The argument could proceed along the following lines: the $1$-dimensional subtori $S_1,\dots,S_d$ for which $X^{S_k}\neq \emptyset$ yield cocharacters $\alpha_1,\dots,\alpha_d\in \sX_\ast(T)$ and the sets $\mu(X^{S_i})$ are hyperplanes $H_k$ of the form $\langle \alpha_k,z\rangle=\lambda_i$, $k=1,\dots, d$. Consider the open subset $U$ of $X$ consisting of points $x$ such that $\mu(x)$ lies at most one of these hyperplanes. The complement of $U$ has codimension $2$ and so, given the normality of $X$, it is sufficient to construct an isomorphism between $U$ and the corresponding subset of a hypertoric variety. For this one needs a local model around each point of $U$. It is easy to see that away from the hyperplanes $U$ is locally isomorphic to $T^\ast T$ (since we assumed that the fibres of $\mu$ are connected). 
Since symplectic singularities are canonical \cite{Beau}, in codimension $2$ they are of the form $Y\times U^\prime$, where $Y$ is a du Val singularity and $U^\prime$ is smooth \cite[Corollary 1.14]{Reid}. Thus, in our case, a local model near $\mu^{-1}(H_k)$  should be the product of an $A_m$-surface $xy=z^m$ for some $m\in \oN$, and $T^\ast T_k$, where $T=S_k\times T_k$ and $S_k$ acts on the first factor in the standard way, and trivially on  $T^\ast T_k$. This will prove that $X$ is hypertoric.
\end{remark}

\subsection{$W$-invariant hypertoric varieties\label{W-inv}}
\begin{definition} Let $W$ be a Weyl group acting linearly on a torus $T$ as a reflection group. A hypertoric variety $X$ with structure torus $T$ is said to be {\em $W$-invariant} if the action of $T$ on $X$ extends to a Hamiltonian action of the semidirect product $T\rtimes W$. In particular, $W$ itself acts on $X$. 
\end{definition}
Clearly, we obtain a $W$-invariant hypertoric variety if the whole symplectic quotient construction is $W$-equivariant, i.e.\ if the collections $\alpha_1,\dots,\alpha_d$ of cocharacters and $\lambda_1,\dots,\lambda_d$ are $W$-invariant. In this case, we shall call $X$  {\em strongly $W$-invariant}.  This, however, is not the only possibility, and we shall now investigate the general case.
\par
We show first the uniqueness of an action of $T\rtimes W$.
\begin{proposition} Let $X$ be a hypertoric variety with structure torus $T$ and let $W$ be a Weyl group acting on $T$
as a reflection group. There exists at most one extension of the action of $T$ to a Hamiltonian action of $T\rtimes W$.\label{W-unique}
\end{proposition}
\begin{proof} An action of $T\rtimes W$ gives an action of $W$ on $X$. Recall from \S\ref{htv} that the coordinate ring $\cx[X]$ is generated by $\cx[\h^\ast]$ and $x_\xi$, $\xi\in \sX^\ast(T)$.  Since each coordinate $x_\xi$ belongs to the $1$-dimensional representation space of $T$ with weight $\xi\in \sX^\ast(T)$, the action of $W$ on $\cx[X]$ must be of the form $w.x_\xi=a_w(\xi) x_{w.\xi}$ for some nonzero scalars $a_w(\xi)$, $w\in W$. Suppose that there is a second action of $W$ on $X$ such that $\mu:X\to \h^\ast$ is $W$-equivariant. Such an action must be given by another collections of functions $a_w^\prime: \sX^\ast(T)\to \cx^\ast$, $w\in W$. Equations \eqref{monomials} imply that the function $b_w:\sX^\ast(T)\to \cx^\ast$, $b_w(\xi)=a^\prime_w(\xi)a_w(\xi)^{-1}$ satisfies $b_w(\xi+\xi^\prime)=b_w(\xi)b_w(\xi^\prime)$ for any $\xi,\xi^\prime\in \sX^\ast(T)$. Therefore there exists a $t_w\in T$ such that $b_w(\xi)=\xi(t_w)$ for all $\xi\in \sX^\ast(T)$. The map $w\to t_w$ is  a $1$-cocycle in $Z^1(W;T)$ and, hence, the two $W$-actions correspond to two splittings of $T\rtimes W\to W$.
\end{proof}
We now address the question of existence. Obviously, $W$-invariance of the integral multiarrangement $(\sA,m)$ which determines $X$ (cf.\ Remark \ref{multiarr}) is a  necessary condition for the $W$-invariance of $X$. We shall show that a choice of a $W$-invariant partition $(m_1,\dots,m_r)$, $m_i:\sA\to \Z_{\geq 0}$, $m=\sum_{i=1}^r m_i$, of the multiplicity function $m$ yields an action on $X$ of an extension $\tilde W$ of $W$ by a $2$-subgroup $\Gamma_2=\Gamma_2(m_1,\dots,m_r)$ of $T$. The fibred semidirect product $E=T\rtimes_{\Gamma_2}\tilde W$ is then an extension of $W$ by $T$ and it also acts on $X$ and $E$. Moreover $E$ and its action will be shown to be independent (up to an isomorphism) of the partition chosen.
\par
 Integral multiarrangements with a choice of a partition of the multiplicity function are equivalent to pairs  $(\sA,m)$ where $\sA$ is a multiset of integral hyperplanes and $m:\sA\to \oN$ is a multiplicity function. Let us call  these objects {\em partitioned integral  multiarrangements}. 
\par
Let then $(\sA,m)$ be a $W$-invariant partitioned integral arrangement such that the underlying integral  multiarrangement corresponds to $X$.
Let $H_1,\dots, H_d$ be the  hyperplanes in $\sA$ and let $m_k=m(H_k)$, $k=1,\dots,d$. 
The collection $\sN=\{\pm m_k\omega_k;k=1,\dots,d\}$, where $\pm\omega_k$ are primitive normals to  the hyperplanes $H_k$, is also $W$-invariant
and we obtain a group homomorphism $\phi:W\to W_d$, where $W_d\simeq (\Z_2)^d\rtimes\Sigma_d$ is the Weyl group  of ${\rm Sp}_{2d}(\cx)$ acting in the obvious way on $\sN$. Let $\alpha:\sA\to \sN$ be a section of the natural projection. It gives us a $\Z$-linear map $\alpha:\oZ^d\to \sX_\ast(T)$ via $\alpha(e_k)=\alpha(H_k)$ which we can use (as explained in the previous subsection) to construct the variety $X$ as a symplectic quotient. 
 The isomorphism type of $X$ does not depend on the choice of the section $\alpha$ (any two sections differ by an element $\gamma\in (\Z_2)^d\subset W_d$).
We observe that the map $\alpha:\oZ^d\to \sX_\ast(T)$ is $W$-equivariant,\ i.e. $w.\alpha(s)=\alpha(\phi(w).s)$ for all $w\in W$ and $s\in \cx^d$.
\par
Let $\tilde W_d\simeq \Z_4^d\rtimes\Sigma_d$ be the Tits group of $W_d$ \cite{Tits}. It is a nontrivial extension
\begin{equation} 1\longrightarrow (\Z_2)^d \longrightarrow \tilde W_d\longrightarrow W_d\longrightarrow 1, \label{Tits}\end{equation}
and it can be realised as a subgroup of the normaliser $N_d$ of the maximal torus $T^d$ in ${\rm Sp}_{2d}(\cx)$ as follows:  the normal subgroup $\Sigma_d$ acts by permutations of the $d$   $\C^2$-factors, while the $j$-th generator of $(\Z_4)^d$ acts on the $j$-th copy of $\C^2$ as $(u,v)\mapsto (-v,u)$ and as the identity on all other factors.  The extension class $[c_d]\in H^2(W_d;(\Z_2)^d)$  of \eqref{Tits} is also the extension class of 
$$1\longrightarrow T^d \longrightarrow N_d\longrightarrow W_d\longrightarrow 1,$$
under the diagonal embedding $(\Z_2)^d\hookrightarrow T^d$. Let $\Gamma_2(m)\subset T$ denote the image of $(\Z_2)^d$ under the homomorphism $\alpha$. It depends only on the partition function $m$, and not on $\alpha$. The cohomology class $\alpha\circ (\phi^\ast[c_d])\in H^2(W;\Gamma_2(m))$
determines a group extension
\begin{equation} 1\longrightarrow \Gamma_2(m) \longrightarrow \tilde W\longrightarrow W\longrightarrow 1, \label{tildeW}\end{equation}
as well as 
\begin{equation*}1\longrightarrow T \longrightarrow E\longrightarrow W\longrightarrow 1.\end{equation*}
The group $E$ is isomorphic to $E=T\rtimes_{\Gamma_2(m)}\tilde W$.
\begin{proposition} The group $\tilde W$ acts symplectically on $X$, and the action of $T$ extends to a Hamiltonian action of the group $E$. If $\alpha,\alpha^\prime$ are two sections of $\sN\to \sA$, then the corresponding groups $\tilde W$ and $\tilde W^\prime$ are isomorphic and there exists an automorphism $\phi:X\to X$ intertwining the respective actions. 
\end{proposition}
\begin{proof} The cohomology class $\phi^\ast[c_d]\in H^2(W;(\Z_2)^d)$ defines an extension
\begin{equation} 1\longrightarrow (\Z_2)^d \longrightarrow \hat W\stackrel{p}{\longrightarrow} W\longrightarrow 1,\label{hatW}\end{equation}
together with a homomorphism $\hat\phi:\hat W\to \tilde W_d$.
 We let $\hat{W}$ acts on $\C^{2d}\times T^\ast T$ by
\[
\hat{w}\cdot (u,v,z,t) = \left( \hat{\phi}(\hat{w})\cdot (u,v), p(\hat{w})\cdot z, p(\hat{w})\cdot t\right)
\]
for all $\hat{w}\in\hat{W}$, $(u,v)\in \C^{2d}$, $z\in\h^\ast$ and $t\in T$. This action is symplectic and it normalises
the action of $T^d$, since $\tilde W_d\subset N_d$ and the map $\alpha$ is $W$-equivariant.
 Furthermore, we claim that the action of $\hat{W}$ on $\C^{2d}\times T^\ast T$ preserves the level set \eqref{moment}  of the moment map for the action of $T^d$ on $\C^{2d}\times T^\ast T$. Indeed, the map $\nu: \C^{2d}\to \C^d$ defined by $\nu(u,v) = (u_1v_1,\dots,u_dv_d)$ satisfies $\nu \circ \tilde{w} = w\circ \nu$ for all $\tilde{w}\in\tilde{W}_d$ with image $w\in W_d$. We conclude that the group $\hat{W}$ acts on the symplectic quotient $X$. This action factors through the homomorphism $\hat W\to \tilde W$ induced by $\alpha$. The group $E=T\rtimes_{\Gamma_2(m)}\tilde W$ acts then as well. The second statement follows from the fact that there is an element $\gamma\in(\oZ_2)^d$ such that $\alpha^\prime=\gamma\alpha$. 
\end{proof}
\begin{example}
Let $X$ be the hypertoric variety associated with the hyperplanes with normals given by the positive roots of $\ssl_3(\C)$ and zero scalars, i.e.\ we have three hyperplanes defined by
\[
H_i = \{ (z_1,z_2,z_3)\in \C^3; z_j - z_k =0\}
\]
for $(ijk)$ a cyclic permutation of $(123)$. The map $\alpha:\C^3\to \C^2$ is $(t_1,t_2,t_3)\mapsto (t_1 t_3^{-1},t_2 t_1^{-1}, t_3 t_2^{-1})$. In particular, note that $(-1,-1,-1)$ lies in the kernel of $\alpha$. On the other hand, the group homomorphism $\phi$ from $W=\Sigma_3$ to $W_3$ is the one given by the representation of $W$ on $\Lambda^2\C^3$, where $\C^3$ is the standard representation of $\Sigma_3$ given by permutation matrices. In other words, if $W=\langle s_1, s_2; s_1^2=s_2^2=(s_1 s_2)^3=1\rangle$ then the action of the generators is given by $s_1(z_1,z_2,z_3)=(-z_1,-z_3,-z_2)$ and $s_2(z_1,z_2,z_3)=(-z_3,-z_2,-z_1)$. The extension $\tilde{W}$ can then be taken the group generated by $\tilde{s}_1,\tilde{s}_2$, where $\tilde{s}_1$ acts on $(\C^2)^3$ as
\[
\tilde{s}_1(u_1,v_1,u_2,v_2,u_3,v_3) = (-v_1, u_1,-v_3, u_3, -v_2, u_3),
\]
and similarly for $\tilde{s}_2$. One checks that $\tilde{s}_1^2=\tilde{s}_2^2=(\tilde{s}_1\tilde{s}_2)^2=(-1,-1,-1)\in (\C^\ast)^3$ and one deduces from this that $\phi^\ast [\alpha \circ c_d]\in H^2(W;T)$ is trivial while $\phi^\ast [c_d]$ is not.
\label{123}
\end{example}
We can compute an explicit $2$-cocycle representing the cohomology class $\alpha\circ (\phi^\ast[c_d])\in H^2(W;T)$ as follows. 
Identify $\sA$ with $\{1,\dots,d\}$ and write $\alpha_i$ for $\alpha(i)$. 
Let $w\in W$ and consider the transformation of $\cx[x_\xi;\,{\xi\in \sX^\ast(T)}]$ given by $x_\xi\mapsto a_w(\xi) x_{w.\xi}$ for some nonzero scalars $a_w(\xi)$. Equations \eqref{monomials} imply that this transformation (together with $z\to w.z$ on $\h^\ast$) induces a transformation of $\cx[X]$ if and only if the  functions $a_w:\sX^\ast(T)\to \cx^\ast$ satisfy the relations:
\begin{equation} a_w(\xi)a_w(\xi^\prime)=a_w(\xi+\xi^\prime)(-1)^{\sum _{i\in I_w} (\langle \alpha_i,w.\xi\rangle_+ + \langle \alpha_i,w.\xi^\prime\rangle_+ - \langle \alpha_i,w.\xi+w.\xi^\prime\rangle_+)}, \label{relations2}\end{equation}
where $I_w=\{i;w.\alpha_i=-\alpha_{w(i)}\}$.
A particular solution of these equations is given by
\begin{equation}
a_w(\xi)=(-1)^{d_w(\xi)},\enskip \text{where}\enskip d_w(\xi)=\sum_{i\in I_w} \langle \alpha_i,w.\xi\rangle_+.\label{a_w}\end{equation}
Let us then define  a 
transformation $\tilde w$ acting on $\C[X]$, and hence on $X$, by
$\tilde w(z)=w.z$, $\tilde w(x_\xi)=(-1)^{d_w(\xi)} x_{w.\xi}$. 
A computation shows that 
\begin{equation} \widetilde{w_1w_2}\tilde w_2^{-1}\tilde w_1^{-1}= (-1)^{d(w_1,w_2)}, \enskip\text{where}\enskip d(w_1,w_2)=\sum_{k\in w_1(I_{w_1})\cap I_{w_2}}\alpha_k.\label{cocycle1}
\end{equation}
Thus $c(w_1,w_2)= (-1)^{d(w_1,w_2)}$ is a cocycle in $Z^2(W;\Gamma_2(m))$ which determines the extension \eqref{tildeW}. The same argument as in the proof of Proposition \ref{W-unique} shows that the image of  $c(w_1,w_2)$ in $H^2(W;T)$ does not depend on the choice of a partition of the multiplicity function $m$. It also does not depend on the choice of primitive normals to the hyperplanes in the multiarrangements, i.e.\ it depends only on the integral multiarrangement $(\sA,m)$.
\begin{definition} Let $(\sA,m)$ be a $W$-invariant integral multiarrangement in $\sX_\ast(T)$. We denote by $\o(\sA,m)\in H^2(W;T)$ the cohomology class $\alpha\circ (\phi^\ast[c_d])=\phi^\ast[\alpha\circ c_d]$,
where $\alpha(e_k)=m(H_k)\omega_k$ and $\omega_k$ is an arbitrary choice of a primitive normal to $H_k$, $k=1,\dots,d$.
\label{obstructiono}
\end{definition}
\begin{remark} The obstruction $\o$ is additive on the monoid of $W$-invariant integral multiarrangements. Moreover, $\o(\sA,m)$ does not depend on the scalars $\lambda_k$. In particular,  $\o(\sA,m)= \o(\sA^\circ,m^\circ)$ where $\o(\sA^\circ,m^\circ)$ is the corresponding central arrangement obtained by setting all $\lambda_k$ to $0$, and the multiplicity $m^\circ(H)$ of a hyperplane given by $\langle\alpha,z\rangle=0$ is the sum of multiplicities of hyperplanes in $\sA$ given by $\langle\alpha,z\rangle=\lambda_k$ for some $\lambda_k$.\label{centralarr}
\end{remark}
\begin{corollary} A hypertoric variety is $W$-invariant  if and only if its integral multiarrangement $(\sA,m)$ is $W$-invariant and $\o(\sA,m)$ is trivial. \qed
\label{obstruction}
\end{corollary}
\begin{remark}
The definition of a $W$-hypertoric variety $X$ means that the semidirect product $T\rtimes W$ acts on $X$. Any splitting of $T\rtimes W\to W$ defines an action of $W$ on $X$. By definition, we always choose the action given by the splitting $w\mapsto (1,w)$, i.e.\ by the trivial element of $H^1(W;T)$. The latter group is trivial for many, but not all, maximal tori of simple algebraic groups \cite{HMS}.
\end{remark}
\begin{remark}  Observe that $X$ is strongly $W$-invariant, i.e.\ there is a $W$-invariant choice of primitive normals $\omega_k$ to the hyperplanes $H_k$, $k=1,\dots,d$, if and only if $\phi(W)$ lies in a subgroup conjugate to $\Sigma_d\lhd W_d$. This implies that $\phi^\ast[c_d]=1$. On the other hand, if  $\phi^\ast[c_d]=1$, then  $X$ is $W$-invariant, and it is strongly $W$-invariant if and only if the symplectic representation $\cx^{2d}$ of $W$ is of the form $V\oplus V^\ast$ with $V$ a $d$-dimensional representation of $W$.
\end{remark}
\begin{remark} We have observed in Remark \ref{nu} that a general hypertoric variety $X$ can be obtained as a finite quotient of the product of $T^\ast T'$ and $X_0$, where $X_0$ is a hypertoric variety with structure torus $T_0\simeq T^d/\Ker\alpha$ and the cocharacters spanning $\Lie( T_0)$. In the $W$-setting, the argument works only if the homomorphism $\alpha:T^d\to T$ is $W$-equivariant, i.e.\ if $X$ is strongly $W$-invariant. In this case, $W$ acts by reflections on $\h=\h^\prime\oplus\h_0$, $\h^\prime$ and $\h_0$. This means that $W=W^\prime\times W_0$ with $W^\prime$ acting trivially on $T_0$ and  $W_0$ acting trivially on $T^{\prime}$. \label{W-product}
\end{remark}
\begin{remark}  $W$-invariant hypertoric varieties with structure torus $T$ form a submonoid of the full monoid of hypertoric varieties defined in Remark \ref{monoid}. Similarly, strongly $W$-invariant hypertoric varieties with structure torus $T$ form a subsubmonoid.\label{monoid2}
\end{remark}
The explicit formula \eqref{cocycle1} for the cocycle $c(w_1,w_2)$   gives immediately the following simple sufficient conditions for the vanishing of the obstruction:
\begin{proposition} Let $(\sA,m)$ be a $W$-invariant arrangement such that for any $w\in W$, $I_w=\emptyset$ or $I_w=\{1,\dots,d\}$, and  $\sum_{\alpha_k\in A}\alpha_k\in 2\sX_\ast(T)$.
Then $\o(\sA,m)$ vanishes.
\label{W-enough}\end{proposition}
\begin{proof} The assumptions imply that $c(w_1,w_2)\equiv 1$.
\end{proof}
\begin{example} Let $T$ be the maximal torus of ${\rm GL}_2(\cx)$ and let $X$ be the hypertoric variety determined by a single cocharacter $(N,-N)$ with $N$ odd. Equations \eqref{monomials} imply that $X$ is a subvariety of $\h\times \cx^4$ with coordinates $z_1,z_2,x_{1,0}, x_{0,1}, x_{1,1}, x_{-1,-1}$ given by equations 
$$ x_{1,0}x_{0,1}=(Nz_1-Nz_2)^Nx_{1,1},\quad x_{1,1} x_{-1,-1}=1.$$
If we set $z=Nz_1-Nz_2$, $w=Nz_1+Nz_2$, $x=x_{1,0}$, $y=x_{0,1}$, $t=x_{1,1}$, then
\begin{equation}
X=\{(z,w,x,y,t)\in \cx^4\times \cx^\ast;\, xy=z^Nt\}.\end{equation}
It is a $W$-invariant hypertoric variety, with $W=\Z_2$ acting via 
$$(z,w,x,y,t)\longmapsto (-z,w,y,x,-t),$$
which does not satisfy  the assumption in Proposition \ref{W-enough}.\label{bastard}
\end{example}
\par
The simplest example of a hypertoric variety with an action of a non-trivial extension of $\tilde{W}$ of $W$ are $A_{m+2k-1}$-surfaces of the form $xy=z^m\prod_{i=1}^{k}(z^2-\tau_i)$ with $m$ odd. In this case the involution $z\mapsto -z$ that preserves the collection $\{ 0, \pm \sqrt{\tau_1},\dots, \pm\sqrt{\tau_k}\}$ can only be lifted to the $\Z_4$-action on $X$ generated by $(x,y,z)\mapsto (-y,x,-z)$. This 2-dimensional example implies necessary conditions for the vanishing of $\phi^\ast [\alpha\circ c_d]\in H^2(W;T)$. Indeed, if $W$ acts on the hypertoric variety $X$, then the $\Z_2$-subgroup generated by a reflection $s_\theta$ associated with a root $\theta$ of the Lie group $G_{{\scriptscriptstyle T,W}}$ must act on the symplectic quotient of $X$ by the corank $1$ torus with Lie algebra $\ker\theta$. One can show that this symplectic quotient is an $A_{k-1}$-surface with $k= d(\theta,\mathcal{X}^\ast(T))^{-1}\rm{tr}\,\theta$. Here $d(\theta,\mathcal{X}^\ast(T))$ is the divisibility of $\theta$ in the character lattice of $T$ and
\[
\rm{tr}(\theta) = \sum_{i=1}^d{\langle \theta, \alpha_i\rangle}.
\]
Hence the vanishing of $\phi^\ast [\alpha\circ c_d]$ implies the divisibility constraint
\begin{equation}\label{eq:Anomaly:Cancellation}
d(\theta,\mathcal{X}^\ast(T))^{-1}\rm{tr}\,\theta\in 2\Z,
\end{equation}
for every root $\theta$ of $G_{{\scriptscriptstyle T,W}}$. The following example shows that in general these necessary conditions are not sufficient for the vanishing of $\phi^\ast [\alpha\circ c_d]$. 
\begin{example}\label{eg:Standard:SO(4)}
Consider the product $X=\C^2/\Z_m\times \C^2/\Z_m$ with structure torus $T=\cx^\ast\times \cx^\ast$, which we view as the maximal torus of ${\rm SO}_4(\C)$. The weight lattice is simply the standard lattice $\Z^2\subset \C^2$, and $X$ is determined by the cocharacters $\alpha_1=(m,0)$ and $\alpha_2=(0,m)$. The  action of the Weyl group $W=\Z_2\times\Z_2$ on $\C^2$ is generated by the reflections $s_1:(z_1,z_2)\mapsto (z_2,z_1)$ and $s_2=-s_1$. Note that the remaining non-trivial element $s_1 s_2 = s_2 s_1 = -1$ of $W$ is an involution but is not a reflection. The normals to the reflection hyperplanes (the coroots) can be taken to be $v_{\theta_1} = (1,-1)$ and $v_{\theta_2} = (1,1)$, with corresponding roots $\theta_1:(z_1,z_2)\mapsto z_1-z_2$ and $\theta_2:(z_1,z_2)\mapsto z_1 + z_2$. Note that $\theta_1,\theta_2, v_{\theta_1}$ and $v_{\theta_2}$ are all primitive so the maximal torus of $G_{{\scriptscriptstyle T,W}}$ is $T$. The necessary divisibility  conditions are
\[
\rm{tr}\, \theta_1 = m-m=0, \qquad \rm{tr}\, \theta_2 = m+m=2m
\]
are satisfied and yet the involution $s_1 s_2=s_2s_1$ can  be lifted to an involution of $X$ only if $m$ is even.
\end{example}
\subsubsection{Strongly $W$-invariant vs.\ $W$-invariant\label{sww}}

At the beginning of the subsection we have defined strongly $W$-invariant hypertoric varieties as those which can be obtained via a $W$-equivariant symplectic quotient. It is therefore, apriori, a property of a particular symplectic quotient realisation and not of the integral multiarrangement of the hypertoric variety. 
\par
Let us call a $W$-invariant integral multiarrangement $(\sA,m)$ in $\h^\ast$ {\em orientable}, if  there exists a $W$-equivariant choice of primitive normals to hyperplanes in $\sA$. Thus hypertoric varieties with an orientable multiarrangement are strongly $W$-invariant, but the converse is not necessarily true. In order to understand the difference, let us call an integral hyperplane $H$ in $\h^\ast$ {\em self-dual} if there exists an element $w\in W$ which takes a normal of $H$ to its negative. Clearly, self-dual hyperplanes in a $W$-invariant integral multiarrangement $(\sA,m)$ form a $W$-invariant subarrangement, which we call the self-dual part of $(\sA,m)$.
\begin{lemma} The complement $(\sA^\prime,m^\prime)$ of the self-dual part of a $W$-invariant multiarrangement $(\sA,m)$ is orientable.
\label{self-dual}\end{lemma}
\begin{proof} Let $(\sB,n)$ be a $W$-orbit in $(\sA^\prime,m^\prime)$. Choose a hyperplane $H_0\in \sB$ and its primitive normal $\omega_0$. For any other hyperplane $H\in \sB$ which is equal to $w.H_0$ for some $w\in W$, choose its normal to be $w.\omega_0$. This is independent of $w$, since if $\tilde w.H_0= w.H_1$ and $\tilde w.\omega_0=-w.\omega_0$, then $(w^{-1}\tilde w).\omega_0=-\omega_0$, i.e.\ $H_0$ is self-dual.
\end{proof}
\begin{remark} If the longest element $w_0$ of $W$ acts as $-1$ on $\h$, then every integral hyperplane is self-dual. This is the case for maximal tori of all simple algebraic groups except  $A_n$ ($n>1)$, $D_{2n+1}$, and $E_6$.
\end{remark}
Let us also adopt the following terminology:
\begin{definition} Let $(\sA,m)$ be a $W$-invariant  integral multiarrangement. The dimension of $W$-invariant deformations (obtained by translations of individual hyperplanes) of $(\sA,m)$ can be at most $\sum_{j=1}^s m_j$, where $s$ is the number of $W$-orbits in $\sA$ and $m_j=m(H)$ for any $H$ in the $j$-th orbit. We shall say that $(\sA,m)$ has {\em unobstructed $W$-invariant deformations} if the actual dimension of the space of $W$-invariant deformations is equal to this maximal possible dimension.\label{uW-id} \end{definition}
\begin{remark} Deformations (as affine $T\rtimes W$-Hamiltonian varieties) of a $W$-invariant hypertoric variety $X$ given by a $W$-invariant  integral multiarrangement $(\sA,m)$ correspond exactly to $W$-invariant deformations of 
$(\sA,m)$ (since the obstruction $\o(\sA,m)$ does not depend on the scalars $\lambda_k$).
\end{remark}

Observe now that self-dual hyperplanes in $ \sA$ which do not pass through the origin come in pairs $\pm\langle\omega,z\rangle=\lambda\neq 0$ of equal multiplicity. Therefore such an $H$ contributes an even multiplicity to the corresponding central arrangement $(\sA^\circ,m^\circ)$. Conversely, if $W$-invariant deformations of a central arrangement  $(\sA^\circ,m^\circ)$ are unobstructed, then every self-dual hyperplane in $\sA^\circ$ has even multiplicity. 
Examples \ref{123} and \ref{bastard} provide examples of arrangements not satisfying this condition.
 We have the following characterisation of multiarrangements of strongly $W$-invariant hypertoric varieties.
\begin{proposition} A $W$-invariant hypertoric variety $X$ with integral multiarrangement $(\sA,m)$ is strongly $W$-invariant if and only if every self-dual hyperplane $H\in\sA$ passing through the origin  has even multiplicity.\label{self-dual2}
\end{proposition}
\begin{proof} Deformations of strongly $W$-invariant varieties are clearly unobstructed. The discussion above shows that the multiarrangement of such a variety satisfies the condition in the statement. Conversely, suppose that the condition in the statement holds. Decompose $(\sA,m)$ into $(\sA_0,m)\oplus (\sA_1,m)$ where $\sA_0$ consists of self-dual hyperplanes passing through the origin, and  $\sA_1$ is its complement. Lemma \ref{self-dual} and the above discussion show that $\sA_1$ is orientable. On the other hand, the partitioned multiarrangement $(\sA_0,m/2,m/2)$ is also orientable.
Hence $X$ is strongly $W$-invariant.
\end{proof}
\begin{remark} It follows that strongly $W$-invariant hypertoric varieties are precisely those with unobstructed $W$-invariant deformations.
\end{remark}
\begin{remark} The condition in the last proposition implies that the functions $a_w:\sX^\ast(T)\to \cx^\ast$, defined in \eqref{a_w} are identically $1$, and therefore the action of $W$ on the coordinate ring $\cx[X]$  of such a hypertoric variety is  given simply by $w.x_\xi=x_{w.\xi}$, $\xi\in \sX^\ast(T)$. Presumably the converse is also true.
\end{remark}

\section{$W$-Hilbert schemes of strongly invariant hypertoric varieties\label{Snormal}}

Let $X$ be a strongly $W$-invariant hypertoric variety. 
We fix a $W$-invariant scalar product on $\im \fH$, where $\fH=\Lie(T)$, and use the resulting Hermitian inner product on $\h$  to identify $\h$ with $\h^\ast$. Thus $X$ is a symplectic quotient of $\cx^{2d}\times T\times \h$ by $T^d=(\cx^\ast)^d$, where $T^d$ acts on $T$ via a $W$-equivariant homomorphism $\alpha:\oZ^d\to \sX_\ast(T)$. 
The moment map on $X$ is now a $W$-equivariant map $\mu:X\to \h$.
\par
We begin by investigating the equivariant Hilbert scheme of $\cx^{2d}\times T\times \h$. 
We denote by $\sP$ the product $\cx^{2d}\times T$ and by $\mu:\sP\times \h\to \h$ the projection onto the second factor. Let $p:\cx^{2d}\to \cx^d$ be the map
\begin{equation}\bigl((u_i)_{i=1}^d,(v_i)_{i=1}^d\bigr)\longmapsto \bigl((u_iv_i)_{i=1}^d\bigr),\label{pppp}\end{equation}
and denote by the same letter maps $\cx^{2d}\times \h\to \cx^d\times \h$, $\sP\to \cx^d$, and $\sP\times \h\to \cx^d\times \h$ defined as the identity on $\h$ and the composition of $p$ with natural projections on the first factor.
\par
According to the description in \S\ref{eths}, $W\text{-}\Hilb_\mu(\sP\times \h)$ fibres over $\h/W$ with fibre over $c\in\h/W$ equal to the scheme ${\rm Mor}(D_c,\sP)^W$ of equivariant morphisms from the $0$-dimensional subscheme $D_c$ of $\h$ (corresponding to $c$) to $\sP$. We have an analogous description of  $W\text{-}\Hilb_\mu(\C^d\times \h)$. For every $c\in \h/W$ we have a morphism ${\rm Mor}(D_c,\sP)^W\to {\rm Mor}(D_c,\C^d)^W$ between the corresponding fibres over $c$, given 
 by composing a morphism $D_c\to \sP$ with the map $p:\sP\to \C^d$.
 \par
 Let $\alpha_i=\alpha(e_i)\in\sX_\ast(T)$, $i=1,\dots,d$, be the cocharacters determining the homomorphism $\alpha:T^d\to T$. 
We decompose the set  $\{\alpha_i; i=1,\dots, d\}$ into orbits of $W$. Let $\{1,\dots,d\}=\bigcup_{m=1}^s I_m$ be the corresponding decomposition of the set of indices, and set $ d_m=|I_m|$, $m=1,\dots,s$. 
 We obtain a corresponding decomposition $\cx^d\simeq \bigoplus_{m=1}^s E_m$, where $E_m=\langle e_i; i\in I_m\rangle$.
 For every $m=1,\dots,s$ we 
consider the $W$-equivariant linear  map\footnote{\ $\bigoplus_{m=1}^s L_m$ is the natural map $\h^\ast\to(\cx^d)^\ast$.} $L_m:\h\to E_m$ given by $L_m(z)=\sum_{i\in I_m} \langle z, \alpha_i\rangle e_i$. It induces a morphism $\iota_m:\h/W\to E_m/\Sigma_{d_m}$ and,  for every $c\in \h/W$,  a ring homomorphism from the ring $R_{\iota_m(c)}$ of $\Sigma_{d_m}$-coinvariants to the ring $R_c$ of $W$-coinvariants. We have:
\begin{lemma} Suppose that the vectors $\alpha_i$, $i=1,\dots,d$, form a single orbit of $W$. Then the map $L:\h\to \cx^d$, $L(z)=\sum_{i=1}^d \langle z, \alpha_i\rangle e_i$ induces  an isomorphism 
$$\Hom_\cx(\cx[u_1,\dots,u_d],R_{\iota(c)})^{\Sigma_d}\longrightarrow \Hom_\cx(\cx[u_1,\dots,u_d],R_{c})^W$$
for every $c\in \h/W$.\label{key}
\end{lemma}
\begin{proof}  For any $ c\in \cx^n/W$, an element $U$ of $\Hom_\cx(\cx[u],R_{\iota(c)})^{\Sigma_d}$ is given by
$$u_i=\sum_{j=0}^{d-1} a_j\zeta_i^j,\enskip i=1,\dots,d.$$
The image of $U$ in $\Hom_\cx(\cx[u],R_{ c})^{W}$ is given by substituting $\langle z, \alpha_i\rangle$ for $\zeta_i$, $i=1,\dots,d$. Thus the kernel of the map is nonzero if and only if there exists a polynomial $p(\zeta)$ of degree $d-1$ in one variable such that $p(\langle z, \alpha_i\rangle)=0$ for $i=1,\dots, d$ and every $z$. This means that there exist $i_1\neq i_2$ with $\alpha_{i_1}=\alpha_{i_2}$ which contradicts the assumption. Therefore our map has trivial kernel, and  since both vector spaces have dimension $d$, it is an isomorphism.
\end{proof}
It follows, as in Example \ref{bigHilb}, that in this case:
\begin{equation*} W\text{-}\Hilb_\pi(\cx^d\times \h)\simeq \bigl\{(c,A)\in \h/W\times \gl_d(\cx);\, [A,S_{\iota(c)}]=0\},
\end{equation*}
where $S_{\iota(c)}$ is the element of the Slodowy slice $\mathscr{S}_{\gl_d(\cx)}$ corresponding to $\iota(c)$.
Similarly,
\begin{equation*} W\text{-}\Hilb_\pi(T^d\times \h)\simeq \bigl\{(c,\gamma)\in \h/W\times {\rm GL}_d(\cx);\, \gamma S_{\iota(c)}\gamma^{-1}=S_{\iota(c)}\}.
\end{equation*}
In the general case, let us write:
$$\gl(d,W)=\bigoplus_{m=1}^s \gl(E_m),\enskip {\rm GL}(d,W)=\prod_{m=1}^s {\rm GL}(E_m),\enskip
\iota=\bigoplus_{s=1}^m \iota_m:\h/W\to \gl(d,W).$$
Then:
\begin{equation} W\text{-}\Hilb_\pi(\cx^d\times \h)\simeq \bigl\{(c,A)\in \h/W\times \gl(d,W);\, [A,S_{\iota(c)}]=0\}.\label{glww}
\end{equation}
Let $\lambda_1,\dots,\lambda_d$ be the scalars used to define the hypertoric variety $X$ and let $Z_\lambda$ be the level set \eqref{moment} of the moment map for the $T^d$-action on $Y=\cx^{2d}\times T\times  \h $.
Observe that the matrix $\lambda=\diag(\lambda_1,\dots,\lambda_d)$, where $\lambda_k$ are the scalars defining the hypertoric variety $X$, belongs to the centre of $\gl(d,W)$.  
\par
Recall from Example \ref{sTG} that $\sT_{{\scriptscriptstyle T, W}}=W\text{-}\Hilb_\mu(T\times \h)$ is isomorphic to the universal centraliser in the reductive group $G_{{\scriptscriptstyle T, W}}$ associated to $(T,W)$. 
We easily deduce the following description of the relevant transverse $W$-Hilbert schemes:
\begin{proposition}
\begin{itemize}
\item[(i)] $W\text{-}\Hilb_\mu(\cx^{2d}\times \h)$ is canonically isomorphic to the variety 
$\bigl\{(c,A,B)\in \h/W\times  \gl(d,W)^2;\,  [A,S_{\iota(c)}]=[B,S_{\iota(c)}]=0\bigr\};
$
\item[(ii)] The morphism $\bar p:W\text{-}\Hilb_\mu(\cx^{2d}\times \h)\to  W\text{-}\Hilb_\pi(\cx^d\times \h)$ induced by \eqref{pppp} is given by
$\bar p(c,A,B)=(c,AB)$;
\item[(iii)] $W\text{-}\Hilb_\mu(\h\times T^d)\simeq \bigl\{(c,\gamma)\in \h/W\times {\rm GL}(d,W);\; \gamma S_{\iota(c)}\gamma^{-1}=S_{\iota(c)}\bigr\}$.
\item[(iv)]
$W\text{-}\Hilb_\mu(\sP\times \h)\simeq \sT_{{\scriptscriptstyle T, W}}\times_{\h/W}W\text{-}\Hilb_\mu(\cx^{2d}\times \h)$; 
\item[(v)] $W\text{-}\Hilb_\mu(Z_\lambda)\simeq \bigl\{(c,g,A,B)\in W\text{-}\Hilb_\mu(\sP\times \h);\, AB+\lambda=S_{\iota(c)}\bigr\} $.\hfill\qedsymbol
\end{itemize}\label{hooray}
\end{proposition}
\begin{lemma} The induced morphism $\bar p:W\text{-}\Hilb_\mu(\cx^{2d}\times \h)\to W\text{-}\Hilb_\pi(\cx^d\times \h)$ is faithfully flat.\label{pflat}
\end{lemma}
\begin{proof} Since both schemes are regular (which follows either from the above descriptions or from Proposition \ref{smooth}), it is enough to show that $\bar p$ is equidimensional and surjective. Let $(c,C)\in  W\text{-}\Hilb_\pi(\cx^d\times \h)$, i.e.\ $[C,S_{\iota(c)}]=0$ (cf.\ \eqref{glww}). The fibre of $\bar p$ consists of $(A,B)\in  \gl(d,W)^2$ such that $A,B$ commute with $S_{\iota(c)}$ and $AB=C$. The surjectivity of $\bar p$ follows immediately ($A=C$, $B=\Id$). We need to show that, for a fixed $C$ commuting with $S_{\iota(c)}$, the variety of pairs $(A,B)$ of matrices commuting with $S_{\iota(c)}$ and satisfying $AB=C$ has dimension $d$.
Without loss of generality we may assume that there is only one $W$-orbit of the cocharacters $\alpha_i$. In this case $S_{\iota(c)}$ is a regular element of $\gl_d(\cx)$. Since the action of $\cx^\ast$ given by $\rho.(A,B,C)=(\rho A,\rho B,\rho^2 C)$ preserves the equation $AB=C$ and maps any $C$ to an arbitrarily small (in the analytic category) neighbourhood of $0$, the semicontinuity theorem implies that it is enough to show that the variety of pairs $(A,B)$ of matrices commuting with $S_{\iota(0)}$ and satisfying $AB=0$ has dimension $d$.
We can identify $A$ and $B$ with elements $a(x),b(x)$ of $\cx[x]/(x^d)$. We are therefore asking about the dimension of the variety of pairs $(a(x),b(x))$
of polynomials of degree at most $d-1$ such that $x^d|a(x)b(x)$. The irreducible components of this variety are given by $a(x)=a^\prime(x)x^e$, $b(x)=b^\prime(x)x^f$, where $e+f=d$. Every such component is an affine space of dimension $d$.
\end{proof}
\begin{corollary} The morphism $\bar p:W\text{-}\Hilb_\mu(\sP\times \h)\to  W\text{-}\Hilb_\pi(\cx^d\times \h)$ is faithfully flat.\label{Pflat}
\end{corollary}
\begin{proof} The above lemma and part (iv) of Proposition \ref{hooray} imply that the morphism $W\text{-}\Hilb_\mu(\sP\times \h)\to  \sT_{{\scriptscriptstyle T, W}}\times_{\h/W}W\text{-}\Hilb_\mu(\cx^{d}\times \h)$
is faithfully flat. The morphism $\sT_{{\scriptscriptstyle T, W}}\times_{\h/W}W\text{-}\Hilb_\mu(\cx^{d}\times \h)\to W\text{-}\Hilb_\mu(\cx^{d}\times \h)$ is a surjective equidimensional morphism between regular varieties. 
\end{proof}
Let $\sQ$ denote
the GIT quotient $\sP/\!\!/T^d$ and $\rho:\sP\to\sQ$ the natural projection. We continue denoting by $\mu$ denote the projection from $\sQ\times \h\to \h$ onto the second factor. The ring $\cx[\sQ]=\cx[\sP]^{T^d}$ has been described at the beginning of \S\ref{htv}. It is generated by $x_\xi$, $\xi\in \sX^\ast(T)$, and by $u_iv_i$, $i=1,\dots,d$, with relations \eqref{monomials}.  In particular, there are no relations between the $u_iv_i$, and, consequently, there is a surjective morphism $ \pi:\sQ\to \cx^d$ given by the embedding $\cx[u_1v_1,\dots,u_dv_d]\hookrightarrow \cx[\sQ]$. We denote also by $\pi$ the morphism $\pi\times \Id:\sQ\times \h\to \cx^d\times \h$.
\begin{proposition} $W\text{-}\Hilb_\mu(\sP\times \h)$ is isomorphic to the fibre product 
$$\begin{CD} W\text{-}\Hilb_\mu(\sP\times \h)& @>>> & W\text{-}\Hilb_\mu(\cx^{2d}\times \h)\\
@VV\bar\rho V & & @VV\bar p V\\ W\text{-}\Hilb_\mu(\sQ\times \h)& @>\bar\pi >> & W\text{-}\Hilb_\mu(\cx^{d}\times \h).
\end{CD}
$$\label{PQflat}
\end{proposition}
\begin{proof} Consider the fibres of $W\text{-}\Hilb_\mu(\sP\times \h)$ and of $W\text{-}\Hilb_\mu(\sQ\times \h)$ over $c\in \h/W$. They are isomorphic, respectively, to ${\rm Mor}(D_c,\sP)^W$ and ${\rm Mor}(D_c,\sQ)^W$. 
Given the description of $\cx[\sQ]=\cx[\sP]^{T^d}$ in \S\ref{htv}, we conclude  that the weight decomposition of $\cx[\sP]$ is given by
$$ \cx[\sP]=\bigoplus_{m\in \oZ^d} \Pi(m) \cx[\sQ],\enskip \text{where}\enskip \Pi(m)=\prod_{i=1}^d u_i^{(m_i)_-} v_i^{(m_i)_+},
$$
with the addition and multiplication compatible with this decomposition. If $m,m^\prime \in \oZ^d$, then 
$$\Pi(m)\Pi(m^\prime)=\Pi(m+m^\prime)\prod_{i=1}^d(u_iv_i)^{l_i}, \text{where}\enskip l_i=(m_i)_++(m_i^\prime)_+- (m_i+m_i^\prime)_+.$$
\par
Let $\varphi:D_c\to\sQ$ be a $W$-equivariant morphism, with the corresponding ring homomorphism 
$\varphi^\sharp: \cx[\sQ]\to R_c$, $R_c=\cx[D_c]$. It follows that $\varphi^\sharp=\tilde\varphi^\sharp\circ\rho^\sharp$ if and only if $\tilde \varphi^\sharp:\cx[\sP]\to R_c$ satisfies $\tilde\varphi^\sharp|_{\C[\sQ]}= \varphi^\sharp$ and $\tilde\varphi^\sharp(u_i)\tilde\varphi^\sharp(v_i)=\varphi^\sharp(u_iv_i)$, $i=1,\dots,d$. This means that the fibre of $\bar\rho$ over $\varphi$ is exactly the fibre of $\bar p$ over $\pi\circ \varphi$.
\end{proof}
\begin{corollary} The morphisms $\bar\rho:W\text{-}\Hilb_\mu(\sP\times \h)\to W\text{-}\Hilb_\mu(\sQ\times \h)$
and $\bar\pi:W\text{-}\Hilb_\mu(\sQ\times \h)\to W\text{-}\Hilb_\mu(\C^d\times \h)$ are faithfully flat.\label{Qflat}
\end{corollary}
\begin{proof} Since $\bar p$ is faithfully flat (Lemma \ref{pflat}), so is $\bar\rho$. The composition $\bar\pi\circ\bar\rho$ is equal to 
$\bar p:W\text{-}\Hilb_\mu(\sP\times \h)\to W\text{-}\Hilb_\mu(\cx^{d}\times \h)$, which is faithfully flat, owing to Corollary \ref{Pflat}. Since $\bar \rho$ is faithfully flat, so is $\bar\pi$.
\end{proof}
\begin{corollary} Let $\mu:X\to \h$ be a strongly $W$-invariant hypertoric variety. Then $\bar\mu:W\text{-}\Hilb_\mu(X)\to \h/W$ is faithfully flat.\label{flat2}
\end{corollary}
\begin{proof} We clearly have $X=\h\times_{(\cx^d\times \h)}(\sQ\times \h)$, where the map $\h\to \cx^d\times \h$ is 
$$z\longmapsto\Bigl(\sum_{i=1}^d(\langle z, \alpha_i\rangle e_i-\lambda_i), z\Bigr).$$
Since the functor $W\text{-}\Hilb$ commutes with base change \cite{Blume}, the same is true for  $W\text{-}\Hilb_\mu$. This, together with the functoriality of $W\text{-}\Hilb_\mu$ (Remark \ref{functorial}), shows that 
 $W\text{-}\Hilb_\mu(X)=\h/W\times_{W\text{-}\Hilb_\mu(\C^d\times \h)}W\text{-}\Hilb_\mu(\sQ\times \h)$. Since faithful flatness is preserved by base change, the claim follows.
\end{proof}
\begin{corollary} $\Hilb_\mu^W(X)= W\text{-}\Hilb_\mu(X) $.\label{W=W}
\end{corollary}
\begin{proof} The previous corollary implies that $\delta\in\cx[\h]^W$ does not divide zero in $\C[W\text{-}\Hilb_\mu(X)]$. 
\end{proof}
\subsection{A presymplectic quotient description\label{sq}}
Let us return to the description of $W\text{-}\Hilb_\mu(\cx^{2d}\times T\times\h)$ given in  Proposition \ref{hooray}(iv).
Recall that the universal centraliser $\sT_{{\scriptscriptstyle T, W}}=\{(c,g)\in \h/W\times G_{{\scriptscriptstyle T, W}};\, \ad_g(S_c)=S_c\}$ has a natural symplectic form $\omega_0=\bigl\langle  dgg^{-1}\wedge d\bigl(\Ad_g(S_c)\bigr)\bigr\rangle$ \cite{Bi1}. We obtain a natural closed $2$-form $\omega=\omega_0+\tr dA\wedge dB$ on  $W\text{-}\Hilb_\mu(\cx^{2d}\times T\times\h)$. 
The homomorphism $\alpha:T^d\to T$ induces a homomorphism $\bar\alpha:\sT_{{\scriptscriptstyle W}}^d\to \sT_{{\scriptscriptstyle T, W}}=W\text{-}\Hilb_\mu(T\times \h)$, where $\sT_{{\scriptscriptstyle W}}^d=
W\text{-}\Hilb_\mu(T^{d}\times\h)$ has been described in Proposition \ref{hooray}(iii).
The  group scheme $\sT_{{\scriptscriptstyle W}}^d$  acts on $W\text{-}\Hilb_\mu(\cx^{2d}\times T\times\h)$ via $(c,g,A,B)\mapsto (c,\bar\alpha(\gamma)g,\gamma A, \gamma^{-1} B)$, and hence it preserves the form
$\omega$. The form $\omega$ is not nondegenerate everywhere (this is easily checked for $d=2$), but we can nevertheless hope that a {\em presymplectic} quotient of $W\text{-}\Hilb_\mu(Y)$ by  $\sT_{{\scriptscriptstyle W}}^d$ exists - see \cite{Bot} for an introduction to presymplectic manifolds and quotients. 
We  compute the moment map.  The group ${\rm GL}(d,W)$ is the product of ${\rm GL}(E_m)$, and it is enough to compute the moment map in the case $s=1$, i.e.\ when the vectors $\alpha_i$ form a single $W$-orbit. 
We can identify $\Lie(\sT_{{\scriptscriptstyle W}}^d)$ with $\h/W\times \cx^d$ and view the action of $\sT_{{\scriptscriptstyle W}}^d$ as the fibrewise action of the additive group $\cx^d$. If $a=(a_0,\dots,a_{d-1})\in \cx^d$, then
$$a.(c,g,A,B)=\bigl(c,\bar\alpha(p)g,\, p^{-1}A, \,
pB\bigr), \enskip\text{where $p=\exp\bigl(-\textstyle\sum_{i=0}^{d-1} a_iS_{\iota(c)}^i\bigr)$}.$$
We easily compute the moment map with respect to the form $\omega=\omega_0+\tr dA\wedge dB$ as
 \begin{equation*} \rho(c,g,A,B)=\bigl(\tr (AB)-\tr S_{\iota(c)},\dots, \tr (S_{\iota(c)}^{d-1}AB)-\tr S_{\iota(c)}^{d}\bigr).\label{momentS}
 \end{equation*}
Hence the level set equation $\rho(c,g,A,B)=(-\lambda,-\lambda,\dots,-\lambda)$ is equivalent to $AB+\lambda= S_{\iota(c)}$ (cf.\ Example \ref{central}). The same holds now for any number $s$ of $W$-orbits.
\par
We now want to consider the presymplectic quotient of $W\text{-}\Hilb_\mu(\cx^{2d}\times T\times\h)$ by $\sT_{{\scriptscriptstyle W}}^d$.
There are actually two such quotients, depending on the order of operations: 
\begin{itemize}
\item[(i)] $\hat H^W(X)=\Spec\cx\bigl[W\text{-}\Hilb_\mu(Z_\lambda)\bigr]^{\sT_{{\scriptscriptstyle W}}^d}$;
\item[(ii)] the subscheme $H^W(X)$ of $\Spec\cx\bigl[W\text{-}\Hilb_\mu(\cx^{2d}\times T\times\h)]^{\sT_{{\scriptscriptstyle W}}^d}$ defined by equations $D+\lambda=S_{\iota(c)}$, where the entries of $D=AB$ are viewed as $ \sT_{{\scriptscriptstyle W}}^d$-invariant functions on
$W\text{-}\Hilb_\mu(\cx^{2d}\times T\times\h)$.
\end{itemize}

Apriori, these schemes could be  nonnoetherian.
\begin{theorem} For any strongly $W$-invariant hypertoric variety $X$ the presymplectic quotient $H^W(X)$ is canonically isomorphic to $ W\text{-}\Hilb_\mu(X)$. If none of the defining hyperplanes $H_k$, $k=1,\dots,d$, is contained in the union $\tilde\Delta\subset \h$  of reflection hyperplanes, then  $\hat H^W(X)$ is also isomorphic to $W\text{-}\Hilb_\mu(X)$. \label{pre-H^W}
\end{theorem}
\begin{proof} We shall apply  Proposition \ref{fingen}. For the first statement let $h=\prod_{i=1}^dv_iw_i$. Corollary \ref{Qflat} implies that $(\delta,h)$ is a regular sequence in $W\text{-}\Hilb_\mu(\sQ\times \h)$.
For the second statement let $h=\prod_{k=1}^d\bigl(\langle\alpha_k,z\rangle -\lambda_k\bigr)$. Corollary \ref{Pflat} guarantees that $(\delta,h)$ is regular in $W\text{-}\Hilb_\mu(Z_\lambda)=\Hilb_\mu^W(Z_\lambda)$. 
\end{proof}
\begin{example} Let $X$ be the  $A_1$-singularity $xy=z^2$ so that ${\oZ_2}\text{-}\Hilb_\mu( X)$ is the $D_2$-singularity (Example \ref{D_2}). The variety ${\oZ_2}\text{-}\Hilb_\mu(Z_0)$ is isomorphic to 
$$\bigl\{(c,g,A,B)\in \cx\times {\rm SL}_2(\cx)\times \bigl(\Mat_{2,2}(\cx)\bigr);\,  [g,S_c]=[A,S_c]=[B,S_c]=0,\; AB=S_c
\bigr\}.$$
If we write $A=a_0+a_1S_c$ and $B=b_0+b_1S_c$, then the equation $AB+S_c$ implies that $a_0b_0+a_1b_1c=0$, i.e.\ $c|a_0b_0$. This means that if we write $E=e_0+e_1S_c$, where $E=g^{-1}\bigl(AB_{\rm adj}\bigr)$, then $e_0$ is divisible by $c$ in $\cx\bigl[{\oZ_2}\text{-}\Hilb_\mu(Z_0)\bigr]$. Therefore the entries of $E$ do not generate $\cx\bigl[{\oZ_2}\text{-}\Hilb_\mu(Z_0)\bigr]^{\sT^{2}_{{\scriptscriptstyle \oZ_2}}}$ and hence $\hat H^{\oZ_2}(X)\neq H^{\oZ_2}(X)$ in this case. In fact, $\hat H^{\oZ_2}(X)$ is a $D_1$ surface.
\end{example}
We know from \S\ref{tori} that the transverse Hilbert scheme $\Hilb_\mu^W(X)= W\text{-}\Hilb_\mu(X) $ admits an action of the abelian group scheme
$\sT_{{\scriptscriptstyle T,W}}\simeq W\text{-}\Hilb_\pi(T\times \h)$. We deduce the following decomposition result:
\begin{corollary} Let $X$ be a strongly $W$-invariant hypertoric variety $X$ given as a symplectic quotient of $\cx^{2d}\times T\times \h$ by $T^d=(\cx^\ast)^d$, where $T^d$ acts on $T$ via a $W$-equivariant homomorphism $\alpha:\oZ^d\to \sX_\ast(T)$. Let $\oZ^d=\bigoplus_{m=1}^s\oZ^{d_m}$ be the decomposition into $W$-orbits and write, correspondingly, $\alpha=\oplus_{m=1}^s\alpha^m$. For each $m=1,\dots,s$, let $X_m$ be the strongly   $W$-invariant hypertoric variety obtained from the homomorphism $\alpha^m$. Then
$$\cx[W\text{-}\Hilb_\mu( X)]\simeq \cx[W\text{-}\Hilb_\mu( X_1)\times_{\h/W}\cdots
\times_{\h/W} W\text{-}\Hilb_\mu( X_s)]^{\sT_0},
$$ where
$$\sT_0=\{(\tau_1,\dots,\tau_s)\in \sT_{{\scriptscriptstyle T, W}}\times_{\h/W}\times \cdots \times_{\h/W}\sT_{{\scriptscriptstyle T, W}};\, \tau_1\cdots\tau_s=1\}.$$\label{Hilb-decompose}
\end{corollary}
In other words, if $X$ is a symplectic quotient of $\prod_{m=1}^s X_m$ by $\{(t_1,\dots,t_s)\in T^s;\prod t_m=1\}$, then  $W\text{-}\Hilb_\mu( X)$  is the quotient  of the fibred product of transverse $W$-Hilbert schemes of $X_1,\dots, X_s$ by $\sT_0$.
\begin{proof} According to Theorem \ref{pre-H^W},
$$\cx[\Hilb_\mu^W(X_m)]\simeq \cx\bigl[W\text{-}\Hilb_\mu(\cx^{2d_m}\times T\times\h)]^{\sT_{{\scriptscriptstyle W}}^{d_m}}/I_m,\quad m=1,\dots,s,$$
where the ideal $I_m$ is generated by $\sT_{{\scriptscriptstyle T,W}}$-invariant functions.
Let us write $H_m=\Spec\cx\bigl[W\text{-}\Hilb_\mu(\cx^{2d_m}\times T\times\h)]^{\sT_{{\scriptscriptstyle W}}^{d_m}}$. 
Owing to Proposition \ref{hooray}(ii), we have:
$$W\text{-}\Hilb_\mu(\cx^{2d_m}\times T\times\h)\simeq \sT_{{\scriptscriptstyle T,W}}\times_{\h/W}  W\text{-}\Hilb_\mu(\cx^{2d_m}\times\h).$$
Since the actions $\sT_{{\scriptscriptstyle T, W}}$ and $\sT_{{\scriptscriptstyle W}}^{d_m}$ commute, we conclude, after applying Corollary \ref{sTsquared}, that 
$$ \cx[W\text{-}\Hilb_\mu(\cx^{2d}\times T\times\h)]^{\sT_{{\scriptscriptstyle W}}^{d}}\simeq \cx[H_1\times_{\h/W}\cdots\times_{\h/w} H_s]^{\sT_0}.
$$
The claim follows since   the ideal $I$ of $W\text{-}\Hilb_\mu(X) $ in $\cx[W\text{-}\Hilb_\mu(\cx^{2d}\times T\times\h)]^{\sT_{{\scriptscriptstyle W}}^{d}}$ as well as the ideals $I_m$, $m=1,\dots,s$, are generated by $\sT_{{\scriptscriptstyle T,W}}$-invariant functions, and hence:
$$I\simeq \bigl(I_1\otimes_{\cx[\h]^W}\cdots \otimes_{\cx[\h]^W} I_s\bigr)^{\sT_0}.$$
\end{proof}

\subsection{Normality of $\Hilb_\mu^W(X)$\label{normality}}
We shall now prove that $\Hilb_\mu^W(X)=W\text{-}\Hilb_\mu(X) $ is a normal variety. We use the following idea of Braverman, Finkelberg, and Nakajima \cite[Lemma 6.13]{BFN}: to prove the normality of an affine scheme $Y$ it is enough to find an open normal subset $U$ with $\codim (Y\backslash U)\geq 2$ and such that any regular function on $U$ extends to $Y$.
\par
Let $F\in \h/W$ be a subscheme of $V(\delta)\subset \h/W$ of codimension $\geq 1$ in $V(\delta)$. Let $M=\Hilb_\mu^W(X)$ and $M^\bullet=\bar\mu^{-1}\bigl((\h/W)\backslash F\bigr)$.
\begin{proposition} The subset $M\backslash M^\bullet$ has codimension at least $2$ and the embedding $j:M^\bullet\hookrightarrow M$ induces an isomorphism $ \sO_{M}\to j_\ast\sO_{M^\bullet}$.\label{S_2}\end{proposition}
\begin{proof} 
Let $Z=M\backslash M^\bullet=\bar\mu^{-1}(F)$.
Corollary \ref{flat2} implies that $\codim Z\geq 2$. Since $\h/W$ is regular, in particular Cohen-Macaulay, and since $\bar\mu$ is flat, $ \sO_{M}$ has depth $\geq 2$ along $Z$. The claim follows from Theorem 3.8 and Proposition 1.11 in \cite{Groth} (or \cite[Ex. III.3.5]{Hart}).
\end{proof}
\begin{theorem} Let $X$ be a strongly $W$-invariant hypertoric variety. Then $\Hilb_\mu^W(X)$ is a normal affine variety.\label{normal}
\end{theorem}
\begin{proof} Integrality of $\Hilb_\mu^W(X)$ follows from the integrality $X$. The previous proposition and the remark at the beginning of this subsection imply that  it is  enough to prove the normality of an open subset $U$ of $\Hilb_\mu^W(X)$ of the form $\bar\mu^{-1} \bigl((\h/W)\backslash F\bigr)$ where $F$ is a closed subset of $V(\delta)$ of codimension $1$.
We choose $F$ to be the image of $\tilde F\subset \h$ consisting of
\begin{itemize}
\item[1)] intersections of two or more reflection hyperplanes;
\item[2)] transverse intersections of reflection hyperplanes with one of the hyperplanes $H_k=\{z\in \h; \langle\alpha_k,z\rangle =\lambda_k\}$, $k=1,\dots,d$.
\end{itemize}
Let $x$ be a point of $U$. If $\bar\mu(x)\not\in V(\delta)$, then a neighbourhood of $x$ is isomorphic to a corresponding  neighborhood in $X$ (cf.\ Proposition \ref{local}) and hence it is normal. The other possibility is $\bar\mu(x)$ lying on the image of exactly one reflection hyperplane. In this case, given Proposition \ref{local}, we need to prove the normality of $\Hilb_\mu^{\oZ_2}(X)$ for $\oZ_2$-invariant hypertoric varieties such that all defining $H_k=\{z\in \h; \langle\alpha_k,z\rangle =\lambda_k\}$, $k=1,\dots,d$, coincide with the hyperplane fixed by $\oZ_2$. In particular, all scalars $\lambda_k$ are  equal to $0$.
The reductive group $G=G_{T,\oZ_2}$,  associated to $T$ with a linear $\oZ_2$ action in Example \ref{sTG}, has semisimple rank $1$, and hence is isomorphic to ${\rm SL}_2(\cx)\times (\cx^\ast)^{r-1}$ or to ${\rm GL}_2(\cx)\times (\cx^\ast)^{r-2}$ (cf.\ the proof of Theorem \ref{Slodowy}). Therefore the torus $T$ splits as $T_1\times (\cx^\ast)^{r-1}$ or $T_2\times (\cx^\ast)^{r-2}$, where $T_1$ (resp.\ $T_2$) is the maximal torus of ${\rm SL}_2(\cx)$ (resp.\ of ${\rm GL}_2(\cx)$), and $W=\oZ_2$ acts trivially on the second factor. The assumptions that the hyperplanes $H_k$ coincide with the reflection hyperplane implies that the vectors $\alpha_k$ belong to $\Lie(T_1)$ (resp.\ $\Lie(T_2)$), and hence $X$ is isomorphic to the product of a hypertoric variety $X_0$ with structure torus $T_1$ (resp.\ $T_2$) and $T(\cx^\ast)^{r-1}$ (resp.\ $T(\cx^\ast)^{r-2}$). The Weyl group $\oZ_2$ acts trivially on the second factor and hence  $\Hilb_\mu^{\oZ_2}(X)\simeq  \Hilb_\mu^{\oZ_2}(X_0)\times T(\cx^\ast)^{s}$, $s=r-1$ or $s=r-2$. Therefore it is enough to prove the normality of $\Hilb_\mu^{\oZ_2}(X_0)$.
\par
In the case of $T_1$, this follows from Example \ref{D_2}, after noting that $X_0$ in this case must necessarily be the $A_{N-1}$-singularity with $N$ even. 
It remains to consider $\oZ_2$-invariant hypertoric varieties with structure torus equal to the maximal torus of ${\rm GL}_2(\cx)$, and such that all the hyperplanes $H_k=\{z\in \h; \langle\alpha_k,z\rangle =\lambda_k\}$, $k=1,\dots,d$, coincide with the reflection hyperplane. Hence, if we identify $\h$ with $\cx^2$ so that $\oZ_2$ acts via $(z_1,z_2)\to (z_2,z_1)$ and $\sX_\ast(T)$ is the standard lattice $\oZ\times \oZ$, then the vectors $\alpha_k$ are of the the form $(m_k,-m_k)$, $k=1,\dots,d$, and all scalars $\lambda_k$ are equal to $0$. Observe that $X_0$ is isomorphic to a hypertoric variety determined by the single cocharacter $\alpha=(N,-N)$, $N=\sum_{i=1}^d |m_i|$ (and the corresponding scalar $\lambda$ equal to $0$). Since $X_0$ is strongly $W$-invariant, $N$ is even.
Consider the exact sequence 
 \begin{equation} 1\to\Gamma\longrightarrow \tilde T\stackrel{\phi}\longrightarrow T\to 1,\quad \tilde T=T_1\times\cx^\ast, \quad \phi(t,s)=(t^{-1}s,ts),\label{extension}\end{equation}
 where $\Gamma=\oZ_2$. The vector $\alpha$ belong to $\sX_\ast(\tilde T)$, and hence, as observed in Remark \ref{finiteq}, $X_0$ is the quotient by $\Gamma$ of the hypertoric variety $\tilde X_0$ with structure torus $\tilde T$ and the same vector $\alpha$. It follows that $\tilde X_0$ is the product of the $A_{N-1}$-singularity and of $\cx\times \cx^\ast$.  The action of $W=\Z_2$ lifts to an action on $\tilde{X}_0$, in its standard way on the $A_{N-1}$-singularity and trivially on the second factor $\C\times \C^\ast$. It follows that $\Hilb^{\Z_2}_\mu (\tilde{X}_0)$ is the product of the $D_{\frac{1}{2}N+1}$-singularity and $\C\times \C^\ast$ and, since $\Gamma$ acts freely on $\tilde{X}_0$, $\Hilb^{\Z_2}_\mu (X_0)\simeq \Hilb^{\Z_2}_\mu (\tilde{X}_0)/\Gamma$. We conclude that $\Hilb^{\Z_2}_\mu (X_0)$ is normal in this case.
 This finishes the proof of the theorem.
\end{proof}
The above proof yields also the following generalisation of a result of Teleman \cite[Theorem 6]{Tel1}:
\begin{proposition} Let $X$ be a strongly $W$-invariant hypertoric variety corresponding to an integral multiarrangement $\{\sA,m\}$. Let $(\sA^\prime,m^\prime)$ be the multiarrangement obtained by adding to $\{\sA,m\}$ all reflection hyperplanes with multiplicity two, and let $X^\prime$ be the corresponding hypertoric variety. Then $X^\prime$ is also strongly $W$-invariant hypertoric and $\Hilb_\mu^W(X^\prime)\simeq X/W$.\label{Tel/W}
\end{proposition}
\begin{proof} 
Both $\Hilb_\mu^W(X^\prime)$ and $X/W$ are normal and equipped with morphisms $\bar\mu^\prime,\bar\mu$ to $\h/W$. If we consider the same $F\subset \h/W$ as in the above proof, then the codimensions of complements of $(\bar\mu^\prime)^{-1}(\h/W\backslash F) ,(\bar\mu)^{-1}(\h/W\backslash F)$ are equal to $2$ and, therefore, it is enough to show that the two varieties are isomorphic over $\h/W\backslash F$. The above proof reduces this to the case when $X$ is a $\Z_2$-invariant $A_{2k-1}$-surface given by the equation $xy=p(z^2)$. It follows that $X^\prime$ is an $A_{2k+1}$-surface with equation $xy=z^2p(z^2)$. Both $\Hilb_\mu^{\Z_2}(X^\prime)$ and $X/Z_2$ are then isomorphic to the $D_{k+2}$-surface with equation $a^2-b^2c=cp(c)$.
\end{proof}
\begin{example} The product $\prod_n X$ of $n$ copies of an $A_k$-surface $X$ ($k\geq -1$ with $A_{-1}=\C\times \C^\ast$ and $A_0=\C^2$)  is a strongly $\Sigma_n$-invariant hypertoric variety. Let $Y$ be the hypertoric variety obtained by adding  all reflection hyperplanes $z_i=z_j$ with multiplicity $2$ to the multiarrangement of  $\prod_n X$. According to the above proposition,  $\Hilb_\mu^{\Sigma_n}(Y)\simeq S^n X$. If we add instead the hyperplanes $z_i-z_j=\lambda\neq 0$ for all $i\neq j$, we shall obtain a (strongly) $\Sigma_n$-invariant deformation $Y_\lambda$ of $Y$, and consequently a deformation $\Hilb_\mu^{\Sigma_n}(Y_\lambda)$ of $S^n(X)$. If $X$ is regular, then $Y_\lambda$ is regular for generic $\lambda$ \cite[Thm. 3.2]{BD}, and, consequently $\Hilb_\mu^{\Sigma_n}(Y_\lambda)$ is also regular. We expect that $\Hilb_\mu^{\Sigma_n}(Y_\lambda)$ is isomorphic to the Hilbert scheme of $n$ points on $X$ with respect to a generic complex structure of the hyperk\"ahler structure of $\Hilb^{[n]}(X)$. 
\end{example}
\begin{remark} In the case of an arbitrary $W$-invariant hypertoric variety we do not have a proof of the flatness of the induced morphism $\bar{\mu}:\Hilb_\mu^W(X)\to \h/W$, and, consequently, we do not know whether  Proposition \ref{S_2} is still true. If we assume this, however, then the remainder of the proof goes through (i.e.\ we can still show that the open set $U$ is normal) with the following modification: in the case $T=T_2$ we cannot assume that $N$ is even. The variety $X_0$ can be one of those in Example \ref{bastard}. The action of $W$ on this $X_0$ is free, and therefore $\Hilb^{\Z_2}_\mu (X_0)$ is an open subscheme of $X_0/W$, hence normal.
We remark that  Proposition \ref{S_2} holds as soon as the structure sheaf of $\Hilb_\mu^W(X)$ has depth $\geq 2$ along $\bar\mu^{-1}(F)$. 
\label{notstrong}
\end{remark}

\subsection{Symplectic structure \label{symplstr}}

A hypertoric variety $X$ is symplectic in the sense of Beauville \cite{Beau}, i.e.\ it has a symplectic form $\omega$ on the smooth locus $X^{sm}$, which extends to a closed $2$-form on any resolution of $X$. If $X$ is  $W$-invariant hypertoric, then   $\omega$ is $W$-invariant, and hence  $\omega$ descends to a symplectic form $\bar\omega$ on $X^0/W$, where $X^0$ is the subset of $X^{sm}$ on which $W$ acts freely. The intersection of $X^0/W$ with $\Hilb_\mu^W(X)$ is an open dense subset of the latter and the proof of Theorem \ref{normal} shows that $\bar\omega$ extends to the smooth part of the subset $U$ defined there. If $X$ is strongly $W$-invariant or  the depth of $\sO_{\Hilb_\mu^W(X)}$ along the complement of $U$ is at least $2$, then $\Hilb_\mu^W(X)$ is normal (Remark \ref{notstrong}) and the codimension  of the complement of $U$ is $\geq 2$. Therefore $\bar\omega$ extends to a symplectic form on the smooth locus of  $\Hilb_\mu^W(X)$. In the terminology of Beauville \cite{Beau}, $\Hilb_\mu^W(X)$ has a {\em symplectic form}, i.e.\ a closed reflexive $2$-form which is nondegenerate at smooth points. It follows that $\Hilb_\mu^W(X)$ is a Poisson variety. Moreover, the natural morphism $\Hilb_\mu^W(X)\to X/W$ is Poisson, and Lemma 2.1 in \cite{Bell} implies that  $\Hilb_\mu^W(X)$ has symplectic singularities.

\medskip

According to \S\ref{tori}, the transverse $W$-Hilbert scheme $\Hilb_\mu^W(X)$ of a $W$-invariant hypertoric variety $X$ is equipped with a (generically free) action of the group scheme  $\sT_{{\scriptscriptstyle T,W}}\simeq W\text{-}\Hilb_\pi(T\times \h^\ast)$. If $X$ is strongly $W$-invariant, hence Poisson, then this action is Hamiltonian with moment map $\bar\mu$.  In the category of affine Hamiltonian $\sT_{{\scriptscriptstyle T,W}}$-schemes, we can define the {\em symplectic quotient} of $\bar\mu:Y\to \h/W$ as $\Spec\cx[\bar\mu^{-1}(0)]^{\sT_{{\scriptscriptstyle T,W}}}$ (this is, {\em apriori}, different from the Poisson quotient equal to $\bar\mu^{-1}(0)$ on $\Spec\cx[X]^{\sT_{{\scriptscriptstyle T,W}}}$).
\par
Corollary \ref{Hilb-decompose} implies:
\begin{proposition} Transverse Hilbert schemes of strongly $W$-invariant hypertoric varieties form an abelian monoid with product given by the symplectic quotient by the anti-diagonal action of $\sT_{{\scriptscriptstyle T,W}}$. This monoid is isomorphic to  the monoid of integral multiarrangements in $\h\simeq \h^\ast$ satisfying the condition of Proposition \ref{self-dual2}.\qed \label{sT-monoid}
\end{proposition}
\begin{remark} If $X$ is not strongly invariant, we can still  view 
$\bar\mu$ as an abstract moment map \cite{Kar} for the action of $\sT_{{\scriptscriptstyle T,W}}$ and  form quotients. We expect that the above proposition remains true for the monoid  of $W$-invariant hypertoric varieties and the corresponding monoid of integral multiarrangements in $\h\simeq \h^\ast$ for which the obstruction $\o(\sA,m)$ (Definition \ref{obstructiono}) vanishes. \label{action_of_T}
\end{remark}

\section{Coulomb branches as $W$-Hilbert schemes of hypertoric varieties
\label{Coul}}

As recalled in the introduction a Coulomb branch is a singular hyperk\"ahler space which should arise as a branch of the moduli space of vacua in a $3$-dimensional quantum gauge theory with $\mathcal{N}=4$ supersymmetries. Such a gauge theory is
associated to a compact connected Lie group and its quaternionic representation. As in the introduction, we denote by $K$ the Langlands dual of the gauge group.  In the case when the quaternionic representation of $K^\vee$ is of the form $V\oplus V^\ast$, where $V$ is a complex representation, Braverman, Finkelberg, and Nakajima \cite{BFN} constructed the Coulomb branches as Poisson affine varieties via the equivariant Borel-Moore homology of certain moduli stack. We are now going to prove that these varieties are either $W$-Hilbert schemes of hypertoric varieties, or Poisson quotients of such $W$-Hilbert schemes by algebraic tori. In the next subsection we are also going to discuss the case of an arbitrary quaternionic representation of $K^\vee$. 
\par
Let $G$ be a complex connected reductive Lie group and $V$ a representation of its Langlands dual $G^\vee$. Let $T,T^\vee$ denote maximal tori, $\h,\h^\vee$ the corresponding Cartan algebras, and $W$ the Weyl group (of both $G$ and $G^\vee$). A character of $T^\vee$ is a cocharacter of $T$.
We can  form a strongly $W$-invariant
hypertoric variety $X(G,V)$ by taking all nonzero weights of the representation $V$ as the defining vectors $\alpha_k\in \sX_\ast(T)$ with all scalars $\lambda_k$ equal to $0$. We shall see that the transverse $W$-Hilbert scheme $\Hilb^W_\mu\bigl(X(G,V)\bigr)$ is the Coulomb branch associated to $(G,V)$ in most cases, but not in all. A counterexample is provided by $G={\rm SO}_3(\cx)$ and $V$  the standard $2$-dimensional representation of $G^\vee={\rm SL}_2(\cx)$. The above hypertoric variety $X(G,V)$ is then the $A_{1}$-singularity, and $\Hilb^W_\mu\bigl(X(G,V)\bigr)$ is the $D_{2}$-singularity (cf.\ Example \ref{D_2}). The Coulomb branch $\sM_C(G^\vee,V)$ is, however, the $D_{1}$-surface, see \cite[Lemma 6.9]{BFN}.
\par
The full answer is given by:
\begin{theorem} Let $G$ be a reductive algebraic group and $V$ a representation of its Langlands dual $G^\vee$. Decompose $G$ as $G=G^\prime\times \prod_{i=1}^N {\rm SO}_{2k_i+1}(\cx)$ with $k_i\geq 1$ and $G^\prime$ having no direct ${\rm SO}_{2k+1}$-factors.
\par
Then $V$ is  also canonically a representation of the Langlands dual of $\tilde G=G^\prime\times \prod_{i=1}^N {\rm Spin}^c_{2k_i+1}(\cx)$, and
the Coulomb branch $\sM_C(G^\vee,V)$ is naturally isomorphic\footnote{As affine symplectic varieties.} to the symplectic quotient\footnote{i.e.\ the GIT-quotient $\nu^{-1}(0)/\!\!/(\cx^\ast)^N$, where $\nu$ is the composition of $\bar\mu:\Hilb_\mu^W\bigl(X(\tilde G,V))\to \h^\prime/W$ and the projection $\h^\prime\to \Lie((\cx^\ast)^N)$.\label{centreq}} of $\Hilb_\mu^W\bigl(X(\tilde G,V))$ by the centre $(\cx^\ast)^N$ of the factor $\prod_{i=1}^N {\rm Spin}^c_{2k_i+1}(\cx)$.
\label{C=W}\end{theorem}
The remainder of the subsection is devoted to a proof of this theorem.
\par
The key to identifying Coulomb branches is Theorem 5.26 in \cite{BFN}. We restate it here in a form more aligned with our notation (in particular their $G$ is our $G^\vee$). Let $P$ be a collection of hyperplanes in $\h$ consisting of reflection hyperplanes of $W$ and of hyperplanes of the form $\langle \alpha,z\rangle=0$ for any nonzero weight of $V$ viewed as a cocharacter of $T$. Let $h^\circ\subset \h$ be the complement of all  hyperplanes in $P$ and $\h^\bullet$ the complement of the intersection of two or more  hyperplanes in $P$. Braverman, Finkelberg, and Nakajima prove that an affine scheme $\Pi:\sM\to \h/W$ is the Coulomb branch associated to $(G,V)$ provided:
\begin{itemize}
\item[(i)] the natural embedding $j:\h^\bullet/W\to \h/W$ induces an isomorphism $ \Pi_\ast\sO_{\sM}\to j_\ast\Pi_\ast\sO_{\sM^\bullet}$, where $\sM^\bullet=\Pi^{-1}(\h^\bullet/W)$;
\item[(ii)] Over $\h^\circ/W$ $\sM$ is isomorphic to $(T\times \h^\circ)/W$;
\item[(iii)] For any $h\in \h^\bullet\backslash\h^\circ$ there is a neighbourhood $U$ of $\pi(h)\in \h^\bullet/W$ and an isomorphism between $\Pi^{-1}(U)$ and a corresponding subset in the Coulomb branch $\sM_C(Z_h,V^h)$, where $Z_h$ is the centraliser of $h$ in $G^\vee$ and $V^h\subset V$ is the subspace of $h$-invariants. Moreover, this isomorphism restricts to the one in (ii) over $U\cap \h^\circ/W$.
\end{itemize}
Consider now $\sM=\Hilb^W_\mu(X(G,V))$, where $X(G,V)$ has been defined above. It satisfies (i) owing to Theorem \ref{S_2}.
Condition (ii) is also satisfied, since over $\h^\circ/W$ $\:H^W\bigl(X(G,V)\bigr)$ is isomorphic to $X(G,V)/W$.
Finally, condition (iii) is satisfied for all $h\in \h^\bullet\backslash\h^\circ$ {\em unless} $h$ belongs to a reflection hyperplane and $Z_h\simeq {\rm SL}_2(\cx)\times (\cx)^{r}$. This follows from the proof of Theorem \ref{normal} and from the description of $\sM_C(Z_h,V^h)$ in \cite{BFN} (Lemma 6.9 and the proof of Proposition 6.12). We therefore have to investigate this remaining case. The first step is to determine groups $G^\vee$ for which the centraliser of an $h$ lying on a root hyperplane is isomorphic to ${\rm SL}_2(\cx)\times (\cx)^{r}$. Once again,   \cite[Prop. 3.2]{JMO} (quoted above as Proposition \ref{JMO}) provides an answer. It implies namely (with the notation from that proposition)  that the fundamental group of the centraliser $Z_\theta$ of a coroot $v_\theta$ in $G=K^\cx$   is equal to $\oZ_2$ if $\theta$ is a short root in an ${\rm SO}_{2n+1}(\cx)$ factor, and to $\{1\}$ otherwise. This follows from the following well-known fact:
$$\pi_1([Z_\theta,Z_\theta])=\bigl(\sX_\ast(T)\cap \Ker(\theta)^\perp)/\langle v_\theta\rangle.$$
Therefore $Z_\theta$ is isomorphic to ${\rm PGL}_2(\cx)\times (\cx^\ast)^r$ if and only if $\theta$ is such a short root. Consequently, the centraliser of a coroot $h$ in $G^\vee$ is isomorphic to ${\rm SL}_2(\cx)\times (\cx)^{r}$ if and only if it is a short coroot of ${\rm Sp}_{2n}(\cx)$. 
\par
We are thus reduced to the case when $G$ contains a direct ${\rm SO}_{2k+1}(\cx)$-factor. Let us write $G=G^\prime\times\prod_{i=1}^N {\rm SO}_{2k_i+1}(\cx)$, where $G^\prime$ has no ${\rm SO}_{2k+1}(\cx)$-factors. 
As explained in the introduction, we replace each ${\rm SO}_{2k_i+1}(\cx)$-factor with  ${\rm Spin}^c_{2k_i+1}(\cx)= {\rm Spin}_{2k_i+1}(\cx)\times_{{\scriptscriptstyle \oZ_2}}\cx^\ast$, and denote the resulting group by $G^\prime$. The Langlands dual of ${\rm Spin}_{2k+1}(\cx)$ is the {\em conformal symplectic group} ${\rm Sp}^c_{2k}(\cx)= {\rm Sp}_{2k}(\cx)\times_{{\scriptscriptstyle \oZ_2}}\cx^\ast$. If $\omega$ is the symplectic form used to define ${\rm Sp}_{2k}(\cx)\subset {\rm GL}_{2k}(\cx)$, then 
$${\rm Sp}^c_{2k}(\cx)=\bigl\{A\in {\rm GL}_{2k}(\cx);\: \exists {\lambda\in \cx^\ast}\; \forall {v,w\in \cx^{2k}}\;\:\omega(Av,Aw)=\lambda\,\omega(v,w)\bigr\}.$$
\begin{lemma} Any representation of ${\rm Sp}_{2k}(\cx)$ can be canonically lifted to a representation of ${\rm Sp}_{2k}^c(\cx)$. \label{liftsp}
\end{lemma}
\begin{proof} Let $\rho:{\rm Sp}_{2k}(\cx)\to \End(V)$ be an irreducible representation. If $\rho(-1)={\rm Id}_V$, then $\rho$ is a representation of ${\rm PSp}_{2k}(\cx)$, hence of ${\rm Sp}_{2k}^c(\cx)$. If $\rho(-1)=-{\rm Id}_V$, then the representation of ${\rm Sp}_{2k}(\cx)\times\cx^\ast$ given by $(g,\lambda).v=\lambda\rho(g)v$ descends to a representation of ${\rm Sp}_{2k}^c(\cx)$.
\end{proof}
The representation $V$ of $G^\vee$ can be decomposed as a direct sum of representations of the form $\bigotimes_{i=0}^N V_i$, where $V_0$ is a representation of $(G^\prime)^\vee$ and $V_i$, $i\geq 1$, is a representation of ${\rm Sp}_{2k_i+1}(\cx)$. We can apply the above lemma to each factor $V_i$, $i=1,\dots, N$, and conclude that $V$ is a representation of $\tilde G^\vee$. 
We can therefore consider the strongly $W$-invariant hypertoric variety $X(\tilde G,V)$, and the claim 
follows from  \cite[Prop. 3.18]{BFN} applied to the short exact sequence 
$$1\to G^\vee\to \tilde G^\vee\to (\cx^\ast)^N\to 1$$
of reductive groups.
 Theorem \ref{C=W} has been proved.
\begin{remark} Proposition 3.18 in \cite{BFN} implies that we can use any lift of a representation of ${\rm Sp}_{2k}(\cx)$ to a representation of ${\rm Sp}_{2k}^c(\cx)$ (i.e.\ tensor the canonical lift with a character). In addition, in the case when the representation $V$ of ${\rm Sp}_{2k}(\cx)$ is a lift of a representation of ${\rm PSp}_{2k}(\cx)$, then  $\sM_C({\rm Sp}_{2k}(\cx),V)\simeq \Hilb^W_\mu\bigl(X({\rm Spin}_{2k+1}(\cx),V)\bigr)/\oZ_2$  \cite[\S 3.vii.c]{BFN} (or Remark \ref{finiteq}).
\end{remark}
\begin{remark} The same Proposition 3.18 in \cite{BFN} implies that, instead of replacing each direct factor ${\rm SO}_{2k_i+1}(\cx)$ with ${\rm Spin}_{2k_i+1}^c(\cx)$, we can replace their product with  $\widehat G_0=\C^\ast\times_{\oZ_2^N}\prod_{i=1}^N {\rm Spin}_{2k_i+1}(\cx)$, i.e.\ with $\cx^\ast \times_{\oZ_2}  (G_0^\vee/\oZ_2)^\vee$, where $G_0^\vee=\prod_{i=1}^N{\rm Sp}_{2k_i}(\cx)$, and $\oZ_2$ is embedded diagonally in the centre of $G_0^\vee$.
We obtain a new group $\widehat G=G^\prime\times \widehat G_0$ and a $W$-invariant hypertoric variety $\Hilb_\mu^W\bigl(X(\widehat G,V))$, and $\sM_C(G^\vee,V)$ is its symplectic quotient by a single $\cx^\ast$.\label{hatG}
\end{remark}

\subsection{Coulomb branches for general quaternionic representations\label{genquat}}

We now consider an arbitrary quaternionic representation $\mathbb{V}$ of $G^\vee$. The problem of defining the corresponding Coulomb branch  has been considered in \cite{BDFRT} and in \cite{Tel2}.
Our Theorem \ref{C=W} in the case $\mathbb{V}=V\oplus V^\ast$   can be used as a guide to give another definition of the Coulomb branches in this general case, as we now explain. 
\par
Let $\sX_\ast(\oV)\subset \sX_\ast(T)$ denote the set of nonzero weights of  $\mathbb{V}$. It is $W$-invariant and,
since $\mathbb{V}$ is self-dual, invariant under the involution $\alpha\mapsto -\alpha$. For every $\alpha\in \sX_\ast(\oV)$, let $d(\alpha)$ be its divisibility in $\sX_\ast(T)$, i.e.\ $\alpha=d(\alpha)\omega$, where $\omega$ is primitive  in $\sX_\ast(T)$.
We define a $W$-invariant integral multiarrangement  $(\sA,m)$ of hyperplanes in $\h^\ast$ as follows:
\begin{itemize}
\item $H\in \sA$ if and only if $H^\perp\cap \sX_\ast(\oV)\neq\emptyset$, i.e.\ there exists $\alpha\in \sX_\ast(\oV)$ such that $H=\{z;\langle\alpha,z\rangle=0\}$;
\item $m(H)=\frac{1}{2}\sum \{d(\alpha);\alpha\in H^\perp\cap \sX_\ast(\oV)\}$.
\end{itemize}
As a replacement for $X(G,V)$, we can then consider the hypertoric variety $X(G,\frac{1}{2}\oV)$ corresponding to this multiarrangement.  In the case when $X(G,\frac{1}{2}\oV)$ is $W$-invariant and $G$ does not have direct ${\rm SO}_{2k+1}(\cx)$-factors, we  {\em define} the Coulomb branch corresponding to $\mathbb{V}$ to be $\Hilb^W_\mu (X(G,\frac{1}{2}\oV))$. 
\par
 In the case when 
$G$ does have ${\rm SO}_{2k+1}(\cx)$-factors, however,  we cannot simply follow the recipe of Theorem \ref{C=W}, 
since the lift of the representation $\oV$ to a representation of $\tilde G$  will, in general, no longer be quaternionic. Nonetheless, the construction of Theorem \ref{C=W} would still work if we could 
 construct  a $W$-invariantly hypertoric variety $\tilde X=X(\tilde G,\frac{1}{2}\oV)$ with structure torus equal to the maximal torus of $\tilde G$ (or the group $\widehat G$ of Remark \ref{hatG}), and such that $X(G,\frac{1}{2}\oV)$ is the symplectic quotient of $\tilde X$ by the centre $(\cx^\ast)^N$ of $\prod_{i=1}^N {\rm Spin}^c_{2k_i+1}(\cx)$.
 This is carried out in Appendix \ref{appendix:ht-lifts} for a large class of representations, which is, however, strictly smaller than those with $W$-invariant $X(G,\frac{1}{2}\oV)$. We therefore propose the following alternative construction of Coulomb branches in the case of ${\rm SO}_{2k+1}(\cx)$-factors.
 \par
 Let $G=G_0\times G^\prime$, where $ G_0=\prod_{i=1}^N{\rm SO}_{2k_i+1}(\cx)$ and $G^\prime$ has no direct ${\rm SO}_{2k+1}$-factors.
 The proof of Theorem \ref{C=W} shows that the passage from $ \Hilb^W_\mu (X(G,V))$ to $ \Hilb^W_\mu (X(\tilde G,V))$ amounts to modifying local models of $ \Hilb^W_\mu (X(G,V))$ along  the images in $\h/W$ of reflection hyperplanes corresponding to short roots of $G_0$ (and only along those).
 These local models are products of $D_k$-singularities and $T^\ast(\C^\ast)^{r-1}$, $k\geq 1$, and $ \Hilb^W_\mu (X(\tilde G,V))$ replaces the $D_k$-singularity with the $D_{k-1}$-singularity (these are actually regular for $k=0,1$).
 As shown in Example \ref{D_k-D_{k-1}}, the passage from $D_k$ to $D_{k-1}$ can be realised as a transverse blow-up of the $D_k$-surface.
 Let now $X(G,\frac{1}{2}\oV)$ be a $W$-invariant hypertoric variety corresponding to a quaternionic representation $\oV$ of $G^\vee$.
 Along the image $\bar\mu(H)$ in $\h/W$ of a reflection hyperplane $H$ corresponding to a short root of $G_0$ (and away from the subset $F$ defined in the proof of Theorem \ref{normal}), the local model  $ \Hilb^W_\mu (X(G,\frac{1}{2}\oV))$ is again the product of $T^\ast(\C^\ast)^{r-1}$ and a $D_k$-singularity, $k\geq 1$. The subscheme of the $D_k$-singularity, along which we have to blow it up in order to obtain the $D_{k-1}$-singularity, defines a locally closed subscheme of $(\bar\mu)^{-1}(V(\delta_0)\setminus F)$, where $\delta_0$ is the polynomial in $\C[\h^\ast]^W$ describing the union of reflection hyperplanes corresponding to short roots of $G_0$. Let $Z^\circ$ be the union of all these subschemes (there is one for each $i=1,\dots,N)$), and let $Z$ be the scheme-theoretic closure of $Z^\circ$ in $\Delta_0=(\bar\mu)^{-1}(V(\delta_0))$.
 \par
 We {\em define} the Coulomb branch for the pair $(G,\oV)$ as the blow-up of  $ \Hilb^W_\mu (X(G,\frac{1}{2}\oV))$ along $Z$ transverse to $\Delta_0$.
 \par
 Comparison with alternative definitions  will proceed via the same ideas, including the ones from \cite{BFN}, behind Theorem \ref{C=W}, which imply that  the spaces we defined are uniquely determined amongst affine varieties with flat projection onto $\h^\ast/W$ by their structure over the usual open subset of $\h^\ast/W$. Since neither we nor  \cites{BDFRT,Tel2}   have established these properties, we have to postpone such a direct comparison to the future.
\begin{remark} The obstruction of Corollary \ref{obstruction} for $X(G,\frac{1}{2}\oV)$ to be $W$-invariant becomes now a condition on the representation $\mathbb{V}$. This obstruction vanishes in the case of $\mathbb{V}=V\oplus V^\ast$, since the arrangement $(\sA,m)$ is then orientable. In  \cite{BDFRT} and in \cite{Tel2} different types of obstructions have been defined in terms of the representation $\mathbb{V}$. In particular, the ``anomaly cancellation" of \cite[\S 4.1]{BDFRT} can be shown to be equivalent to condition \eqref{eq:Anomaly:Cancellation} above. 
We note that the representation in Example \ref{eg:Standard:SO(4)} is not quaternionic for odd $m$, so we do not know if \eqref{eq:Anomaly:Cancellation} is equivalent to the vanishing of the extension class in Corollary \ref{obstruction} for quaternionic representations. 
On the other hand, \cite[Appendix B]{ BDFRT}  describes  the anomaly cancellation condition in terms of the characteristic class $w_4(\mathbb{V})\in H^4(BG^\vee;\Z_2)$. In 
\cite[\S 5]{Tel2} Teleman discusses how a slightly stronger condition on $w_4(\mathbb{V})$ (which has to with gradings)  implies the vanishing of a class that should be
identified with our extension class in $H^2(W ; T) $ associated to the hypertoric variety $X$.  
\label{our-Tel}
\end{remark}
There is one more thing to verify. We require the Coulomb branch to map surjectively\footnote{One reason is that Coulomb branches are expected to be {\em complete} stratified hyperk\"ahler manifolds.} onto $\h/W$. As observed in Remark \ref{notstrong} local models of $\Hilb^W_\mu (X)$ for a general $W$-invariant hypertoric variety can be isomorphic to the variety in Example \ref{bastard}. In this example the action of $W$ on $X$ is free, and hence  the induced morphism $\bar\mu:\Hilb^W_\mu (X)\to \h/W$ is not surjective. We therefore need to show that Example \ref{bastard} cannot occur as a local model $\Hilb^W_\mu (X(G,\frac{1}{2}\oV))$   if $\oV$ is a quaternionic representation of $G$. This follows from the following simple lemma: 
\begin{lemma} If $\oV$ is a quaternionic representation of ${\rm GL}_2(\C)$, then any weight of  ${\rm GL}_2(\C)$ in $\oV$ of the form $(N,-N)$ has even multiplicity.\label{evenGL2}
 \end{lemma}
 \begin{proof} The weights $(k,l)$  of  an even-dimensional dimensional irreducible representation of  ${\rm GL}_2(\cx)$ satisfy $k+l\equiv 1\mod 2$. Therefore the weights of the form $(N,-N)$ can occur only among odd-dimensional irreducible summands of $\mathbb{V}$. Since $\oV$ is quaternionic, the odd dimensional irreducible summands appear in pairs $V\oplus V^\ast$, and the claim follows.
 \end{proof}

\section{Hyperk\"ahler structure of $\Hilb^W_\pi(X)$, part  1\label{hk-Hilb^W}}

\subsection{Twistor space\label{twistor}}
The construction of a hypertoric variety as a complex-symplectic quotient can be refined to a hyperk\"ahler quotient. The data consists now of a compact torus $T_\oR$ and a collection $(\alpha_i,\underline{\lambda}_i)$, $i=1,\dots,d$, where $\alpha_i$ are cocharacters and $\underline{\lambda}_i\in \oR^3$. The hyperk\"ahler quotient of $\oH^d\times T_\oR\times \h_\oR^3$ by $T_\oR^d$ is a {\em stratified hyperk\"ahler manifold} \cite{DS,Mayr}. Its metric depends on the choice of a flat metric on $\oH^d\times T_\oR\times \h_\oR^3$. We always choose the standard Euclidean metric on $\oH^d$, while the $T_\oR$-invariant flat hyperk\"ahler metric on $T_\oR\times \h_\oR^3$  can be arbitrary. If we fix a scalar product $\langle\,,\,\rangle$ on $\h_\oR$ (thus identifying $\h_\oR^\ast\simeq \h_\oR$), then such a metric on $T_\oR\times \h_\oR^3$ with coordinates $(t,x^1,x^2,x^3)$ is of the form \cite{PP}
\begin{equation} \langle L^{-1}(d\log t),d\log t\rangle+ \sum_{i=1}^3\langle L (dx^i),dx^i\rangle,\label{flathk}\end{equation}
for a positive-definite self-adjoint linear automorphism $L$ of $\h_\oR$. The resulting (stratified) hyperk\"ahler manifold $X$ is called a {\em toric hyperk\"ahler manifold} \cite{BD} and it is determined by the following multiarrangement of codimension $3$ flats in $\h_\R\otimes \oR^{3}$ (cf.\ Proposition \ref{bijection}):
\begin{equation} \bigl\{\underline{H_i}=\{(x^1,x^2,x^3)\in \h_\oR\otimes\oR^3; \langle \alpha_i,x^j\rangle=\lambda_i^j,\enskip j=1,2,3\}; \enskip i=1,\dots,d\bigr\}.\label{Hflats}\end{equation}
Analogously to the complex case (Remark \ref{multiarr}), we shall view this multiarrangement as a pair $(\sA,m)$, where $\sA$ is an arrangement of distinct codimension $3$ flats and $m:\sA\to \N$ is a multiplicity function.
\par
The twistor space of $X$ can be obtained by performing the GIT symplectic quotient by $T=(T_\oR)^\cx$ along the fibres of the twistor space of  $\oH^d\times T_\oR\times \h_\oR^3$. In order to determine the metric it is enough to consider
a singular model $Z_L(X)$ of the twistor space, obtained by performing the affine symplectic quotient construction of \S\ref{htv} along the fibres of the twistor space of  $\oH^d\times T_\oR\times \h_\oR^3$. We proceed to describe this singular model. 
Let $(U,\zeta), (\tilde U,\tilde \zeta)$ be the standard holomorphic atlas of $\oP^1$ with $\tilde \zeta=\zeta^{-1}$, and let $(U\times \h,\zeta,\eta)$, $(\tilde U\times \h,\tilde \zeta,\tilde\eta)$ be the corresponding atlas on $Y=\sO_{\oP^1}(2)\otimes \h$ (thus $\tilde\eta=\eta/\zeta^2$). We denote by $P^L$ the principal $T$-bundle over $Y$ with transition function $\exp\bigl(-L(\eta)/\zeta\bigr)$ from $U$ to $\tilde U$. Observe that it is the twistor space for our flat hyperk\"ahler metric on $T_\oR\times \h_\oR^3$. For any character $\xi\in \sX^\ast(T)$ denote by  $\sL^{L(\xi)}$ the line bundle on $Y$ associated to $P^L$ via the $1$-dimensional representation of $T$ with weight $\xi$.
The description of $\cx[X]$, given in \S\ref{htv}, in terms of the moment map coordinates $z$ and functions $x_\xi$ labeled by a collection of characters $\xi\in S\subset \sX^\ast(T)$, yields a description of $Z_L(X)$ as a subvariety of the total space of
$$ \bigoplus_{\xi\in S} \sL^{L(\xi)}\otimes\sO_Y\Bigl(\sum_{i=1}^d|\langle \alpha_i,\xi\rangle|\Bigr),
$$
cut out by equations:
\begin{equation} x_\xi\cdot x_{\xi^\prime}=x_{\xi+\xi^\prime}\prod_{i=1}^d (\langle \alpha_i,\eta\rangle-\lambda_i(\zeta))^{\langle \alpha_i,\xi\rangle_+ +\langle \alpha_i,\xi^\prime\rangle_+ - \langle \alpha_i,\xi+\xi^\prime\rangle_+},   \label{sing-tw}
\end{equation}
where $x_\xi$ denotes the fibre coordinate on $\sL^{L(\xi)}\otimes\sO_Y\bigl(\sum_{i=1}^d|\langle \alpha_i,\xi\rangle|\bigr)$, and $\lambda_i(\zeta)=(\lambda_i^2+\sqrt{-1}\lambda_i^3) +2\lambda_i^1\zeta-(\lambda_i^2-\sqrt{-1}\lambda_i^3)\zeta^2$. The real structure $\sigma$ is given by
\begin{equation} \zeta\mapsto -1/\bar\zeta,\enskip \eta\mapsto -\bar\eta/\bar\zeta^2,\enskip x_\xi\mapsto (-1)^{\sum_{i=1}^d\langle\alpha_i,\xi\rangle_+}\,\ol{x_{-\xi}}\,\bar\zeta^{-\sum_{i=1}^d|\langle \alpha_i,\xi\rangle|}.\label{realstr}\end{equation}
\par
Observe that $T$ acts fibrewise on  $Z_L(X)$ and the action is Hamiltonian with respect to the fibrewise Poisson structure with the moment map equal to the projection onto $Y=\sO_{\oP^1}(2)\otimes \h$. The action of $T$ is free except at the points of $\mu^{-1}(\Delta)$, where $\Delta\subset Y$ is the union of hypersurfaces of the form 
\begin{equation} H_i=\{\eta\in \sO_{\oP^1}(2)\otimes\h;\, \langle \alpha_i,\eta\rangle=\lambda_i(p(\eta))\}\quad i=1,\dots,d,\label{H_i}\end{equation} 
where $p:Y\to \oP^1$ the natural projection. 
\begin{remark} The hyperk\"ahler metrics defined above have volume growth $r^{3\dim T_\oR}=r^{\frac{3}{4}\dim_\oR X}$. In the case when the cocharacters $\alpha_k$ span $\h$, $X$ has also a $T_\oR$-invariant hyperk\"ahler metric with Euclidean volume growth, obtained as the hyperk\"ahler quotient of $\oH^d$ by the kernel of the homomorphism $h:T_\oR^d\to T_\oR$, $h(e_i)=\exp\alpha_i$ (as well as metrics with intermediate volume growth obtained as ``Taub-NUT-deformations").\label{evg}  
\end{remark}
Suppose now that our toric hyperk\"ahler manifold $X$ is {\em $W$-invariant}, i.e.\ the action of $T_\R$ extends to a   hyperk\"ahler action of $T_\R\rtimes W$. In particular the scalar product $\langle\,,\,\rangle$ on $\h_\oR$ must be $W$-invariant and the automorphism $L$ must be $W$-equivariant. We can then apply the functor $\Hilb^W_\pi$ to the fibres of $Z_L(X)$ and obtain a new scheme $\pi:Z^W_L(X)\to \oP^1$.
The morphism $\pi$ is flat (\cite[Prop. III.9.7]{Hart}). $Z^W_L(X)$ has a real structure covering the antipodal map. In general, we do not know whether $Z^W_L(X)$ has an $\sO_{\oP^1}(2)$-valued symplectic form in the sense of \S\ref{symplstr}. This is the case for strongly $W$-invariant $X$, but also for regular $X$. We therefore adopt the following definition:
\begin{definition} A $W$-invariant hypertoric variety $X$ is said to be {\em  unexceptional} if
\begin{itemize}
\item[(i)] the depth of $\sO_{\Hilb_\mu^W(X)}$ along the subset  $\bar\mu^{-1}(F)$ defined in the proof of Theorem \ref{normal} is at least $2$, and
\item[(ii)] $W\text{-}\Hilb_\mu(X)=\Hilb_\mu^W(X)$.
\end{itemize}
A $W$-invariant toric hyperk\"ahler (stratified) manifold is said to be {\em unexceptional} if it is  unexceptional with respect to every complex structure.\label{unexceptional}
\end{definition}
Therefore smooth  or strongly $W$-invariant toric hyperk\"ahler manifolds are  unexceptional. An  unexceptional $W$-invariant hypertoric variety is normal (Remark \ref{notstrong}) and it has a symplectic form (\S\ref{symplstr}). Therefore, if $X$ is an  unexceptional $W$-invariant toric hyperk\"ahler manifold, then  $Z_L(X)$  has an induced $\sO_{\oP^1}(2)$-valued symplectic form. Moreover every fibre of $\pi:Z^W_L(X)\to \oP^1$ is normal, and hence  $Z^W_L(X)$ is normal (cf.\ \cite[Theorem 23.9]{Mats}). Condition (ii) will be needed shortly, in order to describe sections of $Z^W_L(X)$ as curves in $Z_L(X)$.

\medskip

In what follows we shall assume that $X$ is an unexceptional $W$-invariant toric hyperk\"ahler manifold.
\begin{definition} We denote by $H_L(X)$ the smooth locus of the Douady space  of real sections of $ \pi$ with normal sheaf isomorphic to $\bigoplus \sO_{\oP^1}(1)$.\label{g_L}\end{definition}
\begin{remark} Since $Z_L(X)$ is singular, the condition on the normal sheaf is apriori not sufficient to guarantee the smoothness of the Douady space at the point corresponding to a section.
\end{remark}
$H_L(X)$ is a hypercomplex manifold with a pseudo-hyperk\"ahler metric on each connected component (possibly of varying signature). Moreover,  the map $H_L(X)\to \Hilb_\mu^W(X)$  which associates to a section its intersection with the fibre of $Z^W_L(X)$ is a local analytic isomorphism for a generic complex structure. 
\par
We now observe that a section of $\pi$ is the same thing\footnote{The idea of constructing hyperk\"ahler metrics by lifting higher degree curves to a twistor space goes back at least to \cite{Nash}.} as a  $W$-invariant curve $\hat C$ in $Z_L(X)$ of degree $|W|$, flat over $\oP^1$ and such that $\mu|_{\hat C}$ is an equivariant isomorphism onto a $W$-invariant curve $C$ in $\h\otimes \sO_{\oP^1}(2)$.  
Our first task is therefore to describe $W$-invariant curves in $\h\otimes \sO_{\oP^1}(2)$ (flat of degree $|W|$ over $\oP^1$), which can be lifted $W$-equivariantly to $Z_L(X)$. We can do this for general, not necessarily $W$-invariant, hypertoric varieties. In the following theorem  we use the additive notation in the abelian group $H^1(Y,T)$ and the identification $H^1(Y,T)\simeq H^1(Y,\sO^\ast)\otimes_\oZ  \sX_\ast(T)$.
\begin{theorem} Let $X$ be a toric hyperk\"ahler manifold corresponding to an integral multiarrangement $\{\sA,m\}$ with $\sA=\{\underline{H_1},\dots, \underline{H_d}\}$ and each $\underline{H_i}$ orthogonal to $\omega_i\otimes \R^3$, where $\omega_i\in\sX_\ast(T)$ is primitive. Set $m_i=m(\underline{H_i})$, $i=1,\dots,d$.
\par
Let $C$ be a curve in $\h\otimes \sO_{\oP^1}(2)$, flat over $\oP^1$, no component of which is contained in any of the hypersurfaces $H_i$, $i=1,\dots,d$. \\
{\rm (a)} Suppose that each Cartier divisor $(H_i)$,  $i=1,\dots,d$,  on $C$ admits a decomposition of the form $\sum_{k=0}^{m_i}H^k_i$, with each $ H^k_i$ effective, such that the principal $T$-bundle
\begin{equation} P^{L}|_C+\sum_{i=1}^d\sO_C(m_i)\Bigl[-\sum_{k=0}^{m_i}kH_i^k\Bigr]\otimes \omega_i\label{trivial-bundle}\end{equation}
on $C$ is trivial. Then $C$ can be lifted to $Z_L(X)$. \\
{\rm (b)} If $C$ is real, then a real lift of $C$ exists if and only if  the real structure interchanges $H_i^k$ and $H_i^{m_i-k}$ for every $k=0,\dots,m_i$ and $i=1,\dots,d$.\\
{\rm (c)} Suppose that $C$  satisfies, in addition:
\begin{itemize} 
\item[(i)] $C$ does not meet intersections of two or more hyperplanes $H_i$;
\item[(ii)] $C$ meets each $H_i$ transversely.
\end{itemize}
If $C$ can be lifted to $Z_L(X)$, then for each $i=1,\dots,d$ there exists a decomposition $C\cap H_i=\bigsqcup_{k=0}^{m_i}H^k_i$ such that the principal bundle \eqref{trivial-bundle} is trivial.\label{lift}
\end{theorem}
\begin{proof} Given the description of $Z_L(X)$ above, it is clear that a lift of a curve $C$ is equivalent to the existence of sections $x_\xi$ of line bundles $\sL^{L(\xi)}|_{C}\Bigl(\sum_{i=1}^d|\langle \alpha_i,\xi\rangle|\Bigr)$, $\xi\in \sX^\ast(T)$, where $\alpha_i=m_i\omega_i$,
which satisfy relations \eqref{sing-tw}. Let us suppose that the assumptions in part (a) are fulfilled. A section of the trivial $T$-bundle $P$ given by \eqref{trivial-bundle} yields, for each $\xi\in \sX^\ast(T)$, a section $s_\xi$ of the line bundle $P\times_\xi \cx$ associated to $P$  via the $1$-dimensional representation of $T$ with weight $\xi$. These sections satisfy $s_\xi s_{\xi^\prime}=s_{\xi+\xi^\prime}$ for all $\xi,\xi^\prime\in \sX^\ast(T)$. We claim that the line bundle $P\times_\xi \cx$ is isomorphic to 
\begin{equation}\sL^{L(\xi)}|_{C}\Bigl(\sum_{i=1}^d|\langle \alpha_i,\xi\rangle|\Bigr)\Bigl[- \sum_{i=1}^d \sum_{k=0}^{m_i}
\bigl(k\langle \omega_i,\xi\rangle_+ + (m_i-k)\langle \omega_i,\xi\rangle_-\bigr)H_i^k\Bigr].\label{asso}\end{equation}
Since $\sL^{L(\xi)}$ has been defined as $P_L\times_\xi \cx$, it is enough to show  that, for each $i=1,\dots,d$, the line bundle associated to $\sO_C(m_i)\Bigl[-\sum_{k=0}^{m_i}kH_i^k\Bigr]\otimes \omega_i$ is isomorphic to 
\begin{equation}\sO_C\Bigl(|\langle \alpha_i,\xi\rangle|\Bigr)\Bigl[-  \sum_{k=0}^{m_i}
\bigl(k\langle \omega_i,\xi\rangle_+ + (m_i-k)\langle \omega_i,\xi\rangle_-\bigr)H_i^k\Bigr].\label{assoc}\end{equation}
This is obvious if $\langle \omega_i,\xi\rangle\geq 0$ (and hence $\langle \alpha_i,\xi\rangle\geq 0$). If $\langle \omega_i,\xi\rangle< 0$, the associated line bundle is clearly isomorphic to $\sO(-m_i|\langle \omega_i,\xi\rangle|)\Bigl[  \sum_{k=0}^{m_i} k\langle \omega_i,\xi\rangle_-\bigr)H_i^k\Bigr]$. This is, however, isomorphic to \eqref{assoc}, owing to  $\bigl[\sum_{k=0}^{m_i} H^k_i\bigr]\simeq \sO_C(2)$.
\par
Let now ${U_j}$, $j=1,\dots,t$, be a covering of $C$ such that each Cartier divisor $H^k_i$ is represented by  $(U_j,h^k_{ij})_{j=1,\dots,t}$, where $h^k_{ij}\in \sO(U_j)$. Then 
$$\Bigl(U_j,s_\xi \sum_{i=1}^d \sum_{k=0}^{m_i}
\bigl(k\langle \omega_i,\xi\rangle_+ + (m_i-k)\langle \omega_i,\xi\rangle_-\bigr)h^k_{ij}\Bigr)_{j=1,\dots,t}$$ represents a section $x_\xi$ of $\sL^{L(\xi)}|_{C}\Bigl(\sum_{i=1}^d|\langle \alpha_i,\xi\rangle|\Bigr)$.
It remains to check that these sections satisfy relations \eqref{sing-tw}. We observe that 
$$k\langle \omega_i,\xi\rangle_+ + (m_i-k)\langle \omega_i,\xi\rangle_-= \langle \alpha_i,\xi\rangle_+ +(k-m_i)\langle \omega_i,\xi\rangle, \quad i=1,\dots,d,\enskip k=0,\dots,m_i.
$$
Relations \eqref{sing-tw} follow now easily from $s_\xi s_{\xi^\prime}=s_{\xi+\xi^\prime}$ and $\sum_{k=0}^{m_i}H^k_i=(H_i)$.
\par
Part (b) is a consequence of the description \eqref{realstr} of the real structure on $Z_L(X)$.
\par
For part (c), let us suppose that a lift of $C$ exists, i.e.\ there exist sections $x_\xi$ of line bundles $\sL^{L(\xi)}|_{C}\Bigl(\sum_{i=1}^d|\langle \alpha_i,\xi\rangle|\Bigr)$, $\xi\in \sX^\ast(T)$,
which satisfy relations \eqref{sing-tw}.  Given the genericity assumption on the curve $C$, the intersection $C\cap H_i$ consists of distinct points. Let $p$ be one of these points and let $d_p(\xi)\geq 0$ be the order of $x_\xi$ at $p$. The relations \eqref{sing-tw} imply that $d_p(\xi)-\langle \alpha_i,\xi\rangle_+$ is a homomorphism $\sX^\ast(T)\to \oZ$, i.e.\ $d_p(\xi)-\langle \alpha_i,\xi\rangle_+= \langle \beta_p,\xi\rangle$ for some $\beta_p\in \sX_\ast(T)$. Were $\oR\beta_p\neq \oR\alpha_i$,  we could find a $\xi\in \sX^\ast(T)$ with $\langle \alpha_i,\xi\rangle=0$ and $\langle \beta_p,\xi\rangle<0$, which would contradict $d_p(\xi)\geq 0$. Similarly, we get a contradiction if $\beta_p=r\alpha_i$ with $r\not\in [-1,0]$.
Therefore the only possibilities are $\beta_p=-k\omega_i$, for a $k\in\{0,1,\dots,m_i\}$. Set 
$$H^k_i=\{p\in C\cap H_i;\, \beta_p=-(m_i-k)\omega_i\},\quad k=1,\dots,m_i.$$ 
For $p\in H^k_i$ and all $\xi\in \sX^\ast(T)$ we have $d_p(\xi)=k\langle \omega_i,\xi\rangle_+ + (m_i-k)\langle \omega_i,\xi\rangle_-$. It follows that  the line bundle \eqref{asso}
is trivial. Therefore all line bundles associated to the torus bundle \eqref{trivial-bundle} are trivial, and hence the torus bundle itself  is trivial.
\end{proof}
\begin{remark} In the case when $m_i=1$  (e.g.\ when $X$ is smooth), the decomposition of $(H_i)$ involves only two divisors $H_i^0$ and $H_i^1$, which in the case of a real curve and a real lift are interchanged by $\sigma$, and hence their  degrees are both equal to  $\deg C$. 
\end{remark} 
\begin{remark} For the hyperk\"ahler metric with Euclidean volume growth defined in Remark \ref{evg}, it may happen that there are no lifts of curves in $\h\otimes \sO_{\oP^1}(2)$ to the twistor space.  This is the case, for example, for $X=\cx^{2d}$, cf.\ \cite[\S 6]{slices}. \end{remark}
\begin{remark} Suppose that $X$ is a $W$-invariant toric hyperk\"ahler manifold and let $\hat C$ be a lift of a $W$-invariant $C$ obtained from Theorem \ref{lift}(a).
If $\hat C$ is also $W$-invariant, then, clearly, the collection $\{H^k_i\}$ is $W$-invariant for each $k$. Conversely, if $\{H^k_i\}$ is $W$-invariant for each $k$, then for every $w\in W$ the sections $x_\xi$  corresponding to $\hat C$ and to $w.\hat C$ have the same $0$-divisor, i.e.\  $w.\hat C=t_w.\hat C$ for some $t_w\in T$. The map $w\to t_w$ is a cocycle in $Z^1(W;T)$, and, since we assumed (at the beginning of \S\ref{Snormal}) that the action of $W$ on $X$ corresponds to the trivial class in $H^1(W;T)$, there exists a $t\in T$ such that $t_w= w(t)t^{-1}$ $\forall w\in W$. The curve $t^{-1}.\hat C$ is then a  $W$-equivariant lift of $C$.
\end{remark}
\begin{remark} Let $X$ be a strongly $W$-invariant toric hyperk\"ahler manifold, and let $X^\prime$ be the $W$-invariant toric hyperk\"ahler manifold obtained as in Proposition \ref{Tel/W}, i.e.\ we add the reflection flats in $\h_\R\otimes \oR^{3}$ with multiplicity two to the multiarrangement determining $X$. Proposition \ref{Tel/W} implies that a $W$-orbit of sections of $Z_L(X)$ defines a section of $Z_L^W(X^\prime)$. This gives us a component of $H_L(X^\prime)$ isomorphic to $X/W$ as a stratified hyperk\"ahler manifold.
In other words, Proposition \ref{Tel/W} remains true in the hyperk\"ahler setting (with the proviso that there could be other components of $H_L(X^\prime)$).
\end{remark}
\begin{remark} The description \eqref{monomials} of the coordinate ring of a hypertoric variety $X$ implies that, away from the hyperplanes $H_i$, $X$ is isomorphic to the variety given by equations
$$x_{\xi}x_{-\xi}=\prod_{i=1}^d (\langle \alpha_i,z\rangle-\lambda_i)^{|\langle \alpha_i,\xi\rangle|},\quad \xi\in \sX^\ast(T).$$
We can define a ``twistor space" $\bar{Z}_L(X)$ of this variety by imposing only equations \eqref{sing-tw} with $\xi^\prime=-\xi$. Clearly, any lift of a curve $C$ in $\h\otimes \sO_{\oP^1}(2)$ to $Z_L(X)$ defines a lift to $\bar Z_L(X)$. Conversely, a generic lift to $\bar Z_L(X)$ will also yield a lift to to $Z_L(X)$. In the presence of a $W$-action, $\bar{Z}_L(X)$  is also $W$-invariant, and the correspondence between lifts to $Z_L(X)$ and to   $\bar{Z}_L(X)$
remains valid for $W$-equivariant lifts. \label{bar{Z}}
\end{remark}
Let us return to the case of an arbitrary unexceptional $W$-invariant toric hyperk\"ahler (stratified) manifold $X$.
Let   $C$ be a $W$-invariant curve in $\h\otimes \sO_{\oP^1}(2)$, flat of degree $|W|$ over $\oP^1$, and suppose that  $\hat C$ is a $W$-equivariant lift of $C$ to $Z_L(X)$. 
As discussed above, such a lift defines a section $s_{\hat C}$ of the twistor space $Z_L^W(X)$ and we ask whether its normal sheaf is isomorphic to $\bigoplus \sO_{\oP^1}(1)$. 
\begin{proposition} The normal sheaf of $s_{\hat C}$ is locally free. The normal sheaf of $\hat C$ is locally free if the intersection of $\hat C$ with ${\rm Sing}\; Z_L(X)$ is a finite set of integral points of $\hat C$.
If $\sN_{\hat C/Z_L(X)}$ is torsion-free, then $\sN_{s_{\hat C}/Z_L^W(X)}\simeq \bigoplus \sO_{\oP^1}(1)$ if and only if  $H^0\bigl(\hat C, \sN_{\hat C/Z_L(X)}(-2)\bigr)^W=0.$\label{singZ}
\end{proposition}
\begin{proof}
Since the normal sheaves of $\hat C$ and of   $s_{\hat C}$ are duals of the coherent conormal sheaves, they are torsion-free as soon as the torsion of the conormal sheaves is supported at integral points. This is always the case for $s_{\hat C}$. Since $\hat C\simeq C$, $\hat C$ is lci, and therefore its conormal sheaf is locally free outside of ${\rm Sing}\; Z_L(X)$. The assumption in the second statement guarantees now that the normal sheaf of $\hat C$ is torsion free. In the case of curves torsion-free is equivalent to locally free, and this proves the first two statements. It follows now from well-known results (see, e.g.\ \cite[Appendix A]{Tel}) that  $\pi_\ast \sN_{\hat C/Z_L(X)}$ is torsion-free, hence locally free. We have an exact sequence
$$0\longrightarrow \sO_{\hat C}\otimes \h \longrightarrow  \sN_{\hat C/Z_L(X)} \longrightarrow  \sN_{C/\sO_{\oP^1}(2)\otimes\h}  \longrightarrow 0.
$$
If we push it down to $\oP^1$ and take the subsheaf of $W$-invariant sections, we obtain a commutative diagram:
$$ \minCDarrowwidth18pt\begin{CD}  
0 @>>> \bigl(\pi_\ast\sO_{\hat C}\otimes \h\bigr)^W  @>>> \pi_\ast\bigl(\sN_{\hat C/Z_L(X)}\bigr)^W  @>>> \pi_\ast\bigl(\sN_{C/\sO_{\oP^1}(2)\otimes\h}\bigr)^W   @>>>0\\
@. @VV\wr V @VVV @VV\wr V @.\\
0@>>> \bigoplus_{i=1}^n\sO_{\oP^1}(-2d_i+2) @>>> \sN_{s_{\hat C}/Z_L^W(X)} @>>>  \bigoplus_{i=1}^n\sO_{\oP^1}(2d_i) @>>> 0,
\end{CD}$$
where $d_1,\dots,d_n$ are the degrees of generators of $\cx[\h]^W$.
Therefore the middle arrow is also an isomorphism, and this proves the third statement.
\end{proof}

\subsection{$W$-invariant curves, tori, and modified Nahm's equations\label{Torsors}} 
We now recall, after \cite{Kan,Scogn,Hurt}, the correspondence between $W$-invariant $T$-bundles over $W$-invariant   curves in $\h\otimes  \sO_{\oP^1}(2)$ and sections of $\ad \Pi\otimes\sO_{\oP^1}(2)$, where $\Pi$ is a principal $G$-bundle over $\oP^1$, where $G= G_{{\scriptscriptstyle T, W}}$. Let $C$ be a $W$-invariant curve in $\h\otimes  \sO_{\oP^1}(2)$, flat   of degree $|W|$ over $\oP^1$. Principal $T$-bundle over $C$ are classified by $H^1(C,T)\simeq H^1(C,\sO^\ast)\otimes_{\oZ}\sX_\ast(T)$. We decompose the branch locus $\sB$ of $C$ as $\sB=\sum_{\beta \in R^+} \sB_\beta$, where $R^+$ is the set of positive roots and $\sB_\beta$ is the divisor fixed by the reflection  in the $\beta$-hyperplane, and   
define a principal bundle $T$-bundle $P_{{\scriptscriptstyle\sB}}=\sum_{\beta\in R^+}[\frac{1}{2}\sB_\beta]\otimes \beta^\vee$.  If $P$ is a $W$-invariant $T$-bundle on $C$, then the sheaf
$\g\times_T(P- P_{{\scriptscriptstyle\sB}})$ of Lie algebras is the pullback of $\ad \Pi$, where $\Pi$ is a principal $G$-bundle  on $\oP^1$ \cite{Scogn,Hurt}. The sheaf $\g\times_T(P- P_{{\scriptscriptstyle\sB}})$ has a tautological section given by $x\mapsto[x,u]\in \h\times_T(P- P_{{\scriptscriptstyle\sB}})\subset \g\times_T(P- P_{{\scriptscriptstyle\sB}})$, where $u$ is any element of $P- P_{{\scriptscriptstyle\sB}}$ over $x$. This section descends to a section of $\ad \Pi$.
Conversely, given a section $\Phi$ of $\ad \Pi\otimes\sO_{\oP^1}(2)$  such that $\Phi(\zeta)$ is a regular element of $\g$ for every $\zeta$, we can conjugate it over any $\zeta\in \oP^1$ to an element $b(\zeta)$ of a fixed Borel subalgebra $\fB\supset\h$: $b(\zeta)=\Ad_{g(\zeta)} \bigl(\Phi(\zeta)\bigr)$. The $W$-orbit of the $\h$-part of $b(\zeta)$ defines the curve $C$, while $\zeta\mapsto g(\zeta)^{-1}Bg(\zeta)$ defines a principal $B$-bundle over $C$ ($B=\exp \fB$). The $T$ bundle obtained from the  projection $B\to T$ is the bundle $P- P_{{\scriptscriptstyle\sB}}$.
\par
Suppose now that the $G$-bundle $\Pi$ is trivial. Once we fix its trivialisation, then the sections of $\ad \Pi\otimes\sO_{\oP^1}(2)$ are quadratic $\g$-valued polynomials (with regular values for every $\zeta\in \oP^1$). We can fix the curve $C$ and  
 consider the flow of $T$-bundles on $C$ in the direction $P^L$, where $P^L$ is the $T$-bundle over $\h\otimes  \sO_{\oP^1}(2)$ defined in the previous subsection. If the flow exists on an interval $I$ (i.e.\ the principal bundle $\Pi$ associated to $(C,T)$ stays trivial on $I$), then  choice of a connection on $p^\ast \Pi$, where $p:\oP^1\times I\to \oP^1$ is the projection, produces a flow of quadratic polynomials. 
 \par
In particular, there is a connection such that the flow $P+ P^{\,t\Id}$ corresponds to the flow $\Phi(\zeta,t)$ of quadratic $\g$-valued polynomials $\Phi(t,\zeta)$ given by $\frac{\partial \Phi}{\partial t} =\frac{1}{2}\bigl[\Phi,\frac{\partial \Phi}{\partial \zeta}\bigr]$ (cf.\ \cite[Eq.\lpar 4.6\rpar]{Hurt}, \cite{Hit}). The same argument shows that the flow $P+ P^{\,tL}$ corresponds to the flow $\frac{\partial \Phi}{\partial t} =\frac{1}{2}L\Bigl(\bigl[\Phi,\frac{\partial \Phi}{\partial \zeta}\bigr]\Bigr),$
where $L:\g\to \g$ is the unique $G$-equivariant linear extension of the $W$-invariant $L:\h\to \h$.
In order to obtain  a solution to Nahm's equations, or to our $L$-version of them, we need to impose reality conditions on $T$-bundles over $C$. We call such a bundle $P$  {\em real} if $\ol{\sigma^\ast P}\simeq -P$. A $\g$-valued polynomial $\Phi(\zeta)=\Phi_0+\Phi_1\zeta+\Phi_2\zeta^2$ corresponding to a real $P$ satisfies
$$\Ad_h(\Phi_0)=-\Phi_2^\ast,\enskip \Ad_h(\Phi_1)=\Phi_1^\ast$$
for some $h\in G$, where the asterisk denotes the negative of the Cartan involution defined by the maximal compact subgroup $K$ (i.e.\ $\ast$ is equal to $-1$ on $\fK$). In particular, there is an open and closed subset of the variety of real $W$-invariant $T$-bundles, where $h$ can be chosen in $\exp (i\fK)$. $\Phi(\zeta)$ can be then $K$-conjugated to a polynomial which satisfies $\Phi_0=T_2+iT_3$, $\Phi_1=2iT_1$, $\Phi_2=T_2-iT_3$ with $T_i\in\fK$. The flow $P+ P^{\,tL}$ corresponds then to a solution of $\fK$-valued {\em $L$-Nahm's equations}:
\begin{equation} \dot T_1=L\bigl([T_2,T_3]\bigr),\quad  \dot T_2=L\bigl([T_3,T_1]\bigr), \quad  \dot T_3=L\bigl([T_1,T_2]\bigr).\label{mNe}
\end{equation}
We discuss these equations in greater detail in Appendix \ref{appendix:B}. 
\par
We wish to give another description of the space of the isomorphism classes of $W$-invariant principal $T$-bundles over a $W$-invariant curve in $\h\otimes  \sO_{\oP^1}(2)$, flat   of degree $|W|$ over $\oP^1$. We have $\bigl(\h\otimes  \sO_{\oP^1}(2)\bigr)\big/ W\simeq \bigoplus_{i=1}^n\sO_{\oP^1}(2d_i)$, where $d_1,\dots,d_n$ are the degrees of generators of $\cx[\h]^W$. Thus a $W$-invariant curve $C$ in $\h\otimes \sO_{\oP^1}(2)$ corresponds to a section $\bar C$ of $\bigoplus_{i=1}^n\sO_{\oP^1}(2d_i)$. We can identify $\h/W$ with the Slodowy slice $\mathscr{S}_{\g}\subset \g$ to the regular nilpotent orbit, and restrict the universal centraliser \eqref{ucentr} to $\bar C$. This is an abelian group scheme over $\bar C$, which we denote by $\sT_{\bar C}$.
If $P\stackrel{\pi}{\rightarrow} C$ is a $W$-invariant $T$-bundle over $C$, then $\Hilb^W_\pi(P)$ is a $\sT_{\bar C}$-torsor, locally trivial in the sense that it is locally isomorphic to $\sT_{\bar C}$ (since $T$ is locally trivial). Locally trivial analytic $\sT_{\bar C}$-torsors over $\bar C\simeq \oP^1$ are classified by the abelian group $H^1(\oP^1,\sT_{\bar C})$, where the cohomology coefficients denote the sheaf of local analytic sections of $\sT_{\bar C}\to \oP^1$. We have: 
\begin{proposition} The map $P\to \Hilb^W_\pi(P)$ induces an isomorphism of abelian groups
$H^1(C,T)^W\to H^1(\oP^1,\sT_{\bar C})$.
\label{torsors} 
\end{proposition}
\begin{proof} It is clear that the induced map is a group homomorphism. We need to construct the inverse homomorphism. Let $\sP$ be a $\sT_{\bar C}$-torsor over $\bar C$, and let $p:C\to\bar C\simeq \oP^1$ be the projection. The pullback $p^\ast\sP$ is a $T$-bundle away from $p^{-1}(R)$, where $R$ is the ramification divisor. We consider the subsheaf of sections of this bundle which remain finite as sections of $p^\ast \sP$ over points of $p^{-1}(R)$. Since $\sP$ is locally trivial, this sheaf is locally isomorphic to the sheaf obtained in this way from $\sT_{\bar C}$, i.e.\ to the trivial $T$-bundle over $C$. We recover thus the $T$-bundle $P$.
\end{proof}
\begin{remark} The $\sT_{\bar C}$-torsor corresponding to a section $\Phi(\zeta)$ of $\ad\Pi\otimes\sO_{\oP^1}(2)$ is obtained by conjugating each $\Phi(\zeta)$ to  the unique element $\bar C(\zeta)$ of the Slodowy slice in the fibre of  $\ad\Pi\otimes\sO_{\oP^1}(2)$ over $\zeta$. This yields the section $\bar C$ and the $\sT_{\bar C}$-torsor is given by $g(\zeta)^{-1}Z_G(\bar C(\zeta))g(\zeta)$, where  $g(\zeta)$ is any element of $G$ with $\Ad_{g(\zeta)}(\Phi(\zeta))=\bar C(\zeta)$.\label{torsors2}
\end{remark}
The moduli space of $W$-invariant $T$-bundles on $C$ has a distinguished divisor $\Theta$ (a generalised theta divisor), where the corresponding principal $G$-bundle on $\oP^1$ is nontrivial. The quadratic polynomials $\Phi(t)$ corresponding to a linear flow on $H^1(C,T)^W$ (and a choice of connection) will acquire a regular singularity at points of $\Theta$. In the case of $\fK$-valued solutions to $L$-Nahm's equations \eqref{mNe}, the  leading term is a simple pole, the residues of which are of the form $L^{-1}(\sigma_1), L^{-1}(\sigma_2), L^{-1}(\sigma_2)$, where $\sigma_1,\sigma_2,\sigma_3\in \fK$ is an $\su(2)$-triple.
 There is a distinguished conjugacy class of $\su(2)$-triples, called principal $\su(2)$-triples, defined by requiring that $\sigma_2+i\sigma_3$ is a regular nilpotent element. 
\begin{proposition} A solution to $L$-Nahm's equations \eqref{mNe} on $(0,\epsilon)$ has a simple pole at $t=0$ with residues $L^{-1}(\sigma_1), L^{-1}(\sigma_2), L^{-1}(\sigma_2)$, where $\sigma_1,\sigma_2,\sigma_3$ is a  principal $\su(2)$-triple, if and only if the corresponding flow of $W$-invariant $T$-bundles on $C$ converges to the trivial bundle as $t\to 0$.\label{poles}
\end{proposition}
\begin{proof} Let $\sigma_1,\sigma_2,\sigma_3$ be a fixed principal $\su(2)$-triple, and $e=\sigma_2+i\sigma_3$, $f=-\sigma_2+i\sigma_3$, $h=2i\sigma_1$ the corresponding $\ssl_2(\cx)$-triple. The quadratic polynomial $\Phi(t,\zeta)$ has residue $L^{-1}(e+h\zeta -f \zeta^2)$. If we conjugate this to $L^{-1}(e)$ via $A(\zeta)\in G$ and then act by the gauge transformation $t^{h/2}$, we shall obtain $\Phi^\prime(t,\zeta)$ which converges, as $t\to 0$, to the section $\bar C$ of $\bigoplus_{i=1}^n\sO_{\oP^1}(2d_i)$, where $\bigoplus_{i=1}^n\sO_{\oP^1}(2d_i)$ is the Slodowy slice obtained from $\bigl(\h\otimes\sO_{\oP^1}(2)\bigr)/W$. The $\sT_{\bar C}$-torsor corresponding to $\Phi^\prime(0,\zeta)$ is simply $\sT_{\bar C}$. For $t\neq 0$, on the other hand, the gauge  transformation $t^{h/2}A(\zeta)$ is just a change of the global trivialisation of the trivial $G$-bundle $\Pi$ on $\oP^1$. Therefore our solution to $L$-Nahm's equations has residues given by a principal $\su(2)$-triple if and only if  the $\sT_{\bar C}$-torsors corresponding to $\Phi(t,\zeta)$ converge to the trivial torsor. The claim follows now from Proposition \ref{torsors}.
\end{proof}

\subsection{Special cases\label{special}} We can describe complete hyperk\"ahler metrics on $\Hilb_\mu^W(X)$ in two basic cases: $X=T\times \h$, and $X=\cx^{2d}$. 
\par
We first consider $X=T\times \h$ with the flat hyperk\"ahler metric \eqref{flathk}
given by a $W$-equivariant $L:\h\to \h$. 
According to the description in \S\ref{twistor} we need to consider $W$-invariant curves $C\subset \h\otimes \sO_{\oP^1}(2)$ , flat of degree $|W|$ over $\oP^1$, such that the $T$-bundle
$P^{\rm L}|_C$  is trivial.  Proposition \ref{poles} and the fact that $L$-Nahm's equations correspond to the flow $ P+ P^{\,tL}$ means that there exists a solution to $L$-Nahm's equations on $(0,1)$ such that the corresponding curve is $C$, and which has simple poles at both ends  with residues $L^{-1}(\sigma_1), L^{-1}(\sigma_2), L^{-1}(\sigma_2)$, where $\sigma_1,\sigma_2,\sigma_3$ is a  principal $\su(2)$-triple.
The  hyperk\"ahler $L^2$-metric, defined in Appendix \ref{appendix:B}, on the  moduli space of solutions to $L$-Nahm's equations with these boundary conditions is  complete. The moduli space itself is biholomorphic to the universal centraliser $\sZ_G\simeq \Hilb^W_\mu(T\times \h)$ for every complex structure. Note that this metric depends on the choice of a $W$-invariant scalar product on $\h_\oR$ we made earlier (which canonically defines an $\Ad_K$-invariant scalar product on $\fK$ owing to the Chevalley theorem).
\par
In the case $X=\cx^{2d}$ and $L=\Id$, the complete hyperk\"ahler metric on $\Hilb^{\Sigma_d}_\mu(\cx^{2d})$ has been already described in \cite[\S6]{slices}. It is the natural $L^2$-metric on a moduli space of $\u(d)$-valued solutions to (usual) Nahm's equations on $(0,1)$ with the same boundary conditions as above and a discontinuity in the middle. This discontinuity is given by the bifundamental representation of ${\rm GL}_d(\cx)$. In other words, the hyperk\"ahler metric on  $\Hilb^{\Sigma_d}_\mu(\cx^{2d})$ is the hyperk\"ahler quotient by $U(d)\times U(d)$ of the product of the following three spaces:
\begin{itemize}
\item[(i)] the moduli space of $\u(d)$-valued solutions to Nahm's equations on $(0,c]$, regular at $t=c$ and with a simple pole at $t=0$ with  residues given by a fixed principal $\su(2)$-triple;
\item[(ii)] the moduli space of $\u(d)$-valued solutions to Nahm's equations on $[c,1)$, regular at $t=c$ and with a simple pole at $t=1$ with  residues given by a fixed principal $\su(2)$-triple;
\item[(iii)] $T^\ast\Mat_{d,d}(\cx)\simeq \Mat_{d,d}(\oH)$ with its flat hyperk\"ahler metric.
\end{itemize}
The hyperk\"ahler quotient matches the values of the Nahm matrices at $t=c$ with the two $\u(k)$-valued moment maps on $T^\ast\Mat_{d,d}(\cx)$. The fact that such a solution to Nahm's equations corresponds to triviality of $P^{\rm  Id}|_C+\sum_{i=1}^d\sO_C(1)[-H^1_i]\otimes e_i$ follows from \cite[Prop. 2.1]{BBC}, and will be reproved in a greater generality in the next section (Proposition \ref{lift2}). We remark that since this triviality does not depend on the location $c$ of the discontinuity, neither does the hyperk\"ahler metric. The point is that a hyperk\"ahler metric alone does not determine a solution to Nahm's equations; one needs additional geometric structures, such as a hyper-Poisson bivector \cite{Bi-hP}.

We wish to describe the hyperk\"ahler metric on $\Hilb^W_\mu(X)$ for a general unexceptional $W$-invariant hypertoric variety $X$ and arbitrary $L$, analogously to a combination of these two examples: as a moduli space of solutions to $L$-Nahm's equations on an interval, with poles at both ends given by a principal $\su(2)$-triple, and an appropriate discontinuity in the middle. In the next section we shall define an appropriate replacement for $T^\ast\Mat_{d,d}(\cx)$.

\section{Hyperk\"ahler structure of $\Hilb^W_\mu(X)$, part 2\label{hypers}}

\subsection{Hypertoric analogues of $G\times \g^{\rm reg}$\label{M_G^circ}}
Let $X$ be an affine $W$-invariant hypertoric variety with structure torus $T$. Let $G=G_{{\scriptscriptstyle T, W}}$ (\S\ref{sTG}) and $\g=\Lie(G)$. The universal centraliser $\sT= \sT_{\scriptscriptstyle{T,W}}$ \eqref{ucentr} acts on   $\Hilb^W_\mu(X)$ and
we consider   a ``quotient" $M^\circ_{G}(X)$ of  $G\times G\times \Hilb^W_\mu(X)$ by the following action of $\sT\times \sT$: 
\begin{equation}(\tau_1,\tau_2).(g_1,g_2,x)=\bigl(g_1\tau_1^{-1}, g_2\tau_2^{-1},(\tau_1\tau_2^{-1}).x\bigr).\label{action-tau}\end{equation}
The quotient is understood in the sense of algebraic relations: two points $(g,h,x)$ and $(g^\prime,h^\prime,x^\prime)$ lie in the same orbit of $\sT\times \sT$ if and only if
\begin{equation*}\bar\mu(x)=\bar\mu(x^\prime),\enskip \Ad_{g^{-1}g^\prime}(\bar\mu(x))=\bar\mu(x)=\Ad_{h^{-1}h^\prime}(\bar\mu(x)), \enskip \bigl(g^{-1}g^\prime {h^\prime}^{-1}h\bigr).x^\prime=x,\label{rel}\end{equation*}
where we have identified $\h/W$ with the Slodowy slice $\mathscr{S}_\g\subset \g$.
It is not clear whether this quotient always exists as a scheme (see, however, Theorem \ref{global-hol} for the strongly $W$-invariant case). It does, however, exist as an algebraic space provided $\Hilb^W_\mu(X)$ is seminormal\footnote{In the analytic category this means that every local continuous holomorphic function which is holomorphic outside a nowhere dense analytic set is holomorphic. For an algebraic definition see \cite{GS}.}. This follows from Artin's extension of GAGA and from 
 a theorem  of Grauert (cf.\ \cite[Thm.\ 7.1]{VII}  which shows that $M^\circ_{G}(X)$ exists as a (seminormal) complex space  (the assumptions of  Grauert's theorem are easily checked in our case). Moreover, if $\Hilb^W_\mu(X)$ is normal, then $M^\circ_{G}(X)$ is normal (cf.\ the second paragraph on p.\ 203 of \cite{VII}). 
\par
In what follows we shall work in the analytic category and assume that $X$ is unexceptional (Definition \ref{unexceptional}). Then  $\Hilb^W_\mu(X)$  is normal an therefore $M^\circ_{G}(X)$ is a normal complex space. It is equipped with a natural morphism $\phi=(\phi_-,\phi_+):M^\circ_{G}(X)\to \g\oplus\g$ given by 
$$\phi[g_1,g_2,x]=\bigl(\Ad_{g_1}\bar\mu(x),\Ad_{g_2}\bar\mu(x)\bigr).
$$
The values of $\phi_-$ and $\phi_+$ are regular elements of $\g$.
\par
Recall now that the (smooth) variety $G\times\mathscr{S}_{\g}$ is symplectic, with symplectic form given by
$\omega_-=\bigl\langle  dgg^{-1}\wedge d\bigl(\Ad_g (X)\bigr)\bigr\rangle,$
where $\langle \phi\wedge \psi\rangle(u,v)=\langle \phi(u),\psi(v)\rangle-\langle \phi(v),\psi(u)\rangle.$
The action of $G$ by left translations is Hamiltonian with moment map $(g,X)\mapsto\Ad( g) X$. It follows that
$M^\circ_G(X)$ is a symplectic quotient of $(G\times\mathscr{S}_{\g})\times (G\times\mathscr{S}_{\g})\times \Hilb^W_\mu(X)$ by $\sT\times \sT$, and hence it has a symplectic form on its regular locus, which yields a Poisson structure on $M^\circ_G(X)$. The action of $G\times G$ on $M_G^\circ(X)$ is Hamiltonian with moment map $(\phi_-,-\phi_+)$, and 
\begin{equation}\Hilb^W_\mu(X)=\{m\in  M_G^\circ(X);\,\phi_-(m)=\phi_+(m)\in \mathscr{S}_{\g}\}.\label{phi=phi}
\end{equation}
\begin{remark} The symplectic form on the regular locus of $M^\circ_G(X)$ (i.e.\ on the subset corresponding to the regular locus of $\Hilb^W_\mu(X)$) can be described as follows. The tangent space at such a point corresponding to an equivalence class $[g_1,g_2,x]$ is generated by the tangent space to $\Hilb^W_\mu(X)$ at $x$ and by a pair of left-invariant vector fields on $G$. The symplectic form is the sum of the symplectic form on $\Hilb^W_\mu(X)$ and of the Kostant-Kirillov-Souriau symplectic forms of the orbits of $\bar\mu(x)$ and $-\bar\mu(x)$. \end{remark}
\begin{remark} The above description of $M_G^\circ(X)$ can be generalised by taking the product of any number of copies of $G$ with $\Hilb_\mu^W(X)$ and quotienting by  the product of the same number of copies of $\sT$, where each copy of $\sT$ acts in the standard way on the corresponding copy of $G$ and by $x\mapsto \tau^{s}x$ on $\Hilb_\mu^W(X)$ with each exponent $s$  equal to $\pm 1$ and their sum  equal to $0$. In the case $X=T\times \h$ these are the Moore-Tachikawa varieties, as constructed in \cite{BiMT}.\end{remark}
The variety $G\times\mathscr{S}_{\g}$  has, in addition to $\omega_-$, a second symplectic form $\omega_+=\bigl\langle  g^{-1}dg\wedge d\bigl(\Ad_{g^{-1}}( X)\bigr)\bigr\rangle$. This time the action of $G$ by right translations is Hamiltonian with moment map  $(g,X)\mapsto -\Ad_{g^{-1}}(X)$. The  symplectic forms $\omega_-$ and $\omega_+$ are interchanged by a Cartan involution.
\begin{proposition} The Poisson quotient of $(G\times\mathscr{S}_{\g})\times (G\times\mathscr{S}_{\g})\times M_G^\circ(X)$ by $G\times G$, where the Poisson structure on the first factor is given by $\omega_-$ and on the second one by $\omega_+$,  is isomorphic to $\Hilb_\mu^W(X)$. \label{M_G-Hilb}
\end{proposition}
\begin{proof} The quotient of this variety by $G\times G$ is $\mathscr{S}_{\g}\times \mathscr{S}_{\g}\times M_G^\circ(X)$, and the moment map on this quotient 
is 
$$ \mathscr{S}_{\g}\times \mathscr{S}_{\g}\times M_G^\circ(X)\ni(S_1,S_2,m) \longmapsto 
\bigl(S_1+\phi_-(m),-S_2-\phi_+(m)\bigr).
$$
Since $\phi_+(m)$ is conjugate to $\phi_-(m)$, the zero level set of the moment map is the subscheme $\{m\in  M_G^\circ(X);\,\phi_-(m)=\phi_+(m)\in \mathscr{S}_{\g}\}$ of $M_G^\circ(X)$, which is equal to $\Hilb_\mu^W(X)$.
\end{proof}
\begin{remark} As Theorem \ref{C=W} and remarks in the introduction indicate, the category of transverse $W$-Hilbert schemes of $W$-invariant hypertoric varieties should be enlarged to include symplectic quotients of $\Hilb_\mu^W(X)$ by $W$-invariant tori of the structure torus $T$ of $X$. 
As shown in Appendix  \ref{appendix:ht-lifts}, such a symplectic quotient $Y$ depends only on the symplectic quotient $\bar X$ of $X$ by $T_0$ and not on $X$ itself. 
Let $\bar T=T/T_0$, $\bar\sT= \sT_{\scriptscriptstyle{\bar T,W}}$, and $\bar G= G_{{\scriptscriptstyle \bar T, W}}$. We can  define $M^\circ_{\bar G}(\bar X)$ as the quotient of  $\bar G\times \bar G\times Y$ by the action \eqref{action-tau} of $\bar\sT\times \bar\sT$. Equivalently,  $M^\circ_{\bar G}(\bar X)$ is the symplectic quotient of $M_G^\circ(X)$ by $T_0\times T_0$. The variety $M^\circ_{\bar G}(\bar X)$  plays the same role for $Y$ as  $M^\circ_{ G}(X)$ does for $\Hilb_\mu^W(X)$.
\end{remark}

\subsection{Twistor space and hyperk\"ahler structure\label{Z_G(X)}}
We shall now construct and analyse a twistor space for $M_G^\circ(X)$.
We  fix a hyperk\"ahler metric on $X$ and obtain a (singular model of) twistor space $Z_L^W(X)$ for $\Hilb^W_\mu(X)$ (\S\ref{twistor}), which is normal. Therefore  we can  define a (singular model of) twistor space for $M_G^\circ(X)$ as the fibrewise quotient
$$ Z_L(M_G^\circ(X))=G\times G\times Z_L^W(X)\big/\!\!\big/\sT\times \sT,$$
where the quotient is understood in the sense of analytic relations, as in the previous subsection. 
The symplectic form on the regular locus of $M_G^\circ(X)$, defined  in the previous subsection, yields  a fibrewise $\sO_{\oP^1}(2)$-valued symplectic form on the regular locus of $Z_L(M_G^\circ(X))$. We also have a compatible real structure, given by the real structure on   $Z_L^W(X)$, and the Cartan involution  with respect to $K$ on both copies of $G$. 
\par
In the case\footnote{
Remark \ref{W-product} implies that this is not a restriction for strongly $W$-invariant toric hyperk\"ahler manifolds.} when the cocharacters $\alpha_k$ span $\h$,  we can also use the Euclidean volume growth metric on $X$ (Remark \ref{evg}).
 Let us denote the  twistor space for $\Hilb^W_\mu(X)$ corresponding to this metric by $Z_0^W(X)$ and the twistor space for $M_G^\circ(X)$ by $Z_G^\circ(X)$. The moment map for the fibrewise action of $G\times G$ induces a map $Z_G^\circ(X)\to (\g\oplus\g)\otimes\sO_{\oP^1}(2)$. Therefore any section of $Z_G^\circ(X)$ defines a pair $\Phi_-(\zeta),-\Phi_+(\zeta)$ of $\g$-valued quadratic polynomials (with values regular elements of $\g$ for every $\zeta$). The $W$-invariant  curves defined by $\Phi_-(\zeta)$ and $\Phi_+(\zeta)$ coincide, owing to the definition of $M_G^\circ(X)$. 
We now address the question when such a pair $(\Phi_-(\zeta),\Phi_+(\zeta))$
of quadratic $\g$-valued polynomials defining a common curve can be lifted to a section of $Z_G^\circ(X)$.
\begin{proposition} Let $\Phi_-(\zeta),\Phi_+(\zeta)$ be a pair
of quadratic $\g$-valued polynomials, with regular values for every $\zeta\in \oP^1$, such that the corresponding $W$-invariant curves in $\h\otimes\sO_{\oP^1}(2)$ coincide. Let $C$ be this common curve and suppose that no component of $C$ is contained in any of the hyperplanes $H_i$, $i=1,\dots,d$. Let  $P_-,P_+$ denote the  $W$-invariant principal $T$-bundles on $C$ corresponding to $\Phi_-(\zeta),\Phi_+(\zeta)$,   and
suppose that, for each $i=1,\dots,d$ with $\alpha_i=m_i\omega_i$, where $m_i\in \oN$ and $\omega_i$ is primitive in $\sX_\ast(T)$, the Cartier divisor $(H_i)$ on $C$ admits a decomposition $\sum_{k=0}^{m_i}H^k_i$ into effective divisors, such that 
\begin{equation} P_+-P_-= \sum_{k=1}^d\sO_C(m_i)\Bigl[-\sum_{k=0}^{m_i}kH^k_i\Bigr]\otimes \omega_k.\label{difference}\end{equation}
Then the pair $(\Phi_-(\zeta),\Phi_+(\zeta))$ can be lifted to a section of $Z_G^\circ(X)$. If $C$
satisfies the genericity conditions (i)-(ii) of Theorem \ref{lift}(c), then the existence of the $H_i^k$ satisfying \eqref{difference} is necessary.\label{lift2} 
\end{proposition}
\begin{proof}  Let $\sP_\pm= \Hilb^W(P_\pm)$ be the $\sT_{\bar C}$-torsors  over the corresponding  section $\bar C$  of $\bigl(\h\otimes  \sO_{\oP^1}(2)\bigr)\big/ W$ (cf.\ Proposition \ref{torsors}). Local isomorphisms $\sP_{\pm}\to\sT_{\bar C}$ are given by local sections $g_\pm(\zeta)\in G$ such that $\Ad_{g(\zeta)^{-1}}(\Phi_\pm(\zeta))=\bar C(\zeta)\in \mathscr{S}_{\g}$ (cf.\ Remark \ref{torsors2}). Let $s(\zeta)$ be a local lift of $\bar C(\zeta)$ to $Z_0^W(X)$. As  in the proof of Theorem \ref{lift}, $s$ can be identified with a local section of $\sP_H=\Hilb_\pi^W(P_H)$, where $P_H$ is the $T$-bundle on the right hand side of \eqref{difference}. Therefore $s$ provides a local isomorphism $\phi:\sP_H\to\sT_{\bar C}$. Now observe that $g_+(\zeta)\phi(t)g_-(\zeta)^{-1}$  is invariant for the action \eqref{action-tau}. Therefore  the section $(\Phi_-(\zeta),\Phi_+(\zeta))$ can be lifted to $Z_G^\circ(X)$ if and only if we can find an open cover $\{U_i\}$ of $\oP^1$ with corresponding trivialisations $g_\pm^i,\phi^i$ such that $g_+^i(\zeta)\phi^i(t)g_-^i(\zeta)^{-1}=g_+^j(\zeta)\phi^j(t)g_-^j(\zeta)^{-1}$ on any $U_i\cap U_j$. Since $(g_-^j)^{-1}g_+^i, (g_-^i)^{-1}g_+^j\in Z_G(\bar \mu(x))$, they commute with $\phi^i(t),\phi^j(t)$, and hence the above equality on $U_i\cap U_j$ is equivalent to
$$ g_+^j(\zeta)^{-1}g_+^i(\zeta)\phi^{i}(t)\phi^{j}(t)^{-1}g_-^i(\zeta)^{-1}g_-^j(\zeta)=1,$$
which is precisely the \v{C}ech cocycle condition equivalent to $\sP_+-\sP_H-\sP_-=\sT_{\bar C}$ in $H^1(\oP^1,\sT_{\bar C})$. Owing to Proposition \ref{torsors}, this is equivalent to \eqref{difference}. For the second statement, the proof of Theorem \ref{lift}(c) shows that a local lift of $C$ to $Z_0(X)$ defines a decomposition of each $(H_i)$ as in the statement, and hence a local section of $P_H$. The above argument shows now that \eqref{difference} must hold if $(\Phi_-(\zeta),\Phi_+(\zeta))$ can be lifted to a section of $Z_G^\circ(X)$.
\end{proof}
\begin{remark} The above proposition shows that the sections of  $Z_G^\circ(X)$ are characterised by an {\em algebraic} condition. Hyperk\"ahler manifolds for which this holds are much simpler than the general ones. In particular, this is a necessary condition for a hyperk\"ahler manifold to arise as a hyperk\"ahler quotient of a finite dimensional vector space.
\end{remark}
We now address the question of the normal sheaf of such a section. Let us assume that it is torsion-free, hence locally free (cf.\ Proposition \ref{singZ}). 
\par
Denote by $\g^\circ$ the set of regular elements in $\g$ and by $\pi:\g^\circ\to \h/W$ the restriction of  the natural morphism induced by $G$-invariant polynomials. Denote by $Z_{\bar C}$ the submanifold 
of $\g\otimes \sO_{\oP^1}(2)$ defined by $\pi(x)=\bar C$. This manifold has been consider by D'Amorim Santa-Cruz \cite{SC}, who showed that a generic section of $Z_{\bar C}$ has normal bundle isomorphic to $\sO_{\oP^1}(1)^{\oplus (\dim G-\rank G)}$.
Let $s$ be a  section of $Z_G^\circ(X)$, not contained in the singular locus,  and let $(\Phi_-(\zeta),\Phi_+(\zeta))$ be the corresponding pair of quadratic polynomials, with a common $W$-invariant curve $C$ in $\h\otimes\sO_{\oP^1}(2)$. 
We consider deformations of such an $s$ which keep   $C$ fixed. We can arbitrarily deform the $T$-bundle $P_-$ and arbitrarily change the trivialisation of $\ad \Pi$ which yields $\Phi_+(\zeta)$ from $P_+$. Therefore the normal bundle $\sN$ of $s$ fits into the exact sequence
\begin{equation}0\longrightarrow \sN_{\Phi_-/Z_{\bar C}}\oplus \sO^{\oplus \dim G} \longrightarrow \sN \longrightarrow \bigoplus_{i=1}^{\rank G}\sO_{\oP^1}(2d_i)\longrightarrow 0.\label{lesN}\end{equation}
This implies, that if $\Phi_-(\zeta)$ is generic with $\sN_{\Phi_-/Z_{\bar C}}\simeq \sO_{\oP^1}(1)^{\dim G-\rank G}$, then the degree of any rank $1$ direct summand in $\sN$ is nonnegative. Hence, if $\sN\not\simeq \sO(1)^{\oplus 2\dim G}$, then  $\sN$ has a trivial summand, which must be isomorphic to a trivial summand in the first term of \eqref{lesN}, i.e.\ it comes from a conjugation of $\Phi_+(\zeta)$. If we now consider the sequence analogous to \eqref{lesN}, but with $\Phi_-$ replaced by $\Phi_+$, then since this trivial summand does not come from a conjugation of $\Phi_-(\zeta)$, it must be isomorphic to a direct summand in $\sN_{\Phi_+/Z_{\bar C}}$.
Thus $\sN_{\Phi_+/Z_{\bar C}}\not\simeq  \sO_{\oP^1}(1)^{\oplus (\dim G-\rank G)}$ in this case. Conversely, if either $\sN_{\Phi_-/Z_{\bar C}}$ or $\sN_{\Phi_+/Z_{\bar C}}$ is not isomorphic to $\bigoplus \sO(1)$, then it has a direct summand of degree $\geq 2$, and \eqref{lesN} implies that the same is true for $\sN$. We therefore conclude:
\begin{proposition} Let $s$ be a  section of $Z_G^\circ(X)$ with torsion-free normal sheaf and corresponding quadratic $\g$-valued polynomials $\Phi_{\pm}(\zeta)$ and  $W$-invariant curve $C$ in $\h\otimes\sO_{\oP^1}$.  The normal bundle of $s$ is isomorphic to  $\sO(1)^{\oplus 2\dim G}$ if and only if   $\sN_{\Phi_\pm/Z_{\bar C}}\simeq \sO_{\oP^1}(1)^{\oplus (\dim G-\rank G)}$.\hfill $\Box$\label{normal2}
\end{proposition}
\begin{remark} D'Amorim Santa-Cruz \cite{SC} has shown that the normal bundle of a section $\Phi$ 
of $\g\otimes\sO_{\oP^1}(2)$ is isomorphic to $\sO_{\oP^1}(1)^{\oplus (\dim G-\rank G)}$ if and only if
the bundle of centralisers $\{(\zeta,\rho)\in \oP^1\times\g;\,[\rho,\Phi(\zeta)]=0\}$ fills $\g$.
This latter condition can be also reformulated as nonvanishing of a determinant: let $p_1,\dots,p_n$ be generators of $\cx[\g]^G$ and view each $dp_i$ as a $G$-equivariant map $\g\to \g^\ast$. Then  $\sN_{\Phi_\pm/Z_{\bar C}}\simeq \sO_{\oP^1}(1)^{\oplus (\dim G-\rank G)}$ if and only if the system of linear equations
$$ \big\langle dp_i\bigl(\Phi(\zeta)\bigr), B\big\rangle=0,\quad i=1,\dots,n, \enskip \zeta\in \oP^1,$$
has only the trivial solution $B=0$.
\label{SC}
\end{remark}
\begin{remark} Toric hyperk\"ahler manifolds with structure torus $T_\oR$ and Euclidean volume growth form a semigroup with respect to the product $X_1\star X_2$ of $X_1$ and $X_2$ defined  as the hyperk\"ahler quotient of $X_1\times X_2$ by the antidiagonal action of $T_\oR$. This semigroup is isomorphic to the semigroup of codimension $3$ flats of the form \eqref{Hflats} such that the cocharacters $\alpha_i$ span $\h$. The definition of $M_G^\circ(X)$ and Proposition \ref{sT-monoid} imply immediately that if $X_1$ and $X_2$ are strongly $W$-invariant toric hyperk\"ahler manifolds with structure torus $T_\oR$ and Euclidean volume growth, then  the  hyperk\"ahler quotient of $M_G^\circ(X_1)\times M_G^\circ(X_2)$ by the diagonal action of $G$ is isomorphic to
 $ M_G^\circ(X_1\star X_2)$. An analogous statement for complex-symplectic quotients is also true, and it does not require the assumption that the cocharacters $\alpha_i$ span $\h$.
 \label{star-product}
\end{remark}

\subsection{Hyperk\"ahler structure of $\Hilb^W_\mu(X)$ via $L$-Nahm's equations\label{10}}
We can use the complex space $M_G^\circ(X)$ to describe the hyperk\"ahler structure of $\Hilb^W_\mu(X)$
analogously to the case $X=T^\ast \cx^d$ in \S\ref{special}.
\par
Let $X$ be an unexceptional $W$-invariant toric hyperk\"ahler manifold with the cocharacters $\alpha_i$ spanning $\h$. 
Let $M_K^\circ(X)$ denote the subscheme of the smooth locus of the Douady space  consisting of  real sections of $Z_G^\circ(X)$  with normal sheaf $\sO(1)^{\oplus 2\dim G}$. Every connected component $B$ of $M_K^\circ(X)$ is a smooth hypercomplex manifold  with a natural $K\times K$-invariant pseudo-hyperk\"ahler metric (recall that $K$ denotes the maximal compact subgroup of $G$), and for a generic complex structure the map $B\to M_G^\circ(X)$ which associates to a section its intersection with the fibre of $Z_G^\circ(X)$ is a local analytic isomorphism. The hypercomplex moment map for the action of $K\times K$ is $(\phi_K^-,-\phi_K^+)$, where
\begin{equation}\phi_K^\pm=\frac{1}{2}\bigl(-i(\Phi_\pm)_1,(\Phi_\pm)_0+i(\Phi_\pm)_2, i(\Phi_\pm)_2-(\Phi_\pm)_0\bigr),\label{hk-moment}
\end{equation}
and $(\Phi_\pm)_k$ is the coefficient of $\zeta^k$ in the quadratic polynomial $\Phi_{\pm}(\zeta)$ in Proposition \ref{normal2}.
\par
In Appendix \ref{appendix:B} we describe complete hyperk\"ahler manifolds $N_K^\pm(L)$, defined as moduli spaces of solutions to $L$-Nahm's equations. Choose $c_-,c_+\in (0,1)$ with $c_-+c_+=1$. 
We can form the hyperk\"ahler quotient of $N_K^-(c_-L)\times M_K^\circ(X)\times N_K^+(c_{+}L)$ by $K\times K$. It is a  moduli space $N_X(L;c_-,c_+)$  of  solutions $(T_0(t),T_1(t),T_2(t),T_3(t))$ to $L$-Nahm's equations on the interval $(-c_-,c_+)$, with simple poles at $t=c_{\pm}$ and a discontinuity at $t=0$. The discontinuity is described as follows:   there exists $m\in  M_K^\circ(X)$ such that $\lim_{t\to 0_\pm} (T_1(t),T_2(t),T_3(t)=\phi_K^\pm(m)$, where $\phi_K^\pm$ is given by \eqref{hk-moment} ($T_0$ is continuous at $t=0$).
\par
The moduli space $N_X(L;c_-,c_+)$ is a hypercomplex manifold, with a pseudo-hyperk\"ahler metric on each connected component. Theorem \ref{lift} and Proposition \ref{lift2}, together with the above discussion, imply:
\begin{theorem} The manifold $N_X(L;c_-,c_+)$ is isomorphic to an open subset of the Douady space $H_L(X)$ of Definition \ref{g_L}.
\hfill$\Box$ \label{N_X}
\end{theorem} 
The isomorphism here is understood as a diffeomorphism which preserves the hypercomplex structure and is an isometry with respect to the pseudo-hyperk\"ahler metric on each connected component.
\begin{remark} The choice of the splitting $1=c_-+c_+$  plays no role for the metric (cf.\ \S\ref{special}). One can therefore always choose $c_-=c_+=\frac{1}{2}$.
\end{remark}
\begin{remark} Proposition \ref{lift2} implies that if $m$ is a point of $H_L(X)$ such  that the corresponding curve $C$ satisfies the genericity conditions of Theorem \ref{lift}(c), then $C$ can be lifted to  $Z_G^\circ(X)$. Let $\Phi_-,\Phi_+$ be the quadratic $\g$-valued polynomials corresponding to $T$-bundles $P^{c_-L}|_C$ and $P^{c_+L}|_C$. It follows from Proposition \ref{normal2} that for a generic $C$ the normal sheaf of the lift of $C$ splits as $\bigoplus\sO(1)$ as soon as $\sN_{\Phi_\pm/Z_{\bar C}}\simeq \sO_{\oP^1}(1)^{\oplus (\dim G-\rank G)}$. In this case $m$ belongs to the open set in the above theorem.\end{remark}

\section{Completions for strongly $W$-invariant hypertoric varieties\label{completions}}

As mentioned in the introduction we expect the hyperk\"ahler metrics on $\Hilb^W_\mu(X)$ to be complete (in a suitable sense if $\Hilb^W_\mu(X)$ is a stratified manifold), at least in the strongly $W$-invariant case. We are not able to verify this directly from the definition of $N_X(L;c_-,c_+)$ since $M_K^\circ(X)$ is never complete. For example, if $X=\oH^{d}$, then $M_{U(d)}^\circ(X)$ is only an open subset of $\Mat_{d,d}(\oH)$. We would like  to replace $M_K^\circ(X)$ by a larger complete $K\times K$-invariant (stratified) hyperk\"ahler manifold $M_K(X)$ such that the hyperk\"ahler quotient $N_K^-(c_-L)\times M_K(X)\times N_K^+(c_{+}L)$ by $K\times K$ is still isomorphic to  $N_X(L;c_-,c_+)$. In this section we shall show that there is a natural candidate for $M_K(X)$.
\subsection{The hyperspherical variety}
\begin{theorem} If $X$ is a strongly $W$-invariant hypertoric variety, then the ring $A_G(X)=\cx[G\times \Hilb_\mu^W(X)\times G]^{\sT\times\sT}$  is finitely generated.\label{global-hol}
\end{theorem}
We begin the proof of Theorem \ref{global-hol} with some general remarks.
The coordinate ring of an affine hypertoric variety $X$, described in \S\ref{htv}, can be viewed as a subring of 
$$ \cx[T\times \h]\simeq \cx[t_\xi,z_1,\dots,z_n]/(t_\xi t_{\xi^\prime}-t_{\xi+\xi^\prime}, t_0-1),\quad \xi\in \sX^\ast(T),
$$
generated by the $z_i$ and $t_\xi\prod_{i=1}^d(\langle\alpha_i,z\rangle-\lambda_i)^{\langle\alpha_i,\xi\rangle_+}$, $\xi \in \sX^\ast(T)$.
If $X$ is strongly $W$-invariant, then this embedding is $W$-equivariant. We therefore have a $W$-equivariant morphism 
$T\times \h\to X$, which induces (cf.\ Remark \ref{functorial}) a morphism $W\text{-}\Hilb_\mu(T\times \h)\to W\text{-}\Hilb_\mu(X)$. 
We therefore obtain a canonical homomorphism $\cx[W\text{-}\Hilb_\mu(X)]\to\cx[W\text{-}\Hilb_\mu(T\times \h)]$, i.e.\ from $\cx[\Hilb_\mu^W(X)]$ to $\cx[\sZ_G]$, where $\sZ_G$ is the universal centraliser \eqref{ucentr} for $G=G_{\scriptscriptstyle{T,W}}$. This homomorphism is also injective since the varieties are irreducible and isomorphic on an open dense set.
\begin{example} The coordinate ring of the $A_1$ surface $xy+z^2=0$ embeds $\oZ_2$-equivariantly into the coordinate ring $\cx[t,s,z]/(ts-1)$ of $\cx^\ast\times \cx$ via $z\mapsto z$, $x\mapsto sz$, $y\mapsto -tz$. The universal centraliser $\sZ_{{\rm SL}_2(\cx)}$ is the $D_1$-surface $a_0^2-a_1^2 c=1$, and the corresponding embedding of $\cx\bigl[\Hilb^{\oZ_2}_\mu(A_1)]\simeq \cx[b_0,b_1,z]/(b_0^2-b_1^2c+c)$  into $\cx[\sZ_{{\rm SL}_2(\cx)}]$ is given by $c\mapsto c$, $b_0=a_1c$, $b_1=a_0$.\label{embeddA_1}
\end{example}
In the case $X=T\times \h$, $M_G^\circ(X)\simeq G\times \g^\circ$, where $\g^\circ$ denotes the set of regular elements. Since the codimension of $\g\setminus\g^\circ$ in $\g$ is two, any regular function on $M_G^\circ(X)$ extends to $G\times \g\simeq T^\ast G$ and we conclude that $M_G( T\times \h)\simeq T^\ast G$. This can be also interpreted as follows: we have a natural ring homomorphism
$$\psi:\cx[T^\ast G]\longrightarrow \cx[G]\otimes \cx[T^\ast G]\otimes \cx[G]\longrightarrow \cx[G]\otimes \cx[\sZ_G]\otimes \cx[G]\simeq \cx[G\times \sZ_G\times G],
$$
where the first map is given by the $G\times G$-action and the second one by the closed immersion $\sZ_G\hookrightarrow T^\ast G$. The statement that $M_G(X)\simeq T^\ast G$ is equivalent to the image of $\psi$ being equal to the subring $\cx[G\times \sZ_G\times G]^{\sT\times\sT}$. For a general $X$,  $\cx[\Hilb_\mu^W(X)]$ is a subring of $\cx[\sZ_G]$, and hence we can conclude that $A_G(X)$ is naturally a subring of $\cx[T^\ast G]$.
\begin{lemma} The ring $A_{{\rm SL}_2(\C)}(X)$ is finitely generated in the case when $X$ is an $A_{2k-1}$-surface.\label{fingenA}
\end{lemma}
\begin{proof} Let $X$ be given by the equation $xy+\prod_{j=1}^k(z^2-\lambda_j^2)=0$. Then $\Hilb^{\oZ_2}_\mu(X)$ is the $D_{k+1}$-surface (Example \ref{D_2}) with equation $b_0^2-b_1^2c+\prod_{j=1}^k(c-\lambda_j^2)=0$. As in Example \ref{embeddA_1} we obtain an embedding $\cx\bigl[\Hilb_\mu^{\oZ_2}(X)]$ into $\cx[\sZ_{{\rm SL}_2(\cx)}]$ given by $c\mapsto c$, $b_0$ (resp.\ $b_1$) equal to the even part of $(a_0+a_1z)\prod_{j=1}^k(z-\lambda_j)$ (resp.\ $z^{-1}\cdot$[odd part of $(a_0+a_1z)\prod_{j=1}^k(z-\lambda_j)$]), where we set $z^2=c$. The ring $A_{{\rm SL}_2(\C)}(X)$ is then the ring of the subvariety of $ \ssl_2(\cx)\times \Mat_{2,2}(\C)\times  \ssl_2(\cx)$
consisting of $(\Phi_-,A,\Phi_+)$ such that $ \det\Phi_-=\det\Phi_+$, $\Phi_-A=A\Phi_+$, and $\det A= -\prod_{j=1}^k(c-\lambda_j^2)$ where $c=-\det\Phi_-$.
\end{proof}
Let $\tilde F\subset \h$ and $F\subset \h/W$ be the subsets defined in the proof of Theorem \ref{normal}. Denote by $\delta\in \cx[\h]^W=\cx[\g]^G$ the polynomial defining the ramification divisor of the natural morphism $\h\to\h/G$.
Let $\g^{\rm rs}$ denote the subset of regular semisimple elements, i.e.\ the complement of $(\delta)$, and  by $\g^{\bullet}$ the subset of regular elements such that their semisimple part can be conjugated to the complement of $\tilde{F}$. We have
$$\g^{\rm rs}\subset \g^\bullet\subset \g^\circ.$$
Let $M_G^\bullet(X)=\phi^{-1}(\g^{\bullet}\times \g^{\bullet})\subset M_G^\circ(X) $. The codimension of the complement of $M_G^\bullet(X)$ in $M_G^\circ(X)$ is $\geq 2$, and since $M_G^\circ(X)$ is normal, it suffices to prove that the ring of restrictions to $M_G^\bullet(X)$ of functions in $A_G(X)$ is finitely generated.
\par
Let $R_1,\dots,R_s$ be distinct $W$-orbits of roots of $\g$. Each such orbit defines a polynomial $p_i\in\cx[\g]^G$ and $\delta=\prod_{i=1}^sp_i$. At any point of $M_G^\bullet(X)$ at most one $p_i$ vanishes (and only to order $1$), and hence $M_G^\bullet(X)=\bigcup_{i=1}^s M_G^{\bullet,i}(X)$, where $M_G^{\bullet,i}(X)$ is the subset where $\prod_{j\neq i} p_j\neq 0$. Since being finitely generated is a local property (cf.\ \cite[Prop. 6.3.3]{EGA}), it is sufficient to show that the ring $A_G^{i}(X)$ of restrictions of elements of $A_G^{i}(X)$ to $M_G^{\bullet,i}(X)$ is finitely generated for each $i$.

Now consider the restrictions of the polynomials $p_j$ to $\h$ and the subset $U_i$ of $\mu^{-1}(\h\setminus \tilde F)$ where $\prod_{j\neq i} p_j\neq 0$. 
As in the proof of Theorem \ref{normal}, $U_i$ is isomorphic to $V_i\times T^\ast T^\prime$, where $T^\prime$ is a subtorus of corank $1$ or $2$ and $V_i$ is an open subset of a $\Z_2$-invariant hypertoric variety $X_0$ with structure torus equal to the maximal torus of ${\rm SL}_2(\C)$ or ${\rm GL}_2(\C)$. Moreover, $V_i$ is the complement of $\mu^{-1}(0)$ and $X_0$ in the ${\rm GL}_2(\C)$-case is a $\Z_2$-quotient of the product of an $A_{2k-1}$-variety and $\C\times \C^\ast$. Let $H$ be the group $G_{{\scriptscriptstyle T, W}}$
for $X_0\times  T^\ast T^\prime$, i.e.\ $H= {\rm SL}_2(\C)\times T^\prime$ or $H= {\rm GL}_2(\C)\times T^\prime$. Lemma \ref{fingenA} implies that
$A_H(X_0\times  T^\ast T^\prime)$ is finitely generated. Now observe that the isomorphism $U_i\simeq V_i\times T^\ast T^\prime$ means that we have a finite branched covering
\begin{equation}G\times_HM_H^\circ(X_0\times  T^\ast T^\prime)\times_H G\longrightarrow  M_G^{\bullet,i}(X),\label{MHMG}\end{equation}
which induces an injective homomorphism $$A_G^{i}(X)\longrightarrow  R=\bigl(\cx[G]\otimes 
A_H(X_0\times  T^\ast T^\prime)\otimes \C[G]\bigr)^{H\times H}.$$
Since $H$ is reductive, $R$ is finitely generated, and  since the morphism \eqref{MHMG} is finite, $R$ is finitely generated as a module over $A_G^{i}(X)$. The Artin-Tate lemma implies now that 
$A_G^{i}(X)$ is finitely generated and the proof of Theorem \ref{global-hol} is complete.


\begin{definition} Let $X$ be a strongly $W$-invariant hypertoric variety with structure torus $T$ and $G=G_{{\scriptscriptstyle T, W}}$. The affine variety $M_G(X)=\Spec A_G(X)$ will be called the {\em hyperspherical variety} associated to $X$.
\end{definition}
Thus $M_G(X)$ is an affine completion of $M_G^\circ(X)$. The following properties of $M_G(X)$ are obvious from the definition:
\begin{proposition} $M_G(X)$ is a normal Poisson variety of dimension $2\dim G$ with a Hamiltonian action of $G\times G$. Moreover $\cx[M_G(X)]^{G\times G}\simeq \cx[\h]^W$, so that $M_G(X)$  is {\em coisotropic} in the sense of Losev \cite{Losev}.\qed
\end{proposition}
In the case when the cocharacters $\alpha_k$ span $\h$ we can  apply the construction $M_G^\circ(X)\mapsto A_G(X)\mapsto M_G(X)$ to the fibres of the twistor space $Z^\circ_G(X)$ and obtain thus a new twistor space $Z_G(X)$, the fibres of which are affine varieties. We denote by $M_K(X)$ denote the subscheme of the smooth locus of the Douady space  of $Z_G(X)$ consisting of  real sections  with normal sheaf $\sO(1)^{\oplus 2\dim G}$. The hyperk\"ahler quotient of 
$N_K^-(c_-L)\times M_K(X)\times N_K^+(c_{+}L)$ by $K\times K$ is still isomorphic to  $N_X(L;c_-,c_+)$.
\begin{proposition} Suppose that a connected component $M$ of $ M_K(X)$ is a complete (stratified) hyperk\"ahler manifold. Then the corresponding connected component of $N_X(L;c_-,c_+)$ is also complete. If, in addition, $M$  admits a $K\times K$-invariant K\"ahler potential for one of the complex structures, then $N_X(L;c_-,c_+)$ is isomorphic to $\Hilb_\mu^W(X)$ with respect to this complex structure. \label{everything}
\end{proposition}
\begin{proof} The first statement follows from the completeness of $M$ and of $N_K^-(c_-L)$, $N_K^+(c_{+}L)$. The second statement  follows from \cite[Theorem 1]{Bi1}.
\end{proof}
\begin{remark} The case when the cocharacters $\alpha_k$ do not span $\h$ is also easily included in the above scheme. According to Remark \ref{W-product}, $X$ is then the quotient by a finite abelian group $\Gamma$ of the product of $T^\ast T^\prime$ and a hypertoric variety $X^{\dprime}$ with structure torus $T^{\dprime}$ such that the defining cocharacters span $\h^{\dprime}=\Lie(T^{\dprime})$.  Moreover $W=W^\prime\times W^{\dprime}$ with
$W^\prime$ acting trivially on $T^{\dprime}$ and  $W^{\dprime}$ acting trivially on $T^{\prime}$. 
Since $\Gamma$ acts freely on $T^\ast T^\prime\times X^{\dprime}$, we obtain
$$\Hilb^W_\pi(X)\simeq \Hilb^W_\pi\bigl(T^\ast T^\prime\times X^{\dprime}\bigr)/\Gamma\simeq \bigl(\Hilb^{W^\prime}_{\pi^\prime} (T^\ast T^\prime)\times\Hilb^{W^{\dprime}}_{\pi^{\dprime}}( X^\prime)\bigr)/\Gamma.
$$
Let $G^\prime= G_{{\scriptscriptstyle T^\prime, W^\prime}}$ and $G^{\dprime}= G_{{\scriptscriptstyle T^{\dprime}, W^{\dprime}}}$.
Then $M_G(X)$ is the quotient  of $T^\ast G^\prime\times M_{G^{\dprime}}(X^{\dprime})$ by $\Gamma$.
The decomposition $\h_\oR=\h^\prime_\oR\oplus \h_\oR^{\dprime}$ yields a decomposition of the automorphism $L$ as $ L^\prime \oplus L^{\dprime}$. The hyperk\"ahler manifold $N_{X^{\dprime}}(L^{^{\dprime}};c_-,c_+)$ is obtained as above from $M_{K^{\dprime}}(X^{\dprime})$. We also have (cf.\ \S\ref{special}) the moduli space $N^\prime$ of solutions to $L^\prime$-Nahm equations describing the complete hyperk\"ahler metric on $\Hilb^{W^\prime}_{\pi^\prime} (T^\ast T^\prime)$.  The quotient by $\Gamma$ of   $N^\prime\times N_{X^{\dprime}}(L^{\dprime};c_-,c_+)$ is isomorphic to an open subset of the Douady space $H_L(X)$.
\label{notspanning}
\end{remark}

\subsection{Minuscule representations\label{minuscule}}
For any strongly $W$-invariant hypertoric variety $X$ we can find a set $S$ of characters of $T$ such that
\begin{itemize}
\item[(i)] $\cx[\h]$ and $x_\xi$, $\xi\in S $, generate $\cx[X]$;
\item[(ii)] $S$ equals to the set of weights of distinct irreducible representations $V_1,\dots, V_l$ of $G$.
\end{itemize}
We remark that the representation $\bigoplus_{i=1}^l V_i$ of $G$ is faithful (since $S$ generates $\sX^\ast(T)$) and that $G$ is a closed subvariety of $ \bigoplus_{i=1}^l\End(V_i)$ (since the $G$-orbit of ${\rm Id}\in \bigoplus_{i=1}^l\End(V_i)$ under the left multiplication is closed, owing to the Hilbert-Mumford criterion). Let $I(G)$ be the ideal of $G\subset \bigoplus_{i=1}^l\End(V_i)$. Observe that $T^\ast G$ can be viewed as the subvariety $\{(\Phi_-,\Phi_+,g);\, \Phi_+=g^{-1}\Phi_-g\}$ of $\g\times\g\times G$   ($\Phi_-,-\Phi_+$ are then the moment maps for the left and right action of $G$ with respect to the standard symplectic structure on $T^\ast G$).  It follows that $T^\ast G$ is a subvariety of $\g\oplus\g\oplus  \bigoplus_{i=1}^l\End(V_i)$ cut out by the ideal generated  by $I(G)$ and $\Phi_-A_i-A_i\Phi_+$, $A_i\in \End(V_i)$, $i=1,\dots,l$.
 \par
 We now embedd $X$ into $ \h \oplus\h \oplus \bigoplus_{i=1}^lZ(\End(V_i)^T)$
 via 
$$\bigl(z,(x_\xi)_{\xi\in S}\bigr)\longmapsto \bigl(z, z,\bigoplus_{i=1}^lx_\xi\Id_{V_i^\xi}\bigr),$$
and define $M_S(X)$ to be the Zariski closure of $(G\times G).X$ in $\g\oplus\g\oplus  \bigoplus_{i=1}^l\End(V_i)$, where the action of $G\times G$ is given by
$$(g,h).(\Phi_-,\Phi_+,A_1,\dots, A_l)=\bigl(\Ad_{g}(\Phi_-), \Ad_{h}(\Phi_+), gA_1h^{-1}, \dots, gA_lh^{-1}\bigr).$$
\begin{theorem} If each $V_i$ is minuscule, then  $M_S(X)$ is isomorphic to the hyperspherical variety $M_G(X)$.\label{minuscul}
\end{theorem}
\begin{proof}
We need to recall some facts about covariants.
Let $V$ be a finite-dimensional $W$-module. Then the space $\sP^W(\h,V)$ of $W$-equivariant polynomial maps from $\h$ to $V$ is a free graded $\cx[\h]^W$-module of rank $\dim V$ \cite{Che}. Similarly, if $V$ is a finite-dimensional $G$-module, then the space  $\sP^G(\g,V)$ of $G$-equivariant polynomial maps from $\g$ to $V$ is a free graded $\cx[\g]^G$-module of rank $\dim V^T$ (i.e.\ the dimension of the zero weight space). 
Chevalley's theorem implies that the restriction $\cx[\g]\to \cx[\h]$ induces an isomorphism $\cx[\g]^G\simeq \cx[\h]^W$. We shall denote this ring by $J$. For an arbitrary $V$, the restriction induces an injective homomorphism $r:\sP^G(\g,V)\to \sP^W(\h,V^T)$ between free graded $J$-modules of the same rank. Broer \cite{Broer} has shown that $r$ is an isomorphism if and only if $V$ is {\em small}, meaning that $2\phi$ is not a weight of $V$ for any dominant root $\phi$ of $\g$. 
We now consider the $G$-module $V^\ast \otimes V\simeq \End(V)$. In this case the modules of covariants $\sP^G\bigl(\g,\End(V)\bigr)$ and $\sP^W\bigl(\h,\End(V)^T\bigr)$  are associative algebras, since $\End(V)$ and $\End(V)^T$ are. In addition, $r$ is an algebra homomorphism.
If $\Lambda\subset \sX^\ast(T)$ are the weights of $V$ and $V\simeq \bigoplus_{\xi\in \Lambda} V^\xi$ is the weight space decomposition, then 
$$ \End(V)^T\simeq \bigoplus_{\xi\in \Lambda} \End(V^\xi).$$
Thus the centre of $\End(V)^T$ is $L^0=\bigoplus_{\xi\in \Lambda}\cx\cdot\Id_{V_\xi}\simeq \cx^\Lambda$. Observe that $\sP^W(\h,L^0)$ is the centre of the algebra $ \sP^W(\h,\End(V)^T)$. 
 \begin{lemma} Let $\rho:\g\to \End(V)$ be a minuscule representation of the Lie algebra of $G$. Any covariant $p\in\sP^W(\h,L^0)$ is of the form $r\bigl(\tilde p(\rho)\bigr)$, where $\tilde p$ is a polynomial in one variable of degree at most $|\Lambda|-1$.\label{rrestriction}
 \end{lemma}
 \begin{proof} If $V$ is minuscule, then $\End(V)$ is small \cite[\S 2]{Reeder} and, consequently, the restriction homomorphism $r:\sP^G\bigl(\g,\End(V)\bigr)\to  \sP^W\bigl(\h,\End(V)^T\bigr)$ is an isomorphism. Clearly, restrictions of powers of $\rho$ belong to 
 $\sP^W(\h,L^0)$. Since both $\sP^W(\h,L^0)$ and the polynomial ring in $\rho$ are free $J$-modules of the same rank, the claim follows.
 \end{proof}
 Let now $S$ be a set of characters as at the beginning of the subsection, with each $V_i$ minuscule. 
For each of the representations $V_i$, $i=1,\dots,l$, consider the following $W$-equivariant map $\h\to Z(\End(V_i)^T)$:
\begin{equation} p_i(z)=\bigoplus_{\xi\in \Lambda_i}\prod_{k=1}^d(\langle\alpha_k,z\rangle-\lambda_k)^{\langle\alpha_k,\xi\rangle_+}\Id_{V_i^\xi},\label{p_k(z)}\end{equation}
where $\Lambda_i$ denotes the set of distinct weights of $V_i$. Lemma \ref{rrestriction} implies that each $p_i(z)$ is of the form $r(\tilde p_i(\rho_i))$ where $\tilde p$ is a polynomial in one variable. This, together with the discussion at the beginning of the subsection, shows that $M_S(X)\subset \g\oplus\g\oplus  \bigoplus_{i=1}^l\End(V_i)$ is cut out by equations $\Phi_-A_i=A_i\Phi_+$ and
\begin{equation}F\bigl(\tilde p_1(\Phi_-^{-1})A_1,\dots \tilde p_l(\Phi_-^{-1})A_l\bigr)=0, \quad \forall F\in I(G).\label{eqsM_S}\end{equation}
The ideal $I(G)$ can be described by the following modification of a theorem of Chevalley  \cite[Theorem 4.27]{Milne} to the case of
algebraic monoids: 
\begin{lemma} Let $G$ be a closed subgroup of a reductive and normal algebraic monoid $M$ with a zero (i.e.\ an element $0\in M$ such that $0m=0=m0$ $\forall m\in M$.) Then there exists a finite-dimensional representation $L$ of $M$ such that $G$ is exactly the stabiliser of an element $\sigma\in L$. 
\end{lemma}
\begin{proof} Consider the action of $M$ on itself by left multiplication. Then the stabiliser of the closed subvariety $G$ is $G$, i.e.\ $G$ is the stabiliser of the ideal $I(G)\subset \cx[M]$. Choose a $G$-invariant finite-dimensional subspace $U$ of $I(G)$ which generates $I(G)$. Then $U$ is contained in an $M$-invariant finite-dimensional subspace $\tilde U$ of $\cx[M]$. This holds since we have a decomposition $\cx[M]=\bigoplus \End(V_\lambda)$ where the direct sum is over all (finite-dimensional) irreducible  representations of $M$ \cite{Vin}. 
Since $G$ is the stabiliser of $U$, it is also the stabiliser of the line $l=\Lambda^n U$ in the representation $L=\Lambda^n \tilde U$ of $M$, where $n=\dim U$. 
Let $L=L^\prime \oplus L^0$, where $L^0$ is the trivial subrepresentation, and let $\sigma=\sigma^\prime + \sigma^0$ be a generator of $l$ with $\sigma^\prime \in L^\prime$ and $\sigma^0\in L^0$. Observe that $0\in M$ acts as $0$ on $L^\prime$: any irreducible summand $V$ of $L^\prime$ is also an irreducible representation of the (reductive) group $M^\ast$ of units of $M$ and any eigenspace of $0$ in $V$ is  $M^\ast$-invariant. Therefore $\sigma^0\neq 0$, since otherwise $0$ stabilises $l$, i.e.\ $0\in G$.
Hence $G$ is exactly the stabiliser of $\sigma^\prime$ in $L^\prime$.
\end{proof}
We apply the above lemma to the monoid $M=\prod_{i=1}^l\End(V_i)$. Let $L$ and $\sigma$ be as in the lemma, so that $G$ is a subvariety of $\prod_{i=1}^l \End(V_i)$ consisting of those $(A_1,\dots,A_l)$ which satisfy  $(A_1,\dots,A_l)\sigma=\sigma$.
It follows that $M_S(X)$ is equal to a subvariety of $ \g\oplus\g\oplus  \bigoplus_{i=1}^l\End(V_i)$ consisting of $(\Phi_-,\Phi_+,A_1,\dots,A_l)$ such that $\Phi_-A_i=A_i\Phi_+$ for each $i$, and 
\begin{equation} 
(A_1,\dots,A_l)\sigma=\bigl(\tilde p_1(\Phi_-),\dots,\tilde p_l(\Phi_-)\bigr)\sigma .
\label{M_S-eq}\end{equation}
With this description we can prove:
\begin{lemma} The dimension of every fibre of the projection $M_S(X)\to \g$, $(\Phi_-,\Phi_+,A_1,\dots,A_l)\mapsto \Phi_-$, is equal to $\dim G$. \label{dimfibre}
\end{lemma}
\begin{proof} Let $B=(B_1,\dots,B_l)$ be a fixed element of $ \bigoplus_{i=1}^l\End(V_i)$  and consider the subvariety $F_B$ of $ \bigoplus_{i=1}^l\End(V_i)$ consisting of $(A_1,\dots,A_l)$  which satisfy \begin{equation}(A_1,\dots,A_l)\sigma=(B_1,\dots,B_l)\sigma.\label{ABsigma}\end{equation} It is enough to show  that $\dim F_B=\dim G$ for every $B$. This is certainly true for invertible $B$, so that $\dim F_B\geq \dim G$ for every $B$. To prove the opposite inequality it is enough, owing to the semicontinuity of dimension and the invariance of \eqref{ABsigma} under scaling, to consider $B=0$.   The group preserving the line $\langle \sigma\rangle\subset L$ is $G^\prime=\cx^\ast\times G$
and, hence, $F_0=\{(A_1,\dots,A_l);\,(A_1,\dots,A_l)\sigma=0\}$ is equal to the variety of those $(A_1,\dots,A_l)$  which map to $\ol{G^\prime}\setminus G^\prime\subset\End(L)$. Since the dimension of the latter variety is $\dim G$, we have $\dim F_0=\dim G$. 
\end{proof}
\begin{lemma} $M_S(X)$ is normal.\label{M_S-normal}
\end{lemma}
\begin{proof} The strategy is the same as in \S\ref{normality}. We observe that $M_S(X)$ is isomorphic to $M_G^{\rm ss}(X)=G\times_T\times X\times_T G$ away from $(\delta)$ and Proposition \ref{tbu} implies that  
 $M_S(X)$ is an affine blow up of $M_G^{\rm ss}(X)$ transverse to $(\delta)$. We consider now again the subset $\g^\bullet\subset \g$ defined in the proof of Theorem \ref{global-hol} and consider the open subset $U$ of $M_S(X)$ defined by $\Phi_-,\Phi_+\in\g^\bullet$. Lemma \ref{dimfibre} implies that Theorem \ref{S_2} holds for $j:U\to M_S(X)$. On the other hand, using the description in \eqref{M_S-eq} and local models as in the proof of Theorem \ref{global-hol} shows $U$ is isomorphic to the corresponding open subset of $M_G(X)$. In particular $U$ is normal, and the argument of Braverman, Finkelberg, and Nakajima, recalled at the beginning of \S\ref{normality}, shows that $M_S(X)$ is normal. 
 \end{proof}
 We can now finish the proof of Theorem \ref{minuscul}. Restricting \eqref{M_S-eq} to $\Phi_-=\Phi_+\in \mathscr{S}_\g$ shows that $(\Phi_-,\Phi_+)^{-1}(\Delta\mathscr{S}_\g)\simeq \Hilb^W_\mu(X)$. Therefore we have a $\sZ_G\times \sZ_G$-invariant dominant morphism $G\times \Hilb^W_\mu(X)\times G\to M_S(X)$, which induces a dominant morphism $M_G(X)\to M_S(X)$. The latter is an isomorphism between open subsets $U_1=\phi^{-1}(\g^\circ\times g^\circ)$ and $U_2=(\Phi_-,\Phi_+)^{-1}( \g^\circ\times g^\circ)$. We therefore have an injective homomorphism $h:\cx[M_S(X)]\to \cx[M_G(X)]$, and the restriction of an $f\in M_G(X)$ to $U_1$ yields a regular function $\tilde f$ on $U_2$. Owing to Lemmata \ref{dimfibre} and \ref{M_S-normal}, $\tilde f$ extends to a regular function on $M_S(X)$. Therefore $h$ is an isomorphism and the proof is complete.
\end{proof}
We can also describe explicitly the twistor space $Z_G(X)$. Since each $V_i$ is miniscule, there exist $m_i\in \oN$ such that  $\sum_{k=1}^d|\langle \alpha_k,\xi\rangle|=m_i$ for every weight $\xi$ in $V_i$, $i=1,\dots,l$. We obtain $Z_{G}(X)$ by imposing the equations defining  $M_{G}(X)$ as a subvariety of $\g\oplus \g\oplus \bigoplus_{i=1}^l \End(V_i)$ onto the vector bundle
\begin{equation} \g\otimes\sO_{\oP^1}(2)\oplus \g\otimes\sO_{\oP^1}(2)\oplus \bigoplus_{i=1}^l \End(V_i)\otimes\sO_{\oP^1}(m_i).\label{vectorbundle}\end{equation}
This description  identifies $Z_G(X)$ as a complex space over $\oP^1$. The real structure and the $\sO_{\oP^1}(2)$-valued fibrewise symplectic form are defined as in  \S\ref{Z_G(X)}, i.e.\ they do not necessarily arise from such structures on the vector bundle \eqref{vectorbundle}.
\begin{remark} Suppose that $\bigoplus_{i=1}^l \End(V_i)$ admits a linear map $\tau$ which preserves
$\bigoplus_{i=1}^lZ(\End(V_i)^T)$ and, for every $\xi\in S$, sends ${\rm Id}_{V_i^\xi}$ to $(-1)^{\sum_{j=1}^d\langle \alpha_j,\xi\rangle_+}{\rm Id}_{V_{i^\prime}^{-\xi}}$ for some $i^\prime$ (in particular, $S$ must be invariant under the involution $\zeta\to -\zeta$). The map $\tau$ must preserve each subspace $E_k=\bigoplus_{m_i=k} \End(V_i)$.
The real structure of $Z_G(X)$ is then the restriction of a real structure $\sigma$ on the vector bundle \eqref{vectorbundle}, given by $\Phi\mapsto -\Phi^\ast/\bar\zeta^2$ on the two copies of $\g\otimes\sO_{\oP^1}(2)$ (the asterisk denotes the negative of the Cartan involution with respect to the maximal compact Lie subalgebra $\fK$), and by  
 $\sigma(e)=\tau(e)^\ast/\bar\zeta^k$ on $E_k\otimes \sO_{\oP^1}(k)$ (the asterisk denotes the Hermitian transpose).\label{rrealstr}
\end{remark}
We give several examples in which we obtain a complete hyperk\"ahler metric on $M_K(X)$.
\begin{example} Let $X$ be the hypertoric variety determined by cocharacters $\alpha_i$ equal to $e_i$,  $i=1,\dots, d$, each with multiplicity $m$, and arbitrary scalars $\lambda_j$, $j=1,\dots,m$. $X$ is described by equations $v_iw_i=\prod_{j=1}^m(z_i-\lambda_j) $, $i=1,\dots d$, i.e.\ $X$ is the product of $d$ copies of an $A_{m-1}$ surface. The group $W=\Sigma_d$ acts by permutations, and $G=GL_d(\cx)$.
In this case, $M=M_G(X)$ can be described  as a hyperk\"ahler quotient of a quaternionic vector space.
The vector space is the representation space of the double $A_{m+1}$-quiver with dimension vectors all equal to $d$: $V=\bigoplus_{j=1}^{m}\bigl(\Hom(V_{j-1},V_{j})\oplus  \Hom(V_{j},V_{j-1})\bigr)$, where $V_j=\cx^d$ for every $j$. $M$ is a hyperk\"ahler quotient of $V$ by $H=U(d)^{m-1}$ acting in the standard way at the inner $m-1$ vertices. Let us choose a complex structure $J$ and write the corresponding complex coordinates on $V$ as $(A_j,B_j)_{j=1}^m$, where $(A_j,B_j)\in \Hom(V_{j-1},V_{j})\oplus  \Hom(V_{j},V_{j-1})$. The complex moment map equations are then 
\begin{equation} A_jB_j-B_{j+1}A_{j+1}=\kappa_j,\quad j=1,\dots, m-1,\label{mmap}
\end{equation}
for some complex scalars $\kappa_i$. The $H^\cx$-invariant polynomials  are $F(A_jB_j)$, $j=1,\dots, m-1$, $F\in \cx[\gl_d(\cx)]^{{\rm GL}_d(\cx)}$, as well as products of $A_i,B_j$ corresponding to any path starting and ending at  outer vertices.
The moment map equations imply that on the zero set of the moment map the ring of invariant polynomials is generated by $\Psi_1=A_mB_m$, $\Psi_2=B_1A_1$, $A=A_m\cdots A_1$, $B=B_1\cdots B_m$.
Owing to \eqref{mmap} these satisfy the equations
\begin{equation} AB=\Psi^m_1 +\sum_{j=1}^m c_j\Psi^j_1,\quad BA=\Psi^m_2 +\sum_{j=1}^m c_j^\prime\Psi^j_2,\label{AB,BA}\end{equation}
where the coefficients $c_j$ and $c_j^\prime$ are polynomials in the $\kappa_j$, as well as
\begin{equation}A\Psi_2=\Psi_1B+(\kappa_1+\dots+\kappa_{m-1})A.\label{APsi}\end{equation}
Moreover, equations \eqref{mmap} imply that 
$$F(\Psi_2)=F(\Psi_1+\kappa_1+\dots+\kappa_{m-1}), \enskip \forall F\in \cx[\gl_d(\cx)]^{{\rm GL}_d(\cx)}.
$$
 If we set  $\Phi_-=\Psi_1+w$, $\Phi_+=\Psi_2+w-\sum_{j=1}^{m-1}\kappa_j$, $w\in \cx$, then the characteristic polynomials  of $\Phi_-$ and $\Phi_+$ coincide, and (with an appropriate choice of $w$ and of the $\kappa_j$) equations \eqref{AB,BA}-\eqref{APsi} become
 \begin{equation}AB=\prod_{j=1}^m(\Phi_--\lambda_j), \quad BA=\prod_{j=1}^m(\Phi_+-\lambda_j), \quad A\Phi_+=\Phi_-A,\quad \Phi_+B=B\Phi_-.\label{ABm}\end{equation}
 This describes $M$ as an affine variety with respect to the chosen complex structure $J$. The real structure on the twistor space is given as in Remark \ref{rrealstr} with $\tau(A,B)=\bigl( (-1)^m B, A\bigr)$.
 In order to see that the complex-symplectic form of the complex-symplectic quotient coincides with the one on $M_G(X)$ given in \S\ref{M_G^circ}, it is enough to compute it on the open dense subset where $A,B,\Phi_-,\Phi_+$ are all diagonalisable under the action of ${\rm GL}_d(\C)\times{\rm GL}_d(\C)$. This computation is straightforward if lengthy, and is left to the reader. Therefore the twistor spaces of $M$ and $M_K(X)$ are isomorphic, and hence $M$ is isomorphic to $M_K(X)$ as a hyperk\"ahler manifold. Thus the hyperk\"ahler metric on $\Hilb_\mu^W(X)$ can be described as the moduli space $N_X(L;c_-,c_+)$ of $\fU(d)$-valued solutions to Nahm's equations with the discontinuity given by $M$.
 Since $M$ satisfies the assumptions of Proposition \ref{everything} (the existence of a K\"ahler potential  follows from the existence of a bounded from below ${\rm U}(d)^{m+1}$-invariant hyperk\"ahler potential on $V$), we conclude that the hyperk\"ahler metric on $\Hilb_\mu^W(X)$ is complete. 
 \par
Observe that for $m=1$ the above construction gives $T^\ast \Mat_{d,d}(\cx)$ arising  in the discussion of the special case $X=\cx^{2d}$ in  section \ref{special}.
\label{findimq}
\end{example}
\begin{example} Consider the special case of the previous example with $d=2n$ even, $m=2$, and $\lambda_1=-\lambda_2$. Let $J\in {\rm U}(2n)$ be the matrix defining the standard symplectic form on $\cx^{2n}$. The involution 
$$ \sigma:(\Phi_-,\Phi_+,A,B)\longmapsto (J\Phi_-^TJ,J\Phi_+^TJ,-JB^TJ,-JA^TJ)$$
preserves the hyperk\"ahler structure of the manifold $M$ in Example \ref{findimq}, and, for any complex structure, the submanifold $M^\sigma$ of its fixed points is biholomorphic to the subvariety of $\ssp_{2n}(\cx)\oplus\ssp_{2n}(\cx)\oplus\Mat_{2n,2n}(\cx)$ consisting of 
$(\Phi_-,\Phi_+,C)$ such that the characteristic polynomials of $\Phi_-$ and $\Phi_+$ coincide, and (with $\lambda=\lambda_1$) 
 \begin{equation}CJC^TJ=
\lambda^2-\Phi_-^2, \quad JC^TJC=\lambda^2-\Phi_+^2, \quad C\Phi_+=-\Phi_-JC^TJ.\label{involutionC}\end{equation} 
The group ${\rm Sp}(n)\times {\rm Sp}(n)$ acts on $M^\sigma$ and the fixed point set $X$ of the diagonal maximal torus 
is easily seen to be $n$ copies of the $A_1$-surface. We conclude that we obtain a   complete hyperk\"ahler metric
on $\Hilb_\mu^W(X)$ for this $X$ (the structure torus and the group $W$ are the maximal torus and the Weyl group of ${\rm Sp}_{2n}(\cx)$).  
\label{hsph2}
\end{example}
\begin{remark} The last example can be obviously generalised to any $m\in 2\oN$. In particular, if $d=2$, the construction of \S\ref{10} yields a description of (a subfamily of)  gravitational instantons of type $D_k$ ($m=2k-2$) as moduli spaces of $\su(2)$-valued solutions to Nahm's equations on $(0,1)$ with simple poles (and fixed residues) at $0,1$ and a discontinuity in the middle. The discontinuity is given by the fixed point set of the hyperk\"ahler manifold of Example \ref{findimq} under the involution \eqref{involutionC} (observe that, for $d=2$, $-JC^TJ=C_{\rm adj}$), see also the proof of Lemma \ref{fingenA}. The subfamily obtained this way is essentially the one in Example \ref{D_2}. Alternatively, we obtain the full family of $D_k$ gravitational instantons as hyperk\"ahler quotients by $S^1$ of moduli spaces of $\u(2)$-valued solutions to Nahm's equations with the discontinuity described in  Example \ref{findimq} (with $d=2$ and $m=2k$). This second description of $D_k$ gravitational instantons also follows by combining the results of  \cite{CH} and \cite{BC} (see also Example \ref{SO(3)-mon} below).
\end{remark}
In general, equations \eqref{M_S-eq} are complicated.
The next example indicates a  method of identifying $M_K(X)$ without having to deal with the full set of equations.
\begin{example} Let $X$ be a $\oZ_2$-invariant hypertoric variety with structure torus equal to the maximal torus of ${\rm GL}_2(\cx)$ with  defining hyperplanes given by $kz_1+lz_2=0$ and $lz_1+kz_2=0$, $k,l\in\oZ$, $k\neq l$. $X$ is isomorphic to $\cx^4/\oZ_{k^2-l^2}$ and the   equations \eqref{monomials} corresponding to characters $\xi=e_i^\ast$, $\xi^\prime=-e_i^\ast$, $i=1,2$, are
$$ x_1y_1=(kz_1+lz_2)^k(lz_1+kz_2)^l,\quad x_2y_2=(kz_2+lz_1)^k(kz_2+lz_1)^l.
$$
Setting 
$$A=g_-\begin{pmatrix}x_1 & 0\\ 0 & x_2\end{pmatrix} g_+^{-1},\enskip B=g_+\begin{pmatrix}y_1 & 0\\ 0 & y_2\end{pmatrix}g_-^{-1},\enskip \Phi_\pm=g_\pm\begin{pmatrix}z_1 & 0\\ 0 & z_2\end{pmatrix}g_\pm^{-1},$$
we obtain equations
$$AB = \bigl(k\Phi_-+l(\Phi_-)_{\rm adj}\bigr)^k \bigl(l\Phi_-+k(\Phi_-)_{\rm adj}\bigr)^l,\enskip 
BA= \bigl(k\Phi_++l(\Phi_+)_{\rm adj}\bigr)^k \bigl(l\Phi_++k(\Phi_+)_{\rm adj}\bigr)^l.$$
Since, for a $2\times 2$ matrix $X$, $X_{\rm adj}=-X+\tr X$, we can rewrite these equations as
\begin{gather}AB=((k-l)\Phi_-+l\tr \Phi_-)^k((l-k)\Phi_-+k\tr\Phi_-)^l,\label{equations1}\\ BA= ((k-l)\Phi_++l\tr \Phi_+)^k((l-k)\Phi_++k\tr\Phi_+)^l.\label{equations2}\end{gather}
After replacing $(A,B)$ with $(A/(k-l), B/(l-k))$,  the right sides of these equations become monic polynomials in $\Phi_\pm$.
Introducing a parameter $\zeta$, so that $A,B,\Phi_\pm$ become sections of vector bundles over $\oP^1$, we obtain a ``twistor space", a generic section of which is also a section of the twistor space of $M_{{\rm U}(2)}(X)$ (cf.\ Remark \ref{bar{Z}}).
Now consider, as in Example \ref{findimq}, the representation space $V$ of the double $A_{k+l}$-quiver with dimension vectors all equal to $2$. Let $M$ be the hyperk\"ahler quotient of $V$ by ${\rm SU}(2)^{k+l-1}$ acting in the standard way at the inner vertices. $M$ has a ${\rm U}(2)^2\times {\rm U}(1)^{k+l-1}$ symmetry and the moment map for the ${\rm U}(1)^{k+l-1}$-action is given by (linear combinations) of parameters $\lambda_j$ in \eqref{ABm}. On the other hand, the moment map for ${\rm U}(1)\subset {\rm U}(2)^2$ acting via $(A,B)\mapsto (s A, s^{-1} B)$ is $\tr \Phi_-/2=\tr\Phi_+/2$. It follows that equations \eqref{equations1}-\eqref{equations2} are satisfied by the twistor space of the hyperk\"ahler quotient $\bar M$ of $M$ by ${\rm U}(1)^{k+l-1}$ acting via an appropriate homomorphism ${\rm U}(1)^{k+l-1}\to {\rm U}(1)^{k+l}$. A generic solution of these equations (with parameter $\zeta$, i.e.\ a $\oP^1$) is a section of both the twistor space of $\bar M$ and of the twistor space of  $M_{{\rm U}(2)}(X)$ . Therefore the hyperk\"ahler structures of $\bar M$ and  $M_{{\rm U}(2)}(X)$  coincide\footnote{As in Example \ref{findimq}, one verifies that the complex-symplectic forms coincide on the open dense subset where $A,B,\Phi_-,\Phi_+$ are all diagonalisable under the action of ${\rm GL}_2(\C)\times{\rm GL}_2(\C)$.} on an open dense subset, hence on the whole smooth locus. 
\end{example}

\section{Examples: monopole moduli spaces\label{monopoles}}
We shall now discuss how  moduli spaces of monopoles (with singularities) on $\R^3$  fit into the above scheme.
\par
 Let $H$ be a compact Lie group of rank $r$ and of adjoint type (this is necessary if we want to introduce ``minimally charged'' Dirac singularities, and irrelevant when there are no singularities). We consider a moduli space of monopoles on $\R^3$ with structure group $H$ and Dirac-type singularities at a finite collection of points. More precisely, asymptotic and topological data are determined by:
\begin{enumerate}
    \item a generic element $\omega$ in a Cartan subalgebra of $H$, the mass of the monopole (the genericity assumption corresponds to monopoles with maximal symmetry breaking at infinity);
    \item a non-abelian magnetic charge $\gamma_{{m}}=\sum_{\mu=1}^{r}{n_\mu \alpha_\mu^\vee}$, where $\alpha_1^\vee, \dots, \alpha_r^\vee$ are the simple coroots of $H$ determined by $\omega$ and $n_1,\dots, n_r\in \Z_{\geq 0}$;
    \item a collection of distinct points $\underline{\lambda}_1,\dots,\underline{\lambda}_k\in\R^3$, with $\underline{\lambda}_l = (\lambda_l^\R, \lambda_l)$ in a decomposition $\R^3 = \R\times\C$, each labelled by the choice of a fundamental cominuscule coweight $\gamma_l$ for $H$.
\end{enumerate}
The last requirement corresponds to Dirac-type singularities of the monopoles with ``minimal'' charge, leading to complete moduli spaces. The moduli space $\mathcal{M}$ of monopoles with structure group $H$ on $\R^3$ with mass $\omega$, non-abelian charge $\gamma_{m}$ and Dirac-type singularity of charge $\gamma_l$ at $\underline{\lambda}_l$ is expected to be a complete hyperk\"ahler manifold of dimension $4(n_1+\dots+n_{r})$ \cite{MRB}. An asymptotic region of this moduli space is qualitatively described as follows (cf.\  \cite{Taubes,Fos-Ross},  and upcoming follow-up work by C. Ross and the second author): for each $\mu=1,\dots, r$ we have $n_\mu$ indistinguishable distinct points $p^1_\mu,\dots, p^{n_\mu}_\mu$ in $\R^3\backslash \{ \underline{\lambda}_1,\dots,\underline{\lambda}_k\}$; each of these points is endowed with a phase parameter in $S^1$ and corresponds to an $SU(2)$ charge 1 monopole embedded in $H$ via the homomorphism that identifies the coroot of $SU(2)$ with $\alpha_\mu^\vee$. This asymptotic picture leads one to consider the hypertoric variety defined by the following collection of hyperplanes. Let $A$ denote the matrix $2\Id-C$, where $C$ is the Cartan matrix of $H$. We then consider hyperplanes
\[
A_{\mu \nu}\, z_{\mu i} - A_{\nu \mu}\, z_{\nu j} = 0 = \alpha_\mu (\gamma_l)\, z_{\mu i} - \lambda_l
\]
for $\mu,\nu = 1,\dots,r$, $i=1,\dots,n_\mu$, $j=1,\dots,n_\nu$ and $l=1,\dots,k$. Note that since $\gamma_l$ is a fundamental cominuscule coweight, $\alpha_\mu (\gamma_l)\in \{ 0,1\}$ and there exists a unique $\mu$ such that $\alpha_\mu (\gamma_l)=1$; moreover $A_{\mu \nu}$ and $A_{\nu\mu}$ are either both zero or both non-zero with at least one of them equal to 1. These facts can be used to show that $X$ is smooth for generic choices of $\lambda_1,\dots,\lambda_k$. More importantly, note that $X$ is strongly $W$-invariant for the Weyl group $W=\prod_{\mu=1}^r\Sigma_{n_\mu}$, where $\Sigma_{n_\mu}$ permutes the points $p^1_\mu,\dots, p^{n_\mu}_\mu$.
This means that the structure torus of $X$ is the maximal torus of $\prod_{\mu=1}^r {\rm GL}_{n_\mu}(\cx)=G$. Finally, the positive constants $\alpha_1(\omega),\dots,\alpha_r(\omega)$ determine the choice of positive definite $L$, i.e.\  the choice of a flat hyperk\"ahler metric on $T^\ast T$, where $ T=(\C^\ast)^{n_1+\dots+n_r}$ is the structure torus of $X$. We expect that the moduli space $\mathcal{M}$ of singular monopoles can be identified with ${\Hilb}^W_\mu(X)$. Recent work of Braverman, Finkelberg, and Nakajima \cite{BFN2} (together with our Theorem \ref{C=W}) provides strong evidence that this is the case when $H$ is a (product of) simply laced Lie group(s). For an arbitrary $H$ and nonsingular monopoles, evidence could be obtained by checking that the vanishing conditions on spectral curves of such monopoles, written down by Murray \cite{Mu}, coincide with those of Theorem \ref{lift}.
\par
We would like to calculate the coordinate ring of $X$ arising in this way and give a description of the corresponding variety $M_G(X)$. This describes, according to Proposition \ref{M_G-Hilb},  ${\Hilb}^W_\mu(X)$ as a Poisson quotient of the product of $M_G(X)$ and two copies of $G\times \mathscr{S}_{\g}$. In the cases where we can identify the hyperk\"ahler metric on $M_G(X)$, it also gives a description of the hyperk\"ahler metric as the $L^2$-metric on a moduli space of solutions to $L$-Nahm's equations (Theorem \ref{N_X}). We remark this moduli space is very different from the ones used before to describe moduli spaces of monopoles; this is the case even for nonsingular monopoles for classical groups (of rank $>1$), cf.\ \cite{HM}.
\par
We shall work this out explicitly when $H={\rm PU}(r+1)$. Then the $r$ fundamental coweights are all cominuscule. We assume that there are $k_\mu$ Dirac singularities with charge $\varpi_\mu^\vee$ for all $\mu=1,\dots,r$. We also fix an orientation for the Dynkin diagram of $H$. The hyperplanes of $X$ are then
\[
z_{\mu, i} - z_{\mu +1, j} = 0 = z_{\mu, i} - \lambda_{\mu, l}
\]
for $\mu =1,\dots,r$, $i=1,\dots, n_\mu$, $j=1,\dots, n_{\mu +1}$ and $l=1,\dots,k_\mu$. Here $(z_{\mu, i})$ are coordinates on $\h\simeq \h^\ast = \bigoplus_\mu \C^{n_\mu}$ ($\h$ denotes the Lie algebra of the structure torus of $X$, not of $H$).
\par
To calculate a basis $S$ of the coordinate ring of $X$ we argue as follows.
The hyperplanes
\[
\xi_{\mu, i} - \xi_{\mu +1, j} = 0 = \xi_{\mu, i}
\]
subdivide the space $\h^\ast_\R$ into cones. A set of generators $\{ z_{\mu,i}, x_\xi, \xi\in S\}$ for the coordinate ring of $X$ then corresponds to a subset $S$ that generates each of these cones over $\Z$. It is clear we can take $S$ to be the collection
\[
\pm e_{\mu,i}^\ast, \qquad \pm (e_{\mu, i}^\ast + e_{\mu+1,j}^\ast),
\]
i.e.\  the collection of vectors generating the boundaries of the above cones, which is clearly $W$--invariant. This gives us functions
\[
x_{\mu,i}^\pm,\, x_{\mu\to \mu+1,ij}^\pm.
\]
The relations between these functions correspond to linear dependence relations between the vectors in $S$. Under our assumptions, these are given by expressing all the $x_\xi$'s in terms of the $x_{\mu,i}$'s. We then obtain the following relations:
\begin{gather}
    x_{\mu,i}^+ x_{\mu,i}^- = \prod_{l=1}^{k_\mu}(z_{\mu,i}-\lambda_{\mu,l})\prod_{j=1}^{n_{\mu+1}}{(z_{\mu,i}-z_{\mu+1,j})}\prod_{j=1}^{n_{\mu-1}}{(z_{\mu-1,j}-z_{\mu,i})},\label{onemon}\\
x^+_{\mu,i}x^+_{\mu+1,j} = x^+_{\mu\to\mu+1,ij} (z_{\mu,i}-z_{\mu+1,j}),\label{twomon} \\
x^-_{\mu,i}x^-_{\mu+1,j} = x^-_{\mu\to\mu+1,ij} (z_{\mu,i}-z_{\mu+1,j}),\label{threemon}\\
x^+_{\mu\to\mu+1,ij}x^-_{\mu\to\mu+1,ij} =\prod_{l=1}^{k_\mu}(z_{\mu,i}-\lambda_{\mu,l})\prod_{l=1}^{k_{\mu+1}}(z_{\mu+1,j}-\lambda_{\mu+1,l})\,\cdot \label{fourmon} \\ \cdot\prod_{l\neq j}^{n_{\mu+1}}{(z_{\mu,i}-z_{\mu+1,l})} \prod_{l\neq i}^{n_{\mu}}{(z_{\mu,l}-z_{\mu+1,j})}
\prod_{l=1}^{n_{\mu-1}}{(z_{\mu-1,l}-z_{\mu,i})}\prod_{l=1}^{n_{\mu+2}}{(z_{\mu+1,j}-z_{\mu+2,l})}.\notag
\end{gather}
The representation $V$ corresponding to $S$ is then $\bigoplus_\mu{T^\ast\C^{n_\mu}\oplus T^\ast(\C^{n_\mu}\otimes\C^{n_{\mu+1}})}$. Since the irreducible summands are all minuscule, we can use the description of the previous section, i.e.\ view $M_{G}(X)=\Spec A_G(X)$ as the Zariski closure of $(G\times G).X$ in $\g\oplus\g\oplus\End(V)$.
In calculating $M_{G}(X)$ we therefore have for each $\mu$ four matrices
$\Phi_\mu^\pm,A_\mu,B_\mu\in\Mat_{n_\mu,n_\mu}(\C)$ given by
\[
\Phi^\pm_\mu = g_\mu^\pm {\diag}(z_{\mu,1},\dots,z_{\mu,n_\mu}) (g_\mu^\pm)^{-1},\]
\[ A_\mu = g_\mu^- {\diag}(x^-_{\mu,1},\dots,x^-_{\mu,n_\mu}) (g_\mu^+)^{-1},
\]
\[ B_\mu = g_\mu^+ {\diag}(x^+_{\mu,1},\dots,x^+_{\mu,n_\mu}) (g_\mu^-)^{-1},\]
as well as $(n_{\mu}n_{\mu+1}\times n_{\mu}n_{\mu+1})$-matrices
$$C_{\mu\to\mu+1}=(g_\mu^-\otimes g_{\mu+1}^-)\diag\bigl(x^-_{\mu\to\mu+1,ij}\bigr)_{i,j=1}^{{i=n_\mu},{j=n_{\mu+1}}} (g_\mu^+\otimes g_{\mu+1}^+)^{-1},
$$
$$D_{\mu\to\mu+1}=(g_\mu^+\otimes g_{\mu+1}^+)\diag\bigl(x^+_{\mu\to\mu+1,ij}\bigr)_{i,j=1}^{{i=n_\mu},{j=n_{\mu+1}}} (g_\mu^-\otimes g_{\mu+1}^-)^{-1}.
$$
The matrices $\Phi_\mu^\pm$ clearly satisfy  $\tr(\Phi_\mu^+)^i={\tr}(\Phi_\mu^-)^i$ for all $i$. Similarly, the commuting relations $\Phi_\mu^-A_\mu=A_\mu\Phi_\mu^+$ etc.\ are clearly satisfied.
In order to write down  equations corresponding to \eqref{onemon}--\eqref{fourmon} let us introduce the following notation: if $M\in \Mat_{m,m}(\cx)$, $N\in \Mat_{n,n}(\cx)$, then $P_N(M)$ denotes the characteristic polynomial of $N$ evaluated at $M$, i.e.: 
\begin{equation} P_N(M)=\sum_{j=0}^{n} (-1)^{j}e_{j}(N)  M^{n-j},\label{P_N(M)}
\end{equation}
where $e_j(N)$ denotes the $j$-th elementary symmetric polynomial in eigenvalues of the matrix $N$. Let us also write $\Lambda_\mu$ for the diagonal matrix $\diag(\lambda_{\mu,1},\dots,\lambda_{\mu,k_\mu})$.
The equations corresponding to \eqref{onemon} can be now written as:

\begin{gather}  A_\mu B_\mu = (-1)^{n_{\mu-1}}P_{\Lambda_\mu}(\Phi_\mu^- )P_{\Phi_{\mu+1}^-}(\Phi_\mu^-) P_{\Phi_{\mu-1}^-}(\Phi_\mu^-),\label{ABPhi}\\
B_\mu A_\mu = (-1)^{n_{\mu-1}}P_{\Lambda_\mu}(\Phi_\mu^+ )P_{\Phi_{\mu+1}^+}(\Phi_\mu^+) P_{\Phi_{\mu-1}^+}(\Phi_\mu^+).\label{BAPhi}
\end{gather}
Equations corresponding to  \eqref{twomon} and \eqref{threemon} are:
\begin{gather} A_\mu\otimes A_{\mu+1}= C_{\mu\to\mu+1}\bigl(\Phi_{\mu}^+\otimes 1-1\otimes \Phi_{\mu+1}^+\bigr),\label{AAC} \\  B_\mu\otimes B_{\mu+1}= D_{\mu\to\mu+1}\bigl(\Phi_{\mu}^-\otimes 1-1\otimes \Phi_{\mu+1}^- \bigr). \label{BBD}
\end{gather}
In order to write down equations corresponding to \eqref{fourmon}, let $D$ denote the $(n_{\mu}n_{\mu+1}\times n_{\mu}n_{\mu+1})$ diagonal matrix, the $((ij),(ij))$-entry of which is the second line of \eqref{fourmon} multiplied by $(z_{\mu,i}-z_{\mu+1,j})^2$ (so that there are no omitted indices in the first two products).  
 Observe that 
$(g_\mu^\pm\otimes g_{\mu+1}^\pm)D (g_\mu^\mp\otimes g_{\mu+1}^\mp)^{-1}$ is equal to \begin{equation}(-1)^{n_{\mu-1}+n_\mu}P_{\Phi_{\mu+1}^\pm}(\Phi_\mu^\pm) P_{\Phi_{\mu-1}^\pm}(\Phi_{\mu}^\pm)\otimes P_{\Phi_{\mu}^\pm}(\Phi_{\mu+1}^\pm) P_{\Phi_{\mu+2}^\pm}(\Phi_{\mu+1}^\pm).\label{last}\end{equation}
 Similarly, 
$$(g_\mu^\pm\otimes g_{\mu+1}^\pm)\diag\bigl((z_{\mu,i}-z_{\mu+1,j})^2\bigr)_{i,j=1}^{{i=n_\mu},{j=n_{\mu+1}}} (g_\mu^\mp\otimes g_{\mu+1}^\mp)^{-1}=\bigl(\Phi_{\mu}^\pm\otimes 1-1\otimes \Phi_{\mu+1}^\pm\bigr)^2.$$
For each choice of sign in $\pm$-superscripts, the matrix in \eqref{last}  and $\bigl(\Phi_{\mu}^\pm\otimes 1-1\otimes \Phi_{\mu+1}^\pm\bigr)^2$ commute and the latter divides the former. We have therefore well defined $(n_{\mu}n_{\mu+1}\times n_{\mu}n_{\mu+1})$-matrices $ R(\Phi_{\mu}^\pm, \Phi_{\mu+1}^\pm)$ obtained by dividing \eqref{last} by $\bigl(\Phi_{\mu}^\pm\otimes 1-1\otimes \Phi_{\mu+1}^\pm\bigr)^2$. The equations corresponding to \eqref{fourmon} can be now written as:
\begin{gather} 
C_{\mu\to\mu+1}D_{\mu\to\mu+1}=\bigl(P_{\Lambda_\mu}(\Phi_\mu^- )\otimes 1\bigr) R(\Phi_{\mu}^-, \Phi_{\mu+1}^-) \bigl(1\otimes P_{\Lambda_{\mu+1}}(\Phi_{\mu+1}^- )\bigr),\notag\\
D_{\mu\to\mu+1}C_{\mu\to\mu+1}=  \bigl(P_{\Lambda_\mu}(\Phi_\mu^+ )\otimes 1\bigr) R(\Phi_{\mu}^+, \Phi_{\mu+1}^+)\bigl(1\otimes P_{\Lambda_{\mu+1}}(\Phi_{\mu+1}^+ )\bigr).\notag
\end{gather}
The variety $M_{G}(X)$ is an irreducible component of matrices satisfying these equations (the closure of the set where  eigenvalues of each $\Phi_{\mu}^+$ are distinct and distinct from those of $\Phi_{\mu+1}^+$).
As explained in  \S\ref{minuscule}, we can also describe the twistor space for $M_{G}(X)$. The
sections of this twistor space are matrix-valued polynomials $\Phi^\pm_\mu(\zeta)$, $A_\mu(\zeta)$, $B_\mu(\zeta)$,  $C_{\mu\to\mu+1}(\zeta)$, $D_{\mu\to\mu+1}(\zeta)$ in one variable, satisfying the above equations, and with
$$\begin{gathered}\deg\Phi^\pm_\mu(\zeta)=2, \quad \deg A_\mu(\zeta)=\deg B_\mu(\zeta)=k_\mu+n_{\mu-1}+n_{\mu+1}, \\ \deg C_{\mu\to\mu+1}(\zeta)=\deg D_{\mu\to\mu+1}(\zeta)=  k_\mu+k_{\mu+1}+\sum_{s=-1}^2 n_{\mu+s}-2.\end{gathered}
$$
Moreover, the real structure on $Z_{G}(X)$ can be described as in Remark \ref{rrealstr}: the linear map $\tau$ is given by $ \tau(A_\mu,B_\mu)=\bigl((-1)^{\deg B_\mu(\zeta)}B, A)$ and $$ \tau(C_{\mu\to\mu+1}, D_{\mu\to\mu+1})=\bigl((-1)^{\deg D_{\mu\to\mu+1}(\zeta)}D_{\mu\to\mu+1}, C_{\mu\to\mu+1}\bigr).
$$
\begin{remark} In the case when all $k_\mu=0$ (i.e.\ the monopoles are nonsingular), the cocharacters $\alpha_j$ defining the hypertoric variety do not span $\h$. This case has to be dealt with as in Remark \ref{notspanning}.
\end{remark}
We hope that (a component of) the variety of  real polynomials satisfying the above equations can be described as  (a finite cover of) a  closed hyperk\"ahler submanifold of a  hyperk\"ahler  quotient of a vector space. This would prove, among other things, the completeness of the hyperk\"ahler metric on $\Hilb_\mu^W(X)$. In what follows, we shall show that this is the case if:
\begin{itemize}\setlength{\itemindent}{-1,5em}
\item either $r=1$, or
\item $r>1$ and, for each $\mu=1,\dots,r-1$, either $\min(n_\mu,n_{\mu+1})=1$ or $n_\mu=n_{\mu+1}=2$.
\end{itemize}
\begin{example} Example \ref{findimq} identifies the hyperk\"ahler metric on the above $M_G(X)$ in the case when
 $r=1$, i.e.\ in the case of charge $d$ ${\rm SO}(3)$-monopoles  with $k$ Dirac singularities located at $\underline{\lambda}_1,\dots,\underline{\lambda}_k\in\R^3$:
there is a connected component $D$ of  $M_K^\circ(X)$ which is a hyperk\"ahler quotient of a quaternionic representation space of the $A_{k+1}$-quiver with dimension vectors
$\cx^d$ by ${\rm GL}_d(\cx)^{k-1}$ acting on the inner $k-1$ vertices.
 Consequently (Theorem \ref{N_X}), the hyperk\"ahler metric on $\Hilb^W_\mu(X)$ is the natural $L^2$-metric on a moduli space of solutions to Nahm's equations: if $\omega=(c,-c)$, $c>0$, is the mass of monopoles, the solutions are $\u(d)$-valued and defined on $(-c,c)$ with simple poles and residues given by a fixed principal $\su(2)$-triple at $t=\pm c$, and a discontinuity determined by $D$ at $t=0$. This is the description of the moduli space of singular ${\rm SO}(3)$-monopoles given by Blair and Cherkis \cite{BC}.\label{SO(3)-mon}
\end{example}
The general case $r>1$ can be dealt with ``in stages" as follows. Let $M_\mu$, $\mu=1,\dots,r$, be the hyperspherical variety corresponding to ${\rm SO}(3)$-monopoles of charge $n_\mu$ with Dirac singularities at $\underline{\lambda}_{\mu,l}$,
 $l=1,\dots,k_\mu$, and, for $n_\mu,n_{\mu+1}\geq 1$, let $M_{\mu\to\mu+1}$ be the hyperspherical variety corresponding to nonsingular ${\rm SU}(3)$-monopoles with charge $(n_\mu,n_{\mu+1})$.   The equations describing $M_G(X)$ for an arbitrary ${\rm PU}(r+1)$-monopole moduli space are simply the equations of the symplectic quotient of 
 $ \prod_{\mu=1}^rM_\mu\times  \prod_{\mu=1}^{r-1}M_{\mu\to\mu+1}$
 by $\prod_{\mu=1}^r{\rm GL}_{n_\mu}(\C)$  (cf.\ Remark \ref{star-product}). The hyperk\"ahler (as opposed to complex-symplectic) picture is slightly complicated by the fact that for $M_{\mu\to\mu+1}$ the cocharacters defining the hypertoric variety do not span $\cx^{n_\mu}\oplus \cx^{n_{\mu+1}}$. We need to deal with it as in Remark \ref{notspanning}, i.e.\ we consider the hyperspherical variety $M_{\mu\to\mu+1}^0$ corresponding to the hypertoric variety $X_{\mu\to\mu+1}^0$ with structure torus $T_0$ equal to the maximal torus of ${\rm S}({\rm GL}_{n_\mu}(\C)\times {\rm GL}_{n_{\mu+1}}(\C))$.
The hyperspherical variety $M_{\mu\to\mu+1}$ is then the quotient of $M_{\mu\to\mu+1}^0\times T^\ast \cx^\ast$ by $\oZ_{n_\mu +n_{\mu+1}}$, and the hyperk\"ahler metric on an arbitrary moduli space  of monopoles is obtained as in Remark \ref{notspanning}. One needs, therefore, to identify the hyperk\"ahler metric on $M_{\mu\to\mu+1}^0$. 
\begin{example} Consider the case when $n_{\mu+1}=1$. The morphism $T^{n_\mu}\to T_0$ is given by 
$$(t_1,\dots,t_{n_\mu})\longmapsto (t_1,\dots,t_{n_\mu},t_1^{-1}\cdots t_{n_\mu}^{-1} ),$$
and is therefore an isomorphism. It follows that  $X_{\mu\to\mu+1}^0$ is simply  $\cx^{2n_\mu}$ and, consequently, $M_{\mu\to\mu+1}^0$ is $T^\ast\Mat_{n_\mu\times n_\mu}(\C)$ with its invariant flat hyperk\"ahler metric. Analogously, 
$M_{\mu\to\mu+1}^0\simeq T^\ast\Mat_{n_{\mu+1}\times n_{\mu+1}}(\C)$ if $n_\mu=1$.
This is not surprising: Hurtubise and Murray's \cite{HM} description of the moduli space of (nonsingular) ${\rm SU}(3)$-monopoles of charge $(n,1)$ (or $(1,n)$) is very similar to the Cherkis and Kapustin's \cite{CK} description of the moduli space of charge $n$ ${\rm SO(3)}$-monopoles with one Dirac singularity.\label{n,1}
\end{example}

\subsection{${\rm SU}(3)$-monopoles of charge $(2,2)$}
We are going to identify the hyperk\"ahler metric on $M_{\mu\to\mu+1}^0$ in the case $n_\mu=n_{\mu+1}=2$. According to the above discussion, this, together with Example \ref{n,1},  provides a description of moduli spaces of ${\rm PU}(r+1)$-monopoles in terms of solutions of Nahm's equations in the case when all $n_\mu\leq 2$ (and the $k_\mu$ are arbitrary).
\par
The structure torus $T_0$ of the hypertoric variety $X_{\mu\to\mu+1}^0$ is the maximal torus of ${\rm S}({\rm GL}_{2}(\C)\times {\rm GL}_{2}(\C))$. The character lattice of $T_0$ is therefore
$$ \oZ\langle e^\ast_{\mu,1},e^\ast_{\mu,2},e^\ast_{\mu+1,1},e^\ast_{\mu+1,2}\rangle \big/ \oZ(e^\ast_{\mu,1}+e^\ast_{\mu,2}+e^\ast_{\mu+1,1}+e^\ast_{\mu+1,2}).
$$
We can take the same generating set $S$ as for $X_{\mu\to\mu+1}$, and obtain generators $z_{\mu,i}\,,z_{\mu+1,j}\,, x_{\mu,i}^\pm,\, x_{\mu+1,j}^\pm,\, x_{\mu\to \mu+1,ij}^\pm$, $i,j=1,2$, of $\cx[X_{\mu\to\mu+1}^0]$. These satisfy equations \eqref{onemon}--\eqref{fourmon} (with $k_\mu=k_{\mu+1}=0$), and in addition (with $\sigma(1)=2,\,\sigma(2)=1$):
\begin{gather} z_{\mu,1}+ z_{\mu,2}+ z_{\mu+1,1}+ z_{\mu+1,2}=0,\label{m-1}\\
 x_{\mu\to \mu+1,ij}^\pm=x_{\mu\to \mu+1,\sigma(i)\sigma(j)}^\mp,\quad x_{\mu,1}^\pm x_{\mu,2}^\pm= x_{\mu+1,1}^\mp x_{\mu+1,2}^\mp,
 \label{middle}\\ x_{\mu,1}^\pm x_{\mu,2}^\pm x_{\mu+1,1}^\pm x_{\mu+1,2}^\pm=\prod_{i,j=1}^2(z_{\mu,i}-z_{\mu+1,j})=
 \prod_{i,j}x_{\mu\to \mu+1,ij}^\pm.\label{m+1}
\end{gather}
The coordinate ring of the hyperspherical variety $M_{\mu\to\mu+1}^0$ is generated, as for $M_{\mu\to\mu+1}$, by $2\times 2$ matrices $\Phi^\pm_\mu,\Phi^\pm_{\mu+1}, A_\mu,B_\mu, A_{\mu+1},B_{\mu+1}$, and by $4\times 4$ matrices $C_{\mu\to\mu+1},D_{\mu\to\mu+1}$.
These satisfy the equations for $M_{\mu\to\mu+1}$, and in addition, equations corresponding to \eqref{m-1}--\eqref{m+1}. The first of those yields
$ \tr\Phi_\mu^\pm+\tr\Phi_{\mu+1}^\pm=0.$ The equations corresponding to \eqref{m+1} are:
\begin{gather} 
\det A_{\mu}\det A_{\mu+1}=\det\bigl(\Phi_{\mu}^-\otimes 1-1\otimes \Phi_{\mu+1}^-\bigr)=\det C_{\mu\to\mu+1},\label{det(A)}\\
\det B_{\mu}\det B_{\mu+1}=\det\bigl(\Phi_{\mu}^+\otimes 1-1\otimes \Phi_{\mu+1}^+\bigr)=\det D_{\mu\to\mu+1}.\label{det(B)}\end{gather}
The second equation in \eqref{middle} yields:
$$\det A_\mu=\det B_{\mu+1},\enskip \det A_{\mu+1}=\det B_{\mu}.
$$
In order to understand the equation corresponding to the first equation in \eqref{middle}, identify $\Mat_{4\times 4}(\cx)$ with $\Mat_{2\times 2}(\cx)\otimes \Mat_{2\times 2}(\cx)$ and consider the involution $\sigma$ given by $\sigma(A\otimes B)=A_{\rm adj}\otimes B_{\rm adj}$. This involution is equal to $+1$ on $\cx\langle 1\otimes 1\rangle \oplus \ssl_2(\cx)\otimes \ssl_2(\C)$,
and to $-1$ on $\ssl_2(\cx)\otimes 1\oplus 1\otimes \ssl_2(\cx)$. The latter subspace is a Lie subalgebra of $\gl_4(\C)$, isomorphic to $\so_4(\cx)$, and we conclude that,  in an appropriate basis, $\sigma$ is equal to the matrix transpose. Now observe that the equation corresponding to the first equation in \eqref{middle} is $\sigma(C_{\mu\to\mu+1})=D_{\mu\to\mu+1}$.
\par
We now compute $ R(\Phi_{\mu}^\pm, \Phi_{\mu+1}^\pm)$ from its definition and \eqref{fourmon}:
$$ R(\Phi_{\mu}^\pm, \Phi_{\mu+1}^\pm)=\Bigl(\Phi_\mu^\pm\otimes 1-1\otimes\bigl(\Phi_{\mu+1}^\pm\bigr)_{\rm adj}\Bigr) \Bigl(\bigl(\Phi_{\mu}^\pm\bigr)_{\rm adj}\otimes 1-1\otimes\Phi_{\mu+1}^\pm\Bigr).$$
Since, for a $2\times 2$ matrix $X$, $X_{\rm adj}=-X+(\tr X)\cdot 1$, and since $\tr \Phi_{\mu}^\pm +\tr \Phi_{\mu+1}^\pm =0$, we can rewrite this as
\begin{equation}  R(\Phi_{\mu}^\pm, \Phi_{\mu+1}^\pm)=-\Bigl(\Phi_{\mu}^\pm\otimes 1 +
1\otimes\Phi_{\mu+1}^\pm\Bigr)^2+\bigl(\tr  \Phi_{\mu}^\pm\bigr)^2 1\otimes 1 .\label{R-adj}\end{equation}

\medskip 

We now consider, analogously to Example \ref{findimq}, the representation space $V$ of the double $A_2$-quiver with vertices $\{1,2,3\}$ and all dimension vectors  equal to $4$. A representation is given by the collection of $4\times 4$ matrices $A_{ij}$, $i,j=1,2,3$, $|i-j|=1$, where $A_{ij}$ corresponds to the arrow $i\to j$.
Let $M$ be the hyperk\"ahler quotient of $V$ by ${\rm SU}(4)$ acting in the standard way at the middle vertex.
We proceed to describe $M$ as a complex manifold for a fixed complex structure, i.e.\  the complex-symplectic quotient  of  $V$ by ${\rm SL}_4(\C)$. The ${\rm SL}_4(\C)$-invariant functions on $V$ are generated by the entries of $\Psi_{\rm L}=A_{21}A_{12}$, $\Psi_{\rm R}=A_{23}A_{32}$, $C=A_{21}A_{32}$, $D=A_{23}A_{12}$ , as well as by the four functions $a_{ij}=\det A_{ij}$.
Write $\kappa=\tr(A_{12}A_{21}-A_{32}A_{23})/4=\tr(\Psi_L-\Psi_R)/4$. The moment map equation for the action of ${\rm SL}_4(\C)$ at the middle vertex is $A_{12}A_{21}-A_{32}A_{23}=\kappa\cdot 1$, which yields
\begin{equation} CD=\Psi_{\rm L}^2-\kappa \Psi_{\rm L},\enskip DC=\Psi_{\rm R}^2+\kappa \Psi_{\rm R}.\label{CD,DC}\end{equation}
The remaining equations satisfied by the ${\rm SL}_4(\C)$-invariant functions are $\det \Psi_{\rm L}=a_{21}a_{12}$, $\det \Psi_{\rm R}=a_{23}a_{32}$, $\det C=a_{21}a_{32}$, and $\det D=a_{23}a_{12}$. 
These equations describe $M$ as a complex manifold. The twistor space of $M$ is described by the same equations in the total space of a vector bundle over $\oP^1$: $\Psi_L,\Psi_R, C,D\in \Mat_{4,4}(\C)\otimes \sO_{\oP^1}(2)$, $a_{ij}\in 
\sO_{\oP^1}(4)$. 
 Observe that if $\det C\neq 0$ and $\det D\neq 0$, then  \eqref{CD,DC} implies
\begin{equation} \det (\Psi_{\rm L}-\kappa)=a_{23}a_{32},\quad \det (\Psi_{\rm R}+\kappa)=a_{12}a_{21}.\label{additional}\end{equation}
Moreover, if $\det C$ and $\det D$ are nonzero, then so are $\det \Psi_L$ and $\det\Psi_R$, which means that the action of  ${\rm SL}_4(\C)$ is free on the subset of $V$ where $\det C\neq 0\neq \det D$. Since the complex-symplectic quotient $M^\prime$ of the subset of $V$, where the action of  ${\rm SL}_4(\C)$ is free, is regular, equations \eqref{additional} hold on all of $M^\prime$. This in turn implies, that \eqref{additional} 
is satisfied on the stratum $M_{(1)}$ of the hyperk\"ahler quotient $M$ corresponding to the free action of   ${\rm SU}(4)$. Since this stratum is open and dense, equations \eqref{additional} hold on the hyperk\"ahler quotient $M$ for any complex structure.
\par
We now observe that $M$ has an involution $\tau$ given by
\begin{equation*} (C,D,\Psi_{\rm L}, \Psi_{\rm R},a_{12},a_{21},a_{23},a_{32})\longmapsto (-D^T,-C^T, -\Psi_{\rm L}^T+\kappa, -\Psi_{\rm R}^T-\kappa,a_{32},a_{23},a_{21},a_{12}). 
\end{equation*}
This involution acts holomorphically and fibrewise on the twistor space of $M$, interwining the real structures, and therefore it preserves the hypercomplex structure, hence also the Levi-Civita connection of the hyperk\"ahler metric. Computing the complex-symplectic form of $M$ (this is straightforward on the open dense subset where $ C,D,\Psi_{\rm L}, \Psi_{\rm R}$ are diagonalisable under the action of ${\rm GL}_4(\C)\times {\rm GL}_4(\C)$) shows that $\tau$ is an isometry.
We can now identify the hyperk\"ahler structure of our hyperspherical variety $M_{\mu\to\mu+1}^0$:
\begin{proposition} For $n_\mu=n_{\mu+1}=2$, the smooth locus of the hyperspherical variety $M_{\mu\to\mu+1}^0$ is isomorphic, as a hyperk\"ahler manifold, to a double cover of the smooth locus of the fixed point set of $\tau$ on $M$.
\end{proposition}
\begin{proof} 
Setting $\Phi_{\rm L}=\Psi_{\rm L}-\kappa/2$ and $\Phi_{\rm R}=\Psi_{\rm R}+\kappa/2$, equations \eqref{CD,DC} become
$$ CD=\Phi_{\rm L}^2-\kappa^2/4,\quad DC=\Phi_{\rm R}^2-\kappa^2/4,$$
while the involution $\tau$ sends $\Phi_{\rm L}, \Phi_{\rm R}$ to $-\Phi_{\rm L}^T, -\Phi_{\rm R}^T$.
 Hence $M^\tau$ is described by $D=-C^T$, $a_{12}=a_{32}$, $a_{21}=a_{32}$, and $\Phi_{\rm L}, \Phi_{\rm R}$ being skew-symmetric. For a fixed complex structure, we define a holomorphic map $M_{\mu\to\mu+1}^0\to M^\tau$ as follows: to a point $$(A_\mu,A_{\mu+1},B_\mu,B_{\mu+1}, C_{\mu\to\mu+1}, \Phi_\mu^\pm,\Phi_{\mu+1}^\pm)$$ of $M_{\mu\to\mu+1}^0$ we associate the point
 $$\begin{gathered} C=C_{\mu\to\mu+1},  \enskip  a_{12}=\det B_\mu,  \enskip a_{21}=\det A_\mu,  \enskip  \kappa=\tr  \Phi_{\mu}^+=-\tr  \Phi_{\mu}^-,\\ \Phi_{\rm L}=\Phi_{\mu}^-\otimes 1 +
1\otimes\Phi_{\mu+1}^-, \enskip \Phi_{\rm R}=\Phi_{\mu}^+\otimes 1 +
1\otimes\Phi_{\mu+1}^+\end{gathered}
 $$
 of $M^\tau$. This induces a holomorphic map between the twistor spaces, which sends real sections to real sections. Thus we have a smooth map $p:M_{\mu\to\mu+1}^0\to M^\tau$ between the stratified hyperk\"ahler manifolds preserving the hypercomplex structure on the smooth locus.
 \par
 Conversely, given a  section of the twistor space of $M^\tau$, we obtain quadratic polynomials $C_{\mu\to\mu+1}(\zeta)$, $\Phi^\pm_\mu(\zeta)$, and $\Phi^\pm_{\mu+1}(\zeta)$, which satisfy the reality conditions and the equations
 \begin{equation} C_{\mu\to\mu+1}C_{\mu\to\mu+1}^T=R(\Phi_{\mu}^-, \Phi_{\mu+1}^-), \quad C_{\mu\to\mu+1}^TC_{\mu\to\mu+1}=R(\Phi_{\mu}^+, \Phi_{\mu+1}^+).\label{CC^T}\end{equation}
 We need to show that there exist quadratic polynomials $A_\mu(\zeta),A_{\mu+1}(\zeta),B_\mu(\zeta)$, $B_{\mu+1}(\zeta)$ which satisfy the reality conditions and equations \eqref{ABPhi}--\eqref{BBD} (with $D_{\mu\to\mu+1}=C_{\mu\to\mu+1}^T$). Set 
 $F=C_{\mu\to\mu+1}\bigl(\Phi_{\mu}^+\otimes 1-1\otimes \Phi_{\mu+1}^+\bigr).$
 Since the matrix transposition on $\Mat_{4,4}(\cx)\simeq \Mat_{2,2}(\cx)\otimes \Mat_{2,2}(\cx) $ is, as explained above, equal to $\sigma(A\otimes B)=A_{\rm adj}\otimes B_{\rm adj}$, we obtain from \eqref{CC^T} and the definition of $R(\Phi_{\mu}^+, \Phi_{\mu+1}^+)$:
 $$ F^TF=\det\bigl(\Phi_{\mu}^+\otimes 1-1\otimes \Phi_{\mu+1}^+\bigr)\cdot 1.
 $$
 Suppose that the determinant $d$ on the right side is nonzero. Then $F/d^{1/2}$ is an orthogonal matrix. Using again the fact that the matrix transposition on $\Mat_{4,4}(\cx)\simeq \Mat_{2,2}(\cx)\otimes \Mat_{2,2}(\cx) $ is equal to $\sigma(A\otimes B)=A_{\rm adj}\otimes B_{\rm adj}$, we conclude that ${\rm SO}_4(\C)$ is isomorphic to the group of rank $1$ tensors $A\otimes B$, with $\det A\det B=1$.
 Suppose that our given section of the twistor space of $M$ satisfies $d(\zeta)=\det\bigl(\Phi_{\mu}^+(\zeta)\otimes 1-1\otimes \Phi_{\mu+1}^+(\zeta)\bigr)\not\equiv 0$. Then $F(\zeta)$ is a rank $1$ tensor for generic $\zeta$, hence a tensor of rank $\leq 1$ everywhere. We can therefore write $F(\zeta)=A_\mu(\zeta)\otimes A_{\mu+1}(\zeta)$.  Thus we have found (nonzero) solutions of \eqref{AAC} which are  unique up to a scalar multiple. A completely analogous argument shows the existence of  $B_\mu,B_{\mu+1}$ satisfying \eqref{BBD}.
 Observe that equations \eqref{det(A)}--\eqref{det(B)} are also satisfied, and that the matrices  $A_\mu(\zeta),A_{\mu+1}(\zeta),B_\mu(\zeta)$, $B_{\mu+1}(\zeta)$ are invertible for generic $\zeta$.
 If we multiply \eqref{AAC} and \eqref{BBD}, using the commutativity relation  $C_{\mu\to\mu+1}\bigl(\Phi_{\mu}^+\otimes 1-1\otimes \Phi_{\mu+1}^+\bigr)= \bigl(\Phi_{\mu}^-\otimes 1-1\otimes \Phi_{\mu+1}^-\bigr)C_{\mu\to\mu+1}$ and the relation $C_{\mu\to\mu+1}D_{\mu\to\mu+1}=R(\Phi_{\mu}^-,\Phi_{\mu+1}^-)$, we obtain
 $$A_\mu B_\mu\otimes A_{\mu+1}B_{\mu+1}=P_{\Phi_{\mu+1}^-}(\Phi_\mu^-)\otimes P_{\Phi_{\mu}^-}(\Phi_{\mu+1}^-),$$
 and, analogously:
 $$B_\mu A_\mu\otimes B_{\mu+1}A_{\mu+1}=P_{\Phi_{\mu+1}^+}(\Phi_\mu^+)\otimes P_{\Phi_{\mu}^+}(\Phi_{\mu+1}^+).$$
 Since the left-hand sides are not identically zero, there is a unique, up to a sign,  scalar $c$ such that $cA_\mu(\zeta),cB_{\mu}(\zeta)$,  $ c^{-1}A_{\mu+1}(\zeta),c^{-1}B_{\mu+1}(\zeta)$ satisfy equations \eqref{ABPhi}--\eqref{BAPhi} and the matrices $A_\mu(\zeta),A_{\mu+1}(\zeta),B_\mu(\zeta)$,$B_{\mu+1}(\zeta)$ satisfy the reality conditions. We conclude that the map $p:M_{\mu\to\mu+1}^0\to M^\tau$ is a double cover on the  open  subset of $M_{\mu\to\mu+1}^0$ characterised by $d(\zeta)\not\equiv 0$.  Consequently, $p$ is a double cover on  the smooth locus of $M_{\mu\to\mu+1}^0$. We already know that $p$ commutes with the hypercomplex structure. One can now verify, as in Example \ref{findimq}, that $p$ preserves the complex-symplectic forms match on the open dense subset where all matrices are diagonalisable under the action of $\bigl({\rm S}({\rm GL}_{2}(\C)\times {\rm GL}_{2}(\C))\bigr)^2$. Therefore $p$ is a local isometry on the smooth locus. 
 \end{proof}
\begin{remark} The Nahm equations description of  the moduli space of (smooth) ${\rm SU}(3)$-monopoles of charge $(2,2)$, obtained from $M_{\mu\to\mu+1}^0$ as  in \S\ref{10} (see also Remark \ref{notspanning}),  is very different from the one of Hurtubise and Murray \cite{HM}.
\end{remark}

\section{Remarks on the asymptotic behaviour of the metric\label{asymptotics}}
We would like to briefly explain the expected asymptotic behaviour of the hyperk\"ahler metric on $\Hilb^W_\mu(X)$ (cf.\ \cite[\S4.1]{BDG}). In order to do this  we need  to enlarge the class of toric hyperk\"ahler manifolds to include {\em folded toric hyperk\"ahler manifolds}. 
A folded hyperk\"ahler manifold  \cite{Hit-fold, Biq-fold} is a smooth manifold $X$ equipped with a $2$-sphere of closed $2$-forms $\omega_\zeta$, $\zeta\in \oP^1$, and containing a union $D$ (the {\em folding locus}) of codimension $1$ smooth hypersurfaces such that
\begin{itemize}\item[(i)] On each connected component of $X\setminus D$ the $2$-forms define a pseudo-hyperk\"ahler structure;
\item[(ii)] At a smooth point $p$ of $D$ either one or all $2$-forms $\omega_{\zeta}$ degenerate transversally, i.e.\ the highest power of such $\omega_\zeta$ vanishes transversally on a neighbourhood of $p$ in $D$.
\end{itemize}
This definition can be extended to stratified hyperk\"ahler manifolds (provided compatibility  between $D$ and the stratification is observed). Such folded hyperk\"ahler manifolds have well defined twistor space, with all the properties of the twistor space of a hyperk\"ahler manifold. The real sections corresponding to the points of the folding locus $D$ will, however,  have wrong normal bundle. 
\par
We shall now discuss folded toric hyperk\"ahler manifolds.
First of all, analogously to the monoid of affine hypertoric varieties (Remark \ref{monoid}), we can define a semigroup of toric  ALF\footnote{i.e.\ with volume growth equal to $r^{3\dim T_\R}$.} hyperk\"ahler manifolds with structure torus $T_\R$. The product structure is given by the hyperk\"ahler quotient of the Cartesian product of two such manifolds by the anti-diagonal torus. Unlike in the case of hypertoric varieties, there is no neutral element. This semigroup is clearly isomorphic to the semigroup $\oS^+$ consisting of triples $(L,\sA,m)$, where $L$ is a self-adjoint positive-definite linear automorphism of $\h_\oR$, $\sA$ is a collection of distinct codimension $3$ flats of the form \eqref{Hflats}, and $m:\sA\to \oN$ is a multiplicity function.
\par
We now observe that there is a natural extension of this correspondence to a larger semigroup $\oS$, consisting of  $(L,\sA,m)$ as above, but with $m:\sA\to \oZ$. One way to see this is via the generalised Legendre transform construction \cite[\S 3.G]{HKLR}. This construction determines the toric hyperk\"ahler metric  through a polyharmonic function $F:\h_\R\otimes \R^3\to \R$, and it follows from the calculation in  \cite[\S8]{BD} that  $F$  is additive as a function on $\oS^+$ (up to irrelevant linear summands). We can now allow negative signs in the summands of $F$ - the resulting function is still polyharmonic and performing the generalised Legendre transform results in a {\em folded toric hyperk\"ahler manifold} $X$. 
\par
The twistor space of such a folded toric hyperk\"ahler manifold corresponding to $(L,\sA,m)$ differs from the twistor space of the toric hyperk\"ahler manifold given by $(L,\sA,|m|)$ only in the real structure. If we write, for each codimension $3$ flat $\underline{H_i}\in\sA$, $m_i=m(\underline{H_i})$ and, if $m_i\neq 0$, $\omega_i=\alpha_i/|m_i|$ (so that $\omega_i$ are primitive in $\sX_\ast(T)$), then the real structure acting on the fibre coordinate $x_\xi$ of   $\sL^{L(\xi)}\otimes\sO_Y\bigl(\sum_{i=1}^d|m_i\langle \omega_i,\xi\rangle|\bigr)$
 is given by (cf.\ \eqref{realstr})
\begin{equation}  x_\xi\mapsto (-1)^{\sum_{i=1}^dm_i\langle\omega_i,\xi\rangle_+}\,\ol{x_{-\xi}}\,\bar\zeta^{-\sum_{i=1}^dm_i|\langle \omega_i,\xi\rangle|}.\label{realstr2}\end{equation}
The semigroup structure of $\oS$ is isomorphic to the semigroup structure obtained by taking fibrewise symplectic quotients  of fibred (over $\oP^1$) products of these twistor spaces by the anti-diagonal torus.
\par
We can  apply the functor $\Hilb^W_\mu$ to the fibres of the twistor space of a folded toric hyperk\"ahler manifold. Away from the folding locus the results of \S\ref{hk-Hilb^W}  will remain valid.
\par
Let now $X$ be a smooth toric hyperk\"ahler manifold corresponding to $(L,\sA,m)\in\oS^+$. Then $\Hilb^W_\mu(X)$ is also smooth and we conjecture that its putative complete hyperk\"ahler metric exhibits the following cluster-like asymptotic behaviour. 
The asymptotic region of $\Hilb^W_\mu(X)$ decomposes into regions labelled by standard parabolic subgroups of $W$. Essentially, we view a point of $\Hilb^W_\mu(X)$  as a $W$-orbit in $X$, and the region labelled by a standard parabolic subgroup $P$ consists of orbits which are ``almost" equal to singular orbits supported at the union of fixed point sets $\h^{P^\prime}$, where $P^\prime$ runs over conjugates of $P$. 
\par
For each standard parabolic subgroup $P$, we define a folded toric hyperk\"ahler manifold $X_P$ by adding to the arrangement $\sA$ all reflection hyperplanes $H$ such that $\h^P\not\subset H$  with multiplicity $-2$.  The normaliser $N(P)$ of $P$ acts on $X_P$. We conjecture (at least for strongly $W$-invariant $X$) that the hyperk\"ahler metric in the asymptotic region of $\Hilb^W_\mu(X)$ corresponding to $P$ is\footnote{More precisely, there is a local diffeomorphism between the asymptotic regions of $\Hilb^{P}_\mu\bigl(X_P\bigr)/\Gamma$ and of $\Hilb^W_\mu(X)$ which is exponentially close to being a local isometry.}  exponentially close to the hyperk\"ahler metric on the corresponding region (in particular away from the folding locus) of  $\Hilb^{P}_\mu\bigl(X_P\bigr)/\Gamma$, where $\Gamma=N(P)/P$. For example, the ``largest" asymptotic region is the one corresponding to $P=\{1\}$, where we expect the metric to be exponentially close to the hypertoric analogue of the Gibbons-Manton metric \cites{GM,Bi-GM}, namely $X_{\{1\}}/W$,  and $X_{\{1\}}$ is a folded toric hyperk\"ahler manifold obtained from $X$ by adding {\em all} reflection hyperplanes with multiplicity $-2$. We remark that the results of  \cites{Bi-GM,Bi-clusters} provide evidence for this asymptotic picture in the case of $X=T\times \h$, where $T$ is the maximal torus of ${\rm U}(n)$ or of ${\rm SU}(n)$.

\appendix

\section{Modified Nahm's equations\label{appendix:B}}

Let $K$ be a compact Lie group with Lie algebra $\fK$ and a fixed  $\Ad_K$-invariant scalar product $\langle\,,\,\rangle$ on $\fK$. Let $L:\fK\to \fK$ be a $K$-equivariant positive-definite self-adjoint linear map.
The {\em $L$-Nahm equations} for a quadruple $(T_0(t),T_1(t),T_2(t),T_3(t))$ of $\fK$-valued smooth valued functions on an interval $I$ are the  following three ODEs: 
\begin{equation} \dot T_i=[T_i,T_0]+L\bigl([T_j,T_k]\bigr),\enskip \text{where $(ijk)$ is a cyclic permutation of $(123)$}.\label{mNe2}
\end{equation}
They are invariant under the usual $K$-valued gauge transformations $g:I\to K$ acting via 
 \begin{equation} T_0\mapsto \Ad_g(T_0)-\dot gg^{-1}, \quad T_i\mapsto  \Ad_g(T_i),\enskip i=1,2,3.\end{equation}
\begin{remark} In the case when $K$ is simple, the only possibility for $L$ is $c\Id$, $c>0$, and the $L$-Nahm equations are equivalent to the usual Nahm's equations via $T_0(t)\mapsto cT_0(ct)$, $T_i(t)\mapsto T_i(ct)$, $i=1,2,3$. In the general  case there will be a decomposition $\fK\simeq \bigoplus_{i=1}^r \fK_i$ such that  $L=\bigoplus_{i=1}^r c_i\Id_{\fK_i}$, $c_1,\dots,c_r\in (0,\infty)$. Equations \eqref{mNe2} are then simply rescaled Nahm's equations on each direct summand $\fK_i$. The group of gauge transformations, however, is not a product, unless $K$ itself is a product.\label{rescale}
\end{remark}
We consider the following $K$-invariant scalar product on $\fK^{\oplus 4}$ (cf.\ \eqref{flathk}):
\begin{equation} \langle L^{-1}(t_0),t_0\rangle+ \sum_{i=1}^3\langle L (t_i),t_i\rangle.\label{Lflat}\end{equation}
This scalar product is quaternion-Hermitian  with respect to the following action of the standard basis  $J_1,J_2,J_3$ of imaginary quaternions:
$$ J_1(t_0)=-L(t_1),\enskip J_1(t_1)=L^{-1}(t_0),\enskip J_1(t_2)=-t_3, \enskip J_1(t_3)=t_2,
$$
and cyclic permutations of indices $1,2,3$.
\par
The moduli space of solutions to $L$-Nahm's equations, with appropriate boundary conditions, will be a hyperk\"ahler manifold. The hyperk\"ahler structure arises, as for the usual Nahm equations, via an infinite-dimensional hyperk\"ahler quotient: we start with the vector space $\sC$ of $\oR^3$-invariant connections $T_0(t)dt+T_1(t)dx_1+  T_2(t)dx_2+ T_3(t)dx_3 $ on $I\times \oR^3$ with given boundary conditions, and make it into a flat hyperk\"ahler Hilbert space by integrating the scalar product \eqref{Lflat} over $I$. The gauge group of $K$-valued gauge transformations which are equal to $1$ at both ends of $I$ acts on $\sC$ preserving the hyperk\"ahler structure, and the moment map equations are \eqref{mNe2}. 
\par
We are interested mainly in the following moduli spaces: $N_K^-(L)$ (resp.\  $N_K^+(L)$) is the moduli space of $\fK$-valued $L$-Nahm's equations on $(-1,0]$ (resp.\ on $[0,1)$), regular at $t=0$ and with simple poles at $t=-1$ (resp.\ $t=1$), with residues $L^{-1}(\sigma_1), L^{-1}(\sigma_2), L^{-1}(\sigma_2)$, where $\sigma_1,\sigma_2,\sigma_3$ is a  principal $\su(2)$-triple. The gauge group consists of gauge transformations  which are equal to $1$ at both ends of the interval. The arguments in \cite[Appendix A]{slices} can be repeated verbatim to show that $N_K^{\pm}(L)$
are complete hyperk\"ahler manifolds. Both  $N_K^-(L)$ and $N_K^+(L)$ admit a hyperk\"ahler action of $K$ given by gauge tranformations with arbitrary values at $t=0$. Moreover, as in \cite{Bi1}, they are biholomorphic, with respect to  any complex structure, to  $G\times\mathscr{S}_{\g}$, where $G=K^\cx$, $\g=\Lie(G)$, and $\mathscr{S}_{\g}$ is a Slodowy slice to the regular nilpotent orbit in $\g$. 
\begin{remark} We can find a decomposition $\fK\simeq \bigoplus_{i=1}^r \fK_i$ such that  $L=\bigoplus_{i=1}^r c_i\Id_{\fK_i}$, $c_1,\dots,c_r\in (0,\infty)$. The group $K$ is of the form $\bigl(\prod_{i=1}^r K_i\bigr)/\Gamma$, where $\Lie(K_i)=\fK_i$ and $\Gamma$ is a finite abelian group. It follows from Remark \ref{rescale} that $N_K^-(L)$ is isomorphic, as a hperk\"ahler manifold, to $\bigl(\prod_{i=1}^rN_{K_i}^-(c_i\Id_{\fK_i})\bigr)/\Gamma$, and similarly for $N_K^+(L)$. Each $N_{K_i}^-(c_i\Id_{\fK_i})$ is isomorphic to a moduli space of solutions to the usual Nahm equations on $(-c_i,0]$.
\end{remark}
Finally, taking the hyperk\"ahler quotient of the product $N_K^-(L)\times N_K^+(L)$ by the diagonal $K$, we obtain again a complete hyperk\"ahler manifold which is biholomorphic to the universal centraliser \eqref{ucentr} for any complex structure.

\section{Lifting ${\rm SO}$ to ${\rm Spin}^c$\label{appendix:ht-lifts}}
Let $W$ be a Weyl group and 
\begin{equation}1\longrightarrow \C^\ast \longrightarrow \tilde T\longrightarrow T\longrightarrow 1,\label{C*-ext}\end{equation}
be a $W$-equivariant extension of complex tori
with $W$ acting by reflections on $T$ and $\tilde T$ and trivially on $\cx^\ast$. If $\tilde X$ is a $W$-invariant hypertoric variety with structure torus $\tilde T$, then $\Hilb^W_\mu(\tilde X)$ has a Hamiltonian  action of $\cx^\ast$ and we can form a symplectic quotient. In general (cf.\ Theorem \ref{C=W} or Example \ref{D1-D0}), this quotient is different from $\Hilb^W_\mu(X)$, where $X$ is the symplectic quotient of $\tilde X$ by $\cx^\ast$.
The following result shows that, for strongly $W$-invariant hypertoric varieties, the symplectic quotient of $\Hilb^W_\mu(\tilde X)$ by $\C^\ast$ depends only on the symplectic quotient of $\tilde X$ by $\C^\ast$, and not on $\tilde X$ itself.
\begin{proposition}
Let $\tilde X$ and $\tilde X^\prime$ be two strongly $W$-invariant hypertoric variety with structure torus $\tilde T$ such that the symplectic quotients of $\tilde X,\tilde X^\prime$ by $\cx^\ast$ are isomorphic. Then the symplectic quotients of $\Hilb^W_\mu(\tilde X)$ and $\Hilb^W_\mu(\tilde X^\prime)$ by $\C^\ast$ are also isomorphic.\label{XtildeX}
\end{proposition}
\begin{proof} The coordinate ring of $\tilde X$ is given by equations (cf.\ \eqref{monomials}))
 $$ x_\xi\cdot x_{\xi^\prime}=x_{\xi+\xi^\prime}\prod_{i=1}^d (\langle \alpha_k,z\rangle-\lambda_k)^{\langle \alpha_i,\xi\rangle_+ +\langle \alpha_i,\xi^\prime\rangle_+ - \langle \alpha_i,\xi+\xi^\prime\rangle_+}, $$
 where $\langle \alpha_k,z\rangle-\lambda_k=0$ are equations of the hyperplanes in the (signed partitioned) integral multiarrangement $(\tilde \sA,m)$ of $\tilde X$. The coordinate ring of $\tilde X^\prime $ is given completely analogously, with $(\tilde \sA,m)$ replaced by   $(\tilde\sA^\prime,m^\prime)$ of $\tilde X^\prime$.
 $\Hilb^W_\mu(\tilde X)$ and $\Hilb^W_\mu(\tilde X^\prime)$ are obtained by $W$-equivariant 
substitution $x_\xi=\phi_\xi(z)$, where $\phi_\xi$ are elements of the deformed ring $\tilde R_c$ of coinvariants for $\tilde \h^\ast$ (cf.\ Remark \ref{R_c}). This means that the $\cx^\ast$-invariant polynomials in $\Hilb^W_\mu(\tilde X)$ (or $\Hilb^W_\mu(\tilde X^\prime)$) do not depend on $\alpha_k,\lambda_k$.  
The fact that the symplectic quotients of $\tilde X$ and $\tilde X^\prime$ are  isomorphic tmeans that the projections of the cocharacters $\alpha_k,\alpha_k^\prime\in \sX^\ast(\tilde T)$ onto $\sX_\ast(T)$ are equal. Moreover, the
 intersections of $(\tilde \sA,m)$ and  $(\tilde\sA^\prime,m^\prime)$  with the hyperplane $\nu^{-1}(0)$ (where $\nu:\tilde \h^\ast\to \cx$ is the projection induced by \eqref{C*-ext}) are equal. These two facts imply that the factors $\prod_{i=1}^d (\langle \alpha_k,z\rangle-\lambda_k)^{\langle \alpha_i,\xi\rangle_+ +\langle \alpha_i,\xi^\prime\rangle_+ - \langle \alpha_i,\xi+\xi^\prime\rangle_+}$ also induce the same elements of the deformed ring $ R_c$ of coinvariants for $ \h^\ast$. This proves the claim. 
\end{proof}
We now consider the following special case of the above situation.
Let $G$ be a reductive algebraic group of the form $G=G_0\times G^\prime$, where $G_0=\prod_{i=1}^N {\rm SO}_{2k_i+1}(\cx)$, $k_i\geq 1$, and $G^\prime$ has no direct ${\rm SO}_{2k+1}(\C)$-factors. Let $\widehat G_0=\C^\ast\times_{\oZ_2^N}\prod_{i=1}^N {\rm Spin}_{2k_i+1}(\cx)$ (cf.\ Remark \ref{hatG}).
 Denote by $T_0,\widehat T_0,T^\prime$ (resp.\ $\h_0,\hat\h_0,\h^\prime$)  the maximal tori of $G_0,\widehat G_0$ and $G^\prime$  (resp.\ the corresponding  Cartan subalgebras), and by $W_0,W^\prime, W=W_0\times W^\prime$ the Weyl groups of $G_0,G^\prime,G$. We have the following special case of the  exact sequence \eqref{C*-ext}:
 $$ 1\longrightarrow \cx^\ast\longrightarrow \widehat T_0\times T^\prime\longrightarrow T_0\times T^\prime \longrightarrow 1,
 $$
 and we ask whether, given a $W$-invariant hypertoric variety $X$ with structure torus $T_0\times T^\prime$,
 there exists a $W$-invariant hypertoric variety $\hat X$ with structure torus $\widehat T_0\times T^\prime$ such that the symplectic quotient of $\hat X$ by $\C^\ast$ is isomorphic to $X$. It follows from the results in \S\ref{htv} that this is equivalent to lifting integral multiarrangements.
 \par
 We identify the cocharacter lattice $\sX_\ast(T_0)$ with $\bigoplus_{i=1}^N\oZ^{k_i}$, with the $i$-th factor $\oZ_2^{k_i}\rtimes \Sigma_{k_i}$  of the Weyl group $W_0=\prod_{i=1}^N\oZ_2^{k_i}\rtimes \Sigma_{k_i}$ acting by permutations  and sign changes of coordinates of  the $i$-th summand of $\sX_\ast(T_0)$. The lattice $\sX_\ast(\widehat T_0)$ is then the sublattice of $\oZ\oplus\oZ^k$  such that the sum of all coordinates is even. 
  Let $(\sA,m)$ be the integral multiarrangement of the hypertoric variety $X$. We can decompose $(\sA,m)$ into the sum $(\sA^{\rm odd},m)\oplus (\sA^{\rm even},m)$ of two $W$-invariant multiarrangements, where a hyperplane $H$ belongs to $\sA^{\rm odd}$ (resp.\ $\sA^{\rm even}$) if and only if the sum $\sum_{i=1}^k \omega_i$ of $\sX_\ast(T_0)$-coordinates of its primitive normal $(\omega,\tau)\in \sX_\ast(T_0)\oplus\sX_\ast(T^\prime)$ is odd (resp.\ even). We shall show that $X$ can be lifted to $\hat X$ if
   $\sA^{\rm odd}$ has  unobstructed $W$-invariant deformations (Definition \ref{uW-id}). This means, according to   \S\ref{sww}, that every self-dual hyperplane $H\in \sA^{\rm odd}$ (i.e.\ such that there exists $ w\in W$ which sends a normal of $H$ to its negative) passing through the origin has even multiplicity.
\par
Let $\delta$ be the natural map $\sX_\ast(T_0)\to \sX_\ast(\widehat T_0)$ given by $\delta(\alpha)=(0,\alpha)$ (resp.\ $\delta(\alpha)=(1,\alpha)$) if the sum of coordinates of $\alpha$ is even (resp.\ odd).\footnote{Observe that the lift of a representation of ${\rm Sp}_{2k}(\cx)$ to a  representation of ${\rm Sp}_{2k}^c(\cx)$,  described in the proof of Lemma \ref{liftsp}, replaces the weights $\alpha\in \sX_\ast(T_0)$ with $\delta(\alpha)\in \sX_\ast(\widehat T_0)$.}
 We extend $\delta$ to $\sX_\ast(\widehat T_0)\oplus\sX_\ast(T^\prime)$ by the identity on the second summand. We have the following decomposition of the arrangement $(\sA,m)$:
$$ (\sA,m)=(\sA^\prime,m)\oplus(\sA^{\dprime},m)\oplus (\sA^{\rm even}, m),$$
where $(\sA^{\dprime},m)$ is the self-dual part of $\sA^{\rm odd}$ and $(\sA^\prime,m)$ is its complement. Recall (Lemma \ref{self-dual}) that $(\sA^\prime,m)$ is orientable, i.e.\ there exists a $W$-equivariant choice of primitive normals to hyperplanes in $\sA^\prime$. By changing the scalars $\lambda_k$ we may assume that every hyperplane in $\sA$ which does not pass through the origin has equation $\langle\omega,z\rangle=\lambda\neq 0$ with primitive $\omega$. Let us now choose,  
for every $H\in \sA$, its primitive normal $\omega(H)$, with the additional requirement that the choice is $W$-equivariant on $\sA^\prime$. We define an integral multiarrangement $(\widehat\sA,\hat m )$ in $\sX_\ast(\widehat T_0)\oplus\sX_\ast(T^\prime)$ as follows:
\begin{itemize}
\item a hyperplane $H\in \sA^{\dprime}$ passing through the origin contributes two  hyperplanes to $\widehat\sA$, also passing through the origin, and  having normals $\delta(\pm \omega(H))$ and multiplicities $m(H)/2$;
\item any other hyperplane $H\in \sA$ with equation $\langle\omega(H),z\rangle=\lambda $ contributes one hyperplane to $\widehat\sA$, with equation $\langle\delta(\omega(H)),z\rangle=\lambda $ and multiplicity $m(H)$.
\end{itemize}
It is clear that this multiarrangement is $W$-invariant. Moreover, the parts $(\widehat \sA^\prime,\widehat m)$ and $(\widehat\sA^{\dprime},\widehat m)$, corresponding to $(\sA^\prime,m)$ and $(\sA^{\dprime},m)$, are orientable, so that the obstruction $\o$ vanishes on these two parts. On the other hand,   the image of $ \sA^{\rm even}$ lies in $0\oplus\h_0^\ast\oplus{\h^\prime}^\ast$, and hence its obstruction $\o$ is unchanged, i.e.\ trivial\footnote{Since $\o(\sA,m)=0=\o(\sA^{\rm odd},m)$ by assumption,  we also have $\o(\sA^{\rm even},m)=0$.}. Therefore the hypertoric variety $\hat X$ defined by $(\widehat\sA,\widehat m)$ is $W$-invariant.
The symplectic quotient of $\hat X$ by the centre $\cx^\ast$ of $\widehat G_0$ is isomorphic to $X$, owing to the argument in the proof of Proposition \ref{XtildeX}.
\par
As mentioned in \S\ref{genquat}, this construction yields an alternative definition of Coulomb branches for the group $G$ and those quaternionic representation $\oV$ of $G^\vee$ for which $\hat X$ exists: as the symplectic quotient of $\Hilb^W_\mu(\hat X)$ by $\cx^\ast$. The representation $\oV$ must satisfy two conditions:  the hypertoric variety $X( G,\frac{1}{2}\oV)$ must be $W$-invariant {\em and} the multiplicity of every self-dual hyperplane in $\sA^{\rm odd}$ must be even. For example, if the first condition is satisfied, then the second one is  also satisfied  if $G={\rm SO}_3(\C)$, but the second condition fails for $G={\rm SO}_5(\C)$ and the irreducible representation $\oV=\Gamma_{1,1}$ (with highest weight $2L_1+L_2$) of ${\rm Sp}_4(\cx)$. 
\begin{remark} If $(\sA,m)$ is the multiarrangement of $X( G,\frac{1}{2}\oV)$, then the hyperplanes in $\sA^{\rm odd}$, which coincide with reflection hyperplanes, have even multiplicity. This follows from the fact that the reflection $s$ fixing such a hyperplane $H$  must act on the local model of $X( G,\frac{1}{2}\oV)$ along $H$ and this local model is the product of $T^\ast T^\prime$ (with trivial $\Z_2$-action) and either an $A_k$-singularity or  a $4$-dimensional hypertoric variety with structure torus equal to the maximal torus of $GL_2(\cx)$ and the multirarrangement consisting of the unique reflection hyperplane. In the first case, the $\oZ_2$-invariance implies that $k$ is odd (and hence the multiplicity is even), while in the second case the multiplicity must be even, owing to Lemma \ref{evenGL2}.
\label{evenroots}
\end{remark}
\begin{remark}
In the case $\oV=V\oplus V^\ast$ the arrangement $(\widehat\sA,\widehat m)$ may differ from the one obtained from the construction in \S\ref{Coul} (Theorem \ref{C=W} and Remark \ref{hatG}), and therefore $\hat X$ may differ from $X(\widehat G_0\times G^\prime,V)$. Nevertheless, the symplectic quotients of $\Hilb^W_\mu(\hat X)$  and $\Hilb^W_\mu(X(\widehat G_0\times G^\prime,V))$ by $\C^\ast$ are isomorphic, owing to Proposition \ref{XtildeX}. \end{remark}

\end{document}